\documentclass[twoside,a4paper,reqno,11pt]{amsart} 
\usepackage[top=30mm,right=30mm,bottom=30mm,left=30mm]{geometry}

\usepackage{amsfonts, amsmath, amssymb, mathrsfs, bm, latexsym, stmaryrd, array, hyperref, mathtools, bold-extra, xcolor}

\usepackage[T1]{fontenc}
\usepackage[utf8]{inputenc}

\DeclarePairedDelimiter{\ceil}{\lceil}{\rceil}

\renewcommand{\a}{\alpha}
\renewcommand{\b}{\beta}
\newcommand{\g}{\gamma}
\newcommand{\e}{\varepsilon}
\renewcommand{\l}{\lambda} 

\newcommand{\s}{\sigma}

\renewcommand{\O}{\Omega}

\newcommand{\normeq}{\trianglelefteqslant}

\newcommand{\la}{\langle}
\newcommand{\ra}{\rangle}

\renewcommand{\to}{\rightarrow}

\newcommand{\leqs}{\leqslant}
\newcommand{\geqs}{\geqslant}

\newcommand{\vs}{\vspace{2mm}}

\newcommand{\what}{\widehat} 
\newcommand{\Aut}{\mathrm{Aut}}

\newcommand{\Out}{\mathrm{Out}}

\newcommand{\GL}{\mathrm{GL}}
\newcommand{\PGL}{\mathrm{PGL}}
\newcommand{\SL}{\mathrm{SL}}

\newcommand{\PGammaL}{\mathrm{P\Gamma L}}
\newcommand{\PSigmaL}{\mathrm{P\Sigma L}}

\newcommand{\PSp}{\mathrm{PSp}}
\newcommand{\Sp}{\mathrm{Sp}}

\newcommand{\POmega}{\mathrm{P\Omega}}
\newcommand{\LL}{\mathrm{L}}
\newcommand{\UU}{\mathrm{U}}

\newcommand{\PGU}{\mathrm{PGU}}

\newcommand{\GU}{\mathrm{GU}}
\newcommand{\reg}{\mathrm{reg}}

\makeatletter
\newcommand{\imod}[1]{\allowbreak\mkern4mu({\operator@font mod}\,\,#1)}
\makeatother

\newtheorem{theorem}{Theorem} 
\newtheorem*{conj*}{Conjecture}

\newtheorem{conj}[theorem]{Conjecture} 
\newtheorem{corol}[theorem]{Corollary}

\newtheorem{thm}{Theorem}[section] 
\newtheorem{prop}[thm]{Proposition} 
\newtheorem{lem}[thm]{Lemma}
\newtheorem{cor}[thm]{Corollary} 
\newtheorem{con}[thm]{Conjecture}

\newtheorem{hyp}[thm]{Hypothesis}

\theoremstyle{definition}
\newtheorem{rem}[thm]{Remark}
\newtheorem{remk}{Remark}

\newtheorem{defn}[thm]{Definition}
\newtheorem{ex}[thm]{Example}

\begin{document}

\author{Timothy C. Burness}
\address{T.C. Burness, School of Mathematics, University of Bristol, Bristol BS8 1UG, UK}
\email{t.burness@bristol.ac.uk}

\author{Hong Yi Huang}
\address{H.Y. Huang, School of Mathematics, University of Bristol, Bristol BS8 1UG, UK}
\email{hy.huang@bristol.ac.uk}

\title[On base sizes for primitive groups of product type]{On base sizes for primitive groups of product type} 

\begin{abstract}
Let $G \leqslant {\rm Sym}(\Omega)$ be a finite permutation group and recall that the base size of $G$ is the minimal size of a subset of $\Omega$ with trivial pointwise stabiliser. There is an extensive literature on base sizes for primitive groups, but there are very few results for primitive groups of product type. In this paper, we initiate  a systematic study of bases in this setting. Our first main result determines the base size of every product type primitive group of the form $L \wr P \leqslant {\rm Sym}(\Omega)$ with soluble point stabilisers, where $\Omega = \Gamma^k$, $L \leqslant {\rm Sym}(\Gamma)$ and $P \leqslant S_k$ is transitive. This extends recent work of Burness on almost simple primitive groups. We also obtain an expression for the number of regular suborbits of any product type group of the form $L \wr P$ and we classify the groups with a unique regular suborbit under the assumption that $P$ is primitive, which involves extending earlier results due to Seress and Dolfi. We present applications on the Saxl graphs of base-two product type groups and we conclude by establishing several new results on base sizes for general product type primitive groups.
\end{abstract}

\makeatletter
\@namedef{subjclassname@2020}{\textup{2020} Mathematics Subject Classification}
\makeatother


\date{\today}

\maketitle

\section{Introduction}\label{s:intro}

Let $G \leqs \mathrm{Sym}(\Omega)$ be a finite permutation group and recall that a subset of $\Omega$ is a \emph{base} for $G$ if its pointwise stabiliser in $G$ is trivial. We write $b(G)$ for the \emph{base size} of $G$, which is the minimal size of a base for $G$. This classical invariant has been studied intensively for many decades, finding a wide range of applications and connections to a diverse range of problems in group theory and related areas (we refer the reader to the survey articles \cite{BC,LSh3} and \cite[Section 5]{Bur181} for more details). Determining the base size of a given group is a fundamental problem in permutation group theory.

There is an extensive literature on bases for primitive groups, stretching all the way back to the nineteenth century. In recent years, further motivation for studying bases in this setting stems from a circle of highly influential conjectures of Babai, Cameron, Kantor and Pyber proposed in the 1990s, all of which have now been resolved. For example, a conjecture of Pyber \cite{Pyber} asserts that there is an absolute constant $c$ such that $b(G)$ is at most $c\log_n|G|$ for every primitive permutation group of degree $n$ (note that $b(G) \geqs \log_n|G|$ is a trivial lower bound). The proof of Pyber's conjecture, which was finally completed by Duyan, Halasi and Mar\'{o}ti in \cite{DHM}, relies heavily on the 
O'Nan-Scott theorem, which describes the finite primitive groups in terms of the structure and action of the socle of the group. Following \cite{LPS}, this theorem divides the primitive groups into five families, as briefly described in Table \ref{tab:prim} (in the table, $T$ denotes a non-abelian finite simple group). 

\begin{table}[h]
\[
\begin{array}{ll} \hline
\mbox{Type} & \mbox{Description} \\ \hline
\mbox{I} & \mbox{Affine: $G = V{:}H \leqs {\rm AGL}(V)$, $H \leqs{\rm GL}(V)$ irreducible} \\
\mbox{II} & \mbox{Almost simple: $T \leqs G \leqs {\rm Aut}(T)$} \\
\mbox{III} & \mbox{Diagonal type: $T^k \leqs G \leqs T^k.({\rm Out}(T) \times P)$, $P \leqs S_k$ primitive, or $k=2$, $P=1$} \\
\mbox{IV} & \mbox{Product type: $G \leqs L \wr P$, $L$ primitive of type II or III, $P \leqs S_k$ transitive} \\
\mbox{V} & \mbox{Twisted wreath product} \\ \hline
\end{array}
\]
\caption{The types of finite primitive groups in the O'Nan-Scott theorem}
\label{tab:prim}
\end{table}

In this paper, we initiate a systematic study of bases for product type primitive groups, noting that there are very few existing results in the literature when compared to the other families of primitive groups identified by the O'Nan-Scott theorem. Recall that if $G \leqs {\rm Sym}(\O)$ is such a group, then $G \leqs L \wr S_k$ acts on $\O = \Gamma^k$ with its product action, where $k \geqs 2$ and $L \leqs {\rm Sym}(\Gamma)$ is a primitive group with socle $T$, which is either almost simple or diagonal type (see Table \ref{tab:prim}). Moreover, $G$ has socle $T^k$ and the subgroup $P \leqs S_k$ induced by the conjugation action of $G$ on the set of factors of $T^k$ is transitive. In particular, we have $T^k \normeq G \leqs L \wr P$ and thus $b(G) \leqs b(L \wr P)$.

In \cite{HLM}, Halasi et al. prove an explicit version of Pyber's conjecture for finite primitive groups, establishing the general bound
\[
b(G) \leqs 2\frac{\log |G|}{\log |\Omega|}+22
\]
when $G$ is a product type group (the original form of Pyber's conjecture in this setting, with unspecified constants, was established in \cite{BS}). Bases for product type groups of the form $G = L \wr P$ are also considered by Bailey and Cameron in \cite{BC}. In order to state their main result (which imposes no conditions on $L$ or $P$), let $D(P)$ denote the \emph{distinguishing number} of $P$, which is the minimal number of parts in a distinguishing partition for the action of $P$ on $\{1, \ldots, k\}$ (recall that a partition is a distinguishing partition if the intersection of the setwise stabilisers of the parts in $P$ is trivial). In addition, for a positive integer $m$,  let ${\rm reg}(L,m)$ be the number of regular orbits of $L$ in its coordinatewise action on $\Gamma^m$ (note that ${\rm reg}(L,m) \geqs 1$ if and only if $m \geqs b(L)$). Then \cite[Theorem 2.13]{BC} states that
\begin{equation}\label{e:bc}
\mbox{\emph{$b(L \wr P) \leqs m$ if and only if ${\rm reg}(L,m) \geqs D(P)$.}}
\end{equation}

In studying product type groups $G \leqs L \wr P$ as above, there is a natural distinction to make between the full wreath product $L\wr P$ and its proper (primitive) subgroups. As in \cite{BC}, we mainly focus on groups of the form $G = L \wr P$ in this paper,  although we will study the general case in Section \ref{s:gen}. 

Primitive groups with soluble point stabilisers is another common theme throughout the paper. This is partly motivated by the main theorem of \cite{B_sol}, which states that $b(G) \leqs 5$ for every finite primitive permutation group $G$ with soluble point stabilisers (this extends a well known theorem of Seress \cite{S_sol}, which establishes the bound $b(G) \leqs 4$ in the more restrictive setting where $G$ itself is soluble). Moreover, \cite[Theorem 2]{B_sol} determines the exact base size of every almost simple primitive group with soluble stabilisers and it is natural to seek an extension of this result to product type groups (as a consequence of the O'Nan-Scott theorem, every primitive group with soluble stabilisers is either almost simple, affine, or product type).

In order to describe some of our main results, it will be convenient to write $\mathcal{S}$ for the set of finite almost simple primitive groups with soluble point stabilisers (as noted above, the precise base size of each group in $\mathcal{S}$ is computed in \cite{B_sol}). Let $G \leqs L \wr P$ be a product type primitive group as above with soluble point stabiliser $H$, so $L \in \mathcal{S}$ and $G$ has socle $T^k$, where $T$ is the socle of $L$ (see Remark \ref{r:sol}). Since $G = T^kH$, it follows that $H$ induces $P$ on the set of factors of $T^k$ and thus $P$ is soluble. It is easy to see that $b(L \wr P) \geqs b(L)$ and thus \eqref{e:bc} implies that $b(L \wr P) = b(L)$ if and only if ${\rm reg}(L,b(L)) \geqs D(P)$. By \cite[Theorem 1.2]{S_sol}, the solubility of $P$ implies that $D(P) \leqs 5$ and as a consequence we deduce that $b(L \wr P) = b(L)$ if ${\rm reg}(L,b(L)) \geqs 5$. This leads us naturally to the problem of determining the groups $L \in \mathcal{S}$ with ${\rm reg}(L,b(L)) \leqs 4$, which is addressed in Theorem \ref{thm:reg(L,b(L))} below (note that the relevant tables are presented at the end of the paper in Section \ref{s:tab}).

\begin{theorem}\label{thm:reg(L,b(L))}
Let $L$ be a finite almost simple primitive group with soluble point stabiliser $J$. Then $\reg(L,b(L))\leqs 4$ if and only if $(L,J)$ is one of the cases in Table \ref{tab:r(L)} or \ref{tab:reg(L)}. 
\end{theorem}

\begin{remk}
Suppose $L$ is a group as in Theorem \ref{thm:reg(L,b(L))} with $b(L)=2$, so  ${\rm reg}(L,b(L))  = r(L)$ is the number of regular suborbits of $L$. By the theorem, it follows  that $r(L)\leqs 4$ if and only if $(L,J)$ is one of the cases in Table \ref{tab:r(L)}. This extends \cite[Theorem 4]{BH_Saxl}, which determines the groups $L$ of this form with $r(L)=1$. 
\end{remk}

The proof of Theorem \ref{thm:reg(L,b(L))} will be given in Section \ref{s:sol} and it features a combination of probabilistic and computational methods. In particular, the argument relies on the statement and proof of \cite[Theorem 2]{B_sol}, with \cite[Theorem 3.1]{BH_Saxl} as another key ingredient. The relevant computations will be described in Section \ref{ss:comp}.

An ambitious project initiated by Jan Saxl in the 1990s seeks to determine the primitive groups $G$ with $b(G)=2$ and there has been some significant progress towards this goal in several special cases. For example, we refer the reader to Fawcett's work \cite{Faw1,Faw2} on diagonal type groups and twisted wreath products, while various papers have investigated this problem for affine groups (for instance, see Lee's recent work \cite{Lee1,Lee2} on affine groups with quasisimple point stabilisers). The almost simple primitive groups have also been intensively studied and the base-two groups with socle an alternating group or sporadic group \cite{BGS,BOW} have been determined. In addition, a complete solution for almost simple groups with soluble point stabilisers is presented in \cite[Theorem 2]{B_sol}. 

Our next result constitutes the first step towards a classification of the base-two product type primitive groups.  

\begin{theorem}\label{thm:wr_base-two}
Let $G = L \wr P$ be a product type primitive group with soluble point stabiliser $J \wr P$. Then $b(G)=2$ if and only if $b(L)=2$ and either 
\begin{itemize}\addtolength{\itemsep}{0.2\baselineskip}
\item[{\rm (i)}] $(L,J)$ is not one of the cases in Table \ref{tab:r(L)}; or 
\item[{\rm (ii)}] $D(P) \leqs r(L) \leqs 4$ and $(L,J,r(L))$ is one of the cases in Table \ref{tab:r(L)}.
\end{itemize}
\end{theorem}

More generally, we have the following.

\begin{theorem}\label{thm:wr_base}
Let $G = L \wr P$ be a product type primitive group with soluble point stabiliser $J \wr P$. Then either 
\begin{itemize}\addtolength{\itemsep}{0.2\baselineskip}
\item[{\rm (i)}] $b(G) = b(L)$; or  
\item[{\rm (ii)}] $b(G) = b(L)+1$, $\reg(L,b(L))<D(P)$ and $(L,J,{\rm reg}(L,b(L)))$ is one of the cases in Table \ref{tab:r(L)} or \ref{tab:reg(L)}. 
\end{itemize}
\end{theorem}

Next let $r(G)$ be the number of regular suborbits of an arbitrary  product type primitive group of the form $G = L \wr P$, noting that $b(G)=2$ if and only if $r(G) \geqs 1$. In the following result, we determine a formula for $r(G)$, which turns out to have some interesting applications. Here $t_m$ denotes the number of (unordered) distinguishing partitions with $m$ non-empty parts for the action of $P$ on $\{1,\ldots,k\}$, so $t_m \geqs 1$ if and only if $D(P)\leqs m$.

\begin{theorem}\label{thm:r(W)_cal}
Let $G = L \wr P$ be a product type primitive group on $\O = \Gamma^k$. Then
\[
r(G) =\frac{1}{|P|}\sum_{m=D(P)}^km!{r(L) \choose m}t_m.
\]	
\end{theorem}

In general, it is difficult to compute $t_m$ precisely, but this can be achieved in some special cases, which then allows us to simplify the given expression for $r(G)$. For example, see Corollaries \ref{cor:val_P=Sk} and \ref{cor:val_P=Cp}. Notice that we have a trivial upper bound $t_m \leqs S(k,m)$, which is the total number of partitions of $\{1,\ldots,k\}$ into $m$ non-empty parts (a Stirling number of the second kind). In fact, this bound is best possible. For instance, if $P = C_k$ has prime order, then $t_m = S(k,m)$ for all $m\geqs 2$. On the other hand, if $\Pi = \{\pi_1,\dots,\pi_m\}$ is a distinguishing partition for $P$, then $P_{\{\Pi\}}\leqs S_m$ is a permutation group on the set of parts comprising $\Pi$ and thus $t_m\geqs |P|/m!$. Note that $t_m= |P|/m!$ only if each part has the same size in every distinguishing partition for $P$ into $m$ parts. For example, if $P=S_k$ and $m=k$, then $t_m = 1 = |P|/m!$ since the partition of $\{1, \ldots, k\}$ into singletons is the only distinguishing partition for $P$.

As an immediate corollary of Theorem \ref{thm:r(W)_cal}, we obtain the following result on product type groups with a unique regular suborbit.

\begin{corol}
	\label{c:unique}
	Let $G = L \wr P$ be a product type primitive group. Then $r(G) = 1$ if and only if $r(L) = D(P)$ and $t_{D(P)} = |P|/D(P)!$.
\end{corol}

In \cite{S_dist}, Seress classifies the primitive groups $P \leqs S_k$ with $D(P)=2$, which is equivalent to the property that $P$ has a regular orbit on the power set of $\{1, \ldots, k\}$. Using a similar approach, we will show that $(k,P) = (2,S_2)$ is the only example with a unique regular orbit on the power set (see Corollary \ref{c:unique_power}), which we then use to determine the primitive groups $P$ with $t_{D(P)} = |P|/D(P)!$. This allows us to establish the following refinement of Corollary \ref{c:unique} in the case where $P$ is primitive.

\begin{theorem}
	\label{t:new}
	Let $G = L\wr P$ be a product type primitive group with $P \leqs S_k$ primitive. Then $r(G) = 1$ if and only if $r(L) = D(P)$ and one of the following holds:
	\begin{itemize}\addtolength{\itemsep}{0.2\baselineskip}
		\item[\rm (i)] $P = S_k$ and $D(P) = k$.
		\item[\rm (ii)] $P = A_5$, $k = 6$ and $D(P) = 3$.
		\item[\rm (iii)] $P = \PGammaL_2(8)$, $k = 9$ and $D(P) = 3$.
		\item[\rm (iv)] $P = \mathrm{AGL}_3(2)$, $k = 8$ and $D(P) = 4$.
	\end{itemize}
\end{theorem}

In a different direction, we can use Theorem \ref{thm:r(W)_cal} to establish several new results on the Saxl graphs of base-two product type primitive groups. Let $G \leqs {\rm Sym}(\Omega)$ be a permutation group with $b(G)=2$. The \emph{Saxl graph} of $G$, denoted $\Sigma(G)$, was introduced by Burness and Giudici in \cite{BG_Saxl}; the vertices of $\Sigma(G)$ are labelled by the elements of $\Omega$, with two vertices joined by an edge if and only if they form a base for $G$. It is easy to see that if $G$ is primitive, then $\Sigma(G)$ is vertex-transitive and connected. In this case, we write ${\rm val}(G)$ to denote the valency of $\Sigma(G)$, noting that ${\rm val}(G) = |H|r(G)$ with $H$ a point stabiliser. 

Our main result on the valency of Saxl graphs of product type groups  is Corollary \ref{cor:Saxl} below, which extends earlier work in \cite{BG_Saxl} and \cite{CH_val}. For part (i), recall that a connected graph is \emph{Eulerian} if and only if every vertex has even degree.

\begin{corol}
	\label{cor:Saxl}
	Let $G = L\wr P$ be a base-two product type primitive group with Saxl graph $\Sigma(G)$. Then the following hold:
\begin{itemize}\addtolength{\itemsep}{0.2\baselineskip}
\item[{\rm (i)}] $\Sigma(G)$ is Eulerian.
\item[{\rm (ii)}] ${\rm val}(G)$ is a prime power if and only if $L = {\rm M}_{10}$, $J = SD_{16}$, $P$ is a $2$-group and $t_2 \geqs 1$ is a $2$-power.
\end{itemize}
\end{corol}

Note that ${\rm val}(G)$ is a prime power only if $|H| =|J|^k|P|$ is a prime power, which implies that both $J$ and $P$ are soluble. In particular, ${\rm val}(G)$ is a prime power only if $L \in \mathcal{S}$. Let us also observe that there are genuine examples in part (ii), where $P$ is a $2$-group and $t_2$ is a $2$-power. For example, if $P = C_8{:}(C_2 \times C_2)$ is the holomorph of the cyclic group $C_8$ in its natural action on $8$ points, then $t_2 = 16$. We refer the reader to Remark \ref{r:t2} for further comments. 

One of the main open problems concerning the Saxl graphs of base-two primitive groups is a conjecture of Burness and Giudici \cite[Conjecture 4.5]{BG_Saxl}. In order to state this, let $G \leqs {\rm Sym}(\O)$ be a finite primitive  group with $b(G)=2$. For $\a \in \Omega$, let
\[
\Sigma(\alpha) = \{ \gamma \in \Omega \,:\, \mbox{$\{\a,\gamma\}$ is a base for $G$} \}
\]
be the set of neighbours of $\a$ in $\Sigma(G)$ (note that $\Sigma(\alpha)$ is the union of the regular $G_\alpha$-orbits on $\Omega$). Then \cite[Conjecture 4.5]{BG_Saxl} asserts that 
\[
\mbox{\emph{$\Sigma(\alpha) \cap \Sigma(\beta)$ is non-empty for all $\a,\b \in \Omega$.}}
\]
That is, any two vertices in $\Sigma(G)$ have a common neighbour and thus the diameter of $\Sigma(G)$ is at most $2$. Some initial  evidence for the veracity of this conjecture is presented in Sections 4--6 in \cite{BG_Saxl}. In addition, it has been verified for all groups with socle ${\rm L}_2(q)$ and for all almost simple groups with soluble point stabilisers (see \cite{ChenDu2020Saxl} and \cite{BH_Saxl}, respectively). Some positive results for affine groups are given in \cite{LP}. 

By considering this conjecture for primitive wreath products, we deduce that it is equivalent to the following (a priori, stronger) statement.

\begin{conj}\label{conjecture:SBG}
Let $G \leqs {\rm Sym}(\Omega)$ be a finite primitive permutation group with $b(G) = 2$ and let $\a,\b \in \Omega$. Then $\Sigma(\alpha)$ meets every regular $G_\beta$-orbit on $\Omega$.
\end{conj}

In particular, this asserts that $|\Sigma(\alpha) \cap \Sigma(\b)| \geqs r(G)$ for all $\a,\b \in \O$. We will present some evidence for Conjecture \ref{conjecture:SBG} in Section \ref{ss:conn}. For example, we will show that the conclusion holds for the action of ${\rm L}_2(q)$ on the set of pairs of $1$-dimensional subspaces of the natural module $\mathbb{F}_q^2$ (see Proposition \ref{p:SBG_PSL}). 

In the final part of the paper, we consider general product type groups of the form $G \leqs L \wr P$. The analysis of bases in this setting is significantly more difficult and there are very few (if any) existing results in the literature that are tailored to this particular situation. As a starting point, we seek a partial extension of Theorem \ref{thm:wr_base-two} by determining the base-two groups in certain families of product type primitive groups of the form 
\begin{equation}\label{e:inc}
T \wr P < G < L \wr P
\end{equation}
with soluble point stabilisers. Here $P$ is soluble and $L$ is almost simple with soluble point stabilisers. In addition, we may assume $G$ induces $L$ on each copy of $\Gamma$ in the Cartesian product $\O = \Gamma^k$, so both $L$ and $P$ are uniquely determined by $G$ (see Remarks \ref{r:kov} and \ref{r:sol}). 

Our main results are Theorems \ref{thm:P=C_2_sol_gen} and \ref{thm:semiprimitive_sol_stab}. For instance, the following result in the special case  $k=b(L)=2$ is stated as Theorem \ref{thm:P=C_2_sol_gen} in Section \ref{s:gen}.

\begin{theorem}\label{t:new2}
Let $G$ be a product type primitive group as in \eqref{e:inc}, where $k=b(L)=2$ and $G$ has soluble point stabilisers. Then $b(G) \leqs 3$, with equality if and only if $|L \wr P :G|=2$ and one of the following holds, where $J$ is a point stabiliser in $L$:
\begin{itemize}\addtolength{\itemsep}{0.2\baselineskip}
\item[\rm (i)] $(L,J) = (\mathrm{M}_{10},5{:}4)$.
\item[\rm (ii)] $(L,J) = (\mathrm{J}_2.2,5^2{:}(4\times S_3))$.
\item[\rm (iii)] $L = {\rm PGU}_4(3)$ and $J$ is of type $\GU_1(3)\wr S_4$.
\end{itemize}
\end{theorem}

Further observations on the base-two problem for general product type primitive groups are presented at the end of Section \ref{s:gen}.

\subsection*{Notation} Let $G$ be a finite group and let $n$ be a positive integer. We will write $C_n$, or just $n$, for a cyclic group of order $n$ and $G^n$ will denote the direct product of $n$ copies of $G$. An unspecified extension of $G$ by a group $H$ will be denoted by $G.H$; if the extension splits then we may write $G{:}H$. We use $[n]$ for an unspecified soluble group of order $n$ (in addition, we will sometimes write $[n]$ to denote the set $\{1, \ldots, n\}$, but this should not cause any confusion). We adopt the standard notation for simple groups of Lie type from \cite{KL}. All logarithms in this paper are in base $2$.

\subsection*{Organisation} We begin in Section \ref{s:prel} by recording some preliminary results that will be needed in the proofs of our main results, together with some details of the computational methods that we will use in this paper. This includes work of Seress \cite{S_dist, S_sol} on the distinguishing number of permutation groups, which we will combine with a key result of Bailey and Cameron \cite{BC} on bases for product type groups (see Theorem \ref{thm:BC}). Next in Section \ref{s:sol} we focus on the product type groups of the form $G = L \wr P$ with soluble point stabilisers and we will prove Theorems \ref{thm:reg(L,b(L))}, \ref{thm:wr_base-two} and \ref{thm:wr_base}. Our main results on regular suborbits are presented in Section \ref{s:reg} and we investigate applications involving the Saxl graphs of product type groups in Section \ref{s:Saxl}. Finally, in Section \ref{s:gen} we consider general product type groups with $G \leqs L \wr P$, focussing on the case where $G$ contains $P$. The relevant tables referred to in the statement of Theorem  \ref{thm:reg(L,b(L))} are presented in Section \ref{s:tab}.

\subsection*{Acknowledgements} The second author is supported by the China Scholarship Council for his doctoral studies at the University of Bristol.

\section{Preliminaries}\label{s:prel}

In this section we record some preliminary results, which will be needed in the proofs of our main theorems. Throughout this section, $G\leqs\mathrm{Sym}(\Omega)$ denotes a finite transitive permutation group of degree $n$ with point stabiliser $H$. Let $b(G)$ be the base size of $G$.

\subsection{Distinguishing number}\label{ss:distinguishing}

Let $\Pi = \{\pi_1,\dots,\pi_m\}$ be a partition of $\O$ into $m$ non-empty parts. Then $\Pi$ is called a \textit{distinguishing partition} for $G$ if its stabiliser $\bigcap_{i = 1}^mG_{\{\pi_i\}}$ is trivial, where $G_{\{\pi_i\}}$ is the setwise stabiliser of $\pi_i$ in $G$. The \textit{distinguishing number} of $G$, denoted $D(G)$, is defined to be the smallest number $m$ such that there exists a distinguishing partition for $G$ with $m$ parts. For example, $D(G)=1$ if and only if $G$ is trivial, whereas $D(G) \leqs 2$ if and only if $G$ has a regular orbit on the power set of $\Omega$. At the other end of the spectrum, $S_n$ and $A_n$ have distinguishing numbers $n$ and $n-1$, respectively. 

We will need the following theorem due to Seress \cite[Theorem 1.2]{S_sol}, which gives a best possible upper bound on $D(G)$ when $G$ is soluble. 

\begin{thm}\label{t:sol_dist}
If $G$ is soluble, then $D(G)\leqs 5$.
\end{thm}

\begin{rem}\label{r:ser}
It is worth noting that for each $d \in \{2,3,4,5\}$, there are infinitely many soluble transitive permutation groups $G$ with $D(G)=d$. For example, if $t \in \{2,3,4\}$ and $m \geqs 2$, then $D(G) = t+1$ for the natural action of $G = S_t \wr C_m$ of degree $tm$ (see \cite[p.244]{S_sol}). And by Theorem \ref{t:prim} below, there are infinitely many soluble primitive groups $G$ with $D(G)=2$. 
\end{rem}

The following result on the distinguishing number of primitive groups, which is also due to Seress \cite[Theorem 2]{S_dist}, will be useful later.

\begin{thm}\label{t:prim}
Let $G$ be a primitive group of degree $n$. Then either $D(G)=2$ or one of the following holds:
\begin{itemize}\addtolength{\itemsep}{0.2\baselineskip}
		\item[{\rm (i)}] $G = S_n$ or $A_n$.
		\item[{\rm (ii)}] $G$ is one of $43$ groups listed in \cite[Theorem 2]{S_dist}, each with $n \leqs 32$ and $D(G) \in \{3,4\}$. 
	\end{itemize}
\end{thm}

The precise distinguishing numbers of the groups arising in case (ii) of Theorem \ref{t:prim} were determined by Dolfi (see \cite[Lemma 1]{D_dist}).

\subsection{Product type groups}\label{ss:prod}

Let $L \leqs {\rm Sym}(\Gamma)$ be a finite primitive group with socle $T$ and point stabiliser $J$, which is either almost simple or diagonal type (see Table \ref{tab:prim}). Let $k \geqs 2$ be an integer and consider the product action of $L \wr S_k$ on the Cartesian product $\Omega = \Gamma_1 \times \cdots \times \Gamma_k = \Gamma^k$:
\[
(\gamma_1, \ldots, \gamma_k)^{(z_1,\ldots, z_k)\sigma} = \left(\gamma_{1^{\s^{-1}}}^{z_{1^{\s^{-1}}}}, \ldots, \gamma_{k^{\s^{-1}}}^{z_{k^{\s^{-1}}}}\right),
\]
where $\gamma_i \in \Gamma$, $z_i \in L$ and $\s \in S_k$. Since this action is faithful and primitive, we can view $L \wr S_k$ as a product type primitive  group on $\Omega$, with socle $T^k$ and point stabiliser $J \wr S_k$. More generally, a subgroup $G \leqs L \wr S_k$ is a primitive group of product type if $G$ has socle $T^k$ and the subgroup $P \leqs S_k$ induced by the conjugation action of $G$ on the set of factors of $T^k$ is transitive. Therefore
\begin{equation}\label{e:prod}
T^k \normeq G \leqs L \wr P.
\end{equation}

\begin{rem}\label{r:kov}
Let $G_1 = \{(z_1, \ldots, z_k)\s \in G \,:\, 1^{\s}=1\}$ and let $L_1 \leqs L \leqs {\rm Sym}(\Gamma_1)$ be the group induced by $G_1$ on $\Gamma_1$. Then by a theorem of Kov\'{a}cs \cite[(2.2)]{Kov}, we may replace $G$ by a conjugate $G^x$ for some $x \in \prod_{i=1}^k{\rm Sym}(\Gamma_i) < {\rm Sym}(\Omega)$ so that  $G \leqs L_1 \wr P$ and $G$ induces $L_1$ on each factor $\Gamma_i$ of $\Omega$. Since $L_1 \leqs {\rm Sym}(\Gamma_1)$ is primitive, we are free to assume that $L_1 = L$, so 
\eqref{e:prod} holds and the groups $L$ and $P$ are uniquely determined by $G$. This observation will be relevant when we consider general product type groups with $G < L \wr P$ in Section \ref{s:gen}. 
\end{rem}

\begin{rem}\label{r:sol}
Consider the special case where $G \leqs L \wr P$ is a product type primitive group with soluble point stabiliser $H = G_{\a}$, where $\a = (\gamma, \ldots, \gamma) \in \O$ for some $\gamma \in \Gamma$. By the transitivity of the socle $T^k$ on $\O$, we have $G = T^kH$ and thus $H$ induces $P$ on the set of factors of $T^k$. Therefore, $P$ is soluble. In addition, we have $(T_{\gamma})^k = (T^k)_{\a} \leqs H$, so $T_{\gamma}$ is soluble and thus $L$ is almost simple (indeed, if $L$ is a diagonal type group with socle $T=S^m$ for some non-abelian simple group $S$, then $T_{\gamma} \cong S$ is insoluble). In particular, $L/T \leqs {\rm Out}(T)$ is soluble. Now the primitivity of $L$ implies that $L = TL_{\gamma}$, so $L_{\gamma}/T_{\gamma} \cong L/T$ is soluble and we conclude that $L_{\gamma}$ is soluble. 
\end{rem}

Given a positive integer $m$, let ${\rm reg}(L,m)$ be the number of regular $L$-orbits with respect to the natural coordinatewise action of $L$ on the Cartesian product $\Gamma^m$. Note that $\reg(L,2) = r(L)$ is the number of regular suborbits of $L$ on $\Gamma$.

As explained in the introduction, in this paper we are primarily interested in product type primitive groups of the form $G = L \wr P$ as above. In this setting, the following theorem of Bailey and Cameron (see \cite[Theorem 2.13]{BC}) will be an essential tool. 

\begin{thm}\label{thm:BC}
Let $G = L \wr P$ be a product type primitive group. Then $b(G) \leqs m$ if and only if ${\rm reg}(L,m) \geqs D(P)$. In particular, $b(G) = 2$ if and only if $r(L) \geqs D(P)$.
\end{thm}

\subsection{Computational methods}\label{ss:comp}

To conclude this preliminary section, we briefly describe some of the main computational methods that we will use in this paper, working with {\sc Magma} \cite{Magma} (version V2.26-6). Here we focus on the computations arising in the proof of Theorem \ref{thm:reg(L,b(L))} in Section \ref{s:sol}, so let $L \leqs {\rm Sym}(\Gamma)$ be a finite primitive almost simple permutation  group with socle $T$ and point stabiliser $J$. Note that the primitivity of $L$ implies that $T$ acts transitively on $\Gamma$.

\subsubsection{Regular orbits on $\Gamma^2$}\label{sss:reg1}

First assume $b(L)=2$ and let $r(L)$ (respectively, $r(T)$) be the number of regular suborbits of $L$ (respectively, $T$) on $\Gamma$. As explained in \cite[Section 2.2]{BH_Saxl}, we can compute $r(L)$ by working with double cosets. Indeed, if $R$ is a complete set of $(J,J)$ double coset representatives in $L$, then  
\[
r(L) = |\{ x \in R \,:\, |JxJ| = |J|^2\}|.
\]
Similarly, if $R_0$ is a complete set of $(J_0,J)$ double coset representatives, where $J_0 = J \cap T$, then 
\[
r(T) = |\{ x \in R_0 \,:\, |J_0xJ| = |J_0||J|\}|.
\]

It is straightforward to compute these numbers using {\sc Magma}. We first construct $L$ as a permutation group (not necessarily with respect to the action on $\Gamma$) and we then construct $J$ as a maximal subgroup. To do this, we will typically use the functions 
\[
\mbox{\texttt{AutomorphismGroupSimpleGroup} and \texttt{MaximalSubgroups}.}
\] 
The latter returns a set of representatives of the conjugacy classes of maximal subgroups of $L$, but it is not effective in a handful of cases we need to consider. To get around this difficulty, we can exploit the fact that $J = N_L(K)$ for a suitable $p$-subgroup $K$ of $T$ (for instance, see \cite[Example 2.4]{B_sol}). If $|\Gamma| = |L:J|$ is not prohibitively large (for example, if $|\Gamma|< 10^7$), then we can use the function \texttt{DoubleCosetRepresentatives} to determine $R$ and $R_0$, which then allows us to compute $r(L)$ and $r(T)$. 

\subsubsection{Regular orbits on $\Gamma^3$ and $\Gamma^4$}\label{sss:reg2}

In the proof of Theorem \ref{thm:reg(L,b(L))} we will also need to compute ${\rm reg}(L,b(L))$ in a number of cases with $b(L) =3$ or $4$. Fix $\gamma \in \Gamma$ and set $J = L_{\gamma}$. Decompose $\Gamma = \Lambda_1 \cup \cdots \cup \Lambda_t$ as a disjoint union of $J$-orbits and fix $\l_i  \in \Lambda_i$ and $K_i = J_{\l_i}$. 

Suppose $b(L)=3$. Then every regular $L$-orbit on $\Gamma^3$ is represented by an element of the form $(\gamma, \lambda_i, \beta)$, where $\beta$ is contained in a regular orbit of $K_i$ on $\Gamma$. Therefore, 
\begin{equation}\label{e:reg3}
{\rm reg}(L,3) = \sum_{i=1}^t r_{i},
\end{equation}
where $r_{i}$ is the number of regular orbits of $K_i$ on $\Gamma$. 

Now assume $b(L)=4$. Fix $i \in \{1, \ldots, t\}$. Let $\Lambda_{i,1}, \ldots, \Lambda_{i,m_i}$ be the orbits of $K_i$ on $\Gamma$ and for each $j \in \{1, \ldots, m_i\}$ set $K_{i,j} = (K_i)_{\omega_j} = J_{\l_i} \cap J_{\omega_j}$ for some fixed $\omega_j \in \Lambda_{i,j}$. Notice that every regular $L$-orbit on $\Gamma^4$ is represented by an element of the form $(\gamma, \lambda_i, \omega_j, \beta)$, where $\beta$ is in a regular orbit of $K_{i,j}$ on $\Gamma$. Therefore, if $r_{i,j}$ denotes the number of regular orbits of $K_{i,j}$ on $\Gamma$, then
\begin{equation}\label{e:reg4}
{\rm reg}(L,4) = \sum_{i=1}^{t}\sum_{j=1}^{m_i} r_{i,j}.
\end{equation}

Once again, we can implement this approach in {\sc Magma} in order to compute ${\rm reg}(L,b(L))$. As before, we first construct $L$ and $J$, and then we work with the function \texttt{CosetAction} to construct $L$ as a permutation group on $\Gamma$, which then allows us to construct stabilisers and their orbits. This can be expensive (in terms of time and memory) if $|\Gamma|$ is large, but in the proof of Theorem \ref{thm:reg(L,b(L))}, we only need to apply this method when $|\Gamma| < 5 \times 10^6$ and no special difficulties arise.

\section{Soluble stabilisers}\label{s:sol}

In this section, we assume $G = L\wr P$ is a product type primitive group on $\O = \Gamma^k$, with socle $T^k$ and soluble point stabiliser $J \wr P$. In particular, $P$ is soluble and $L \leqs {\rm Sym}(\Gamma)$ is almost simple with socle $T$ and soluble point stabiliser $J$ (so we have $L \in \mathcal{S}$ in terms of our notation in Section \ref{s:intro}). We will prove Theorems \ref{thm:reg(L,b(L))}, \ref{thm:wr_base-two} and \ref{thm:wr_base}.

We begin with a useful observation in this setting.

\begin{lem}
	\label{l:plus1}
	Let $G = L\wr P$ be a product type primitive group with soluble point stabilisers. Then either
	 \begin{itemize}\addtolength{\itemsep}{0.2\baselineskip}
\item[{\rm (i)}] $\reg(L,b(L))\geqs D(P)$ and $b(G) = b(L)$; or
\item[{\rm (ii)}] $\reg(L,b(L)) < D(P)$ and $b(G) = b(L)+1$.
\end{itemize}
\end{lem}

\begin{proof}
First recall that $b(G)\geqs b(L)$, which shows that (i) follows from Theorem \ref{thm:BC}. Now assume $\reg(L,b(L)) < D(P)$, so $b(G) \geqs b(L)+1$. By \cite[Corollary 2.14]{BC}, we have 
\[
b(G)\leqs b(L)+\ceil{\log_m D(P)},
\]
where $m = |\Gamma|$ is the degree of $L$. Since $P$ is soluble, we have $D(P)\leqs 5$ by Theorem \ref{t:sol_dist}, while $m \geqs 5$ since $L$ is almost simple. Therefore, $\ceil{\log_m D(P)} = 1$ and thus $b(G)\leqs b(L)+1$ as required.
\end{proof}

\begin{rem}
	\label{rem:b(G)-b(L)}
In general, if $G = L \wr P$ does not have soluble point stabilisers, then the difference $b(G)-b(L)$ can be arbitrarily large. For example, if $m$ is a positive integer and $k > \reg(L,m)$, then $b(L\wr S_k)> m$ by Theorem \ref{thm:BC}.
\end{rem}

\subsection{Base-two groups}\label{ss:b2}

Here we will prove Theorem \ref{thm:wr_base-two}, which completely describes the product type primitive groups $G = L \wr P$ as above with $b(G)=2$. As recorded in Theorem \ref{thm:BC}, we have $b(G)=2$ if and only if $r(L) \geqs D(P)$, while Theorem \ref{t:sol_dist} implies that $D(P) \leqs 5$. Therefore, we immediately deduce that $b(G)=2$ if $r(L) \geqs 5$ and so we are interested in determining the groups $L \in \mathcal{S}$ with $1 \leqs r(L) \leqs 4$. 

With this aim in mind, the following result establishes Theorem \ref{thm:reg(L,b(L))} in the special case where $b(L)=2$ (the proof of  
Theorem \ref{thm:reg(L,b(L))} will be completed in Proposition \ref{prop:2} below). Recall that Table \ref{tab:r(L)} is presented in Section \ref{s:tab} at the end of the paper.

\begin{prop}\label{prop:r(L)<=4}
Let $L \leqs {\rm Sym}(\Gamma)$ be an almost simple primitive group with socle $T$ and soluble point stabiliser $J$. Then $1 \leqs r(L)\leqs 4$ if and only if $(L,J)$ is one of the cases in Table \ref{tab:r(L)}.
\end{prop}

\begin{proof}
Following \cite{BH_Saxl}, let $\mathcal{G}$ be the set of base-two 
almost simple primitive groups with soluble point stabilisers and note that the groups in $\mathcal{G}$ are determined in \cite{B_sol}. Let us also define $\mathcal{L}$ to be the set of classical groups in $\mathcal{G}$ with socle ${\rm L}_2(q)$ such that $J$ is of type ${\rm GL}_1(q) \wr S_2$ or ${\rm GL}_{1}(q^2)$ (in other words, $J$ is the normaliser of a maximal torus). Given $L \in \mathcal{G}$ as in the statement of the proposition, let 
\begin{equation}\label{e:Q}
\mathcal{P}(L) = \frac{|\{(\a,\b) \in \Gamma^2 \,:\, L_{\a} \cap L_{\b} = 1\}|}{|\Gamma|^2} = \frac{|J|^2r(L)}{|L|}
\end{equation}
be the probability that a random pair of points in $\Gamma$ form a base for $L$ (so the condition $b(L) = 2$ implies that $\mathcal{P}(L) > 0$). 

First assume $L \in \mathcal{G} \setminus \mathcal{L}$ and $\mathcal{P}(L) \leqs 3/4$. By \cite[Theorem 3.1]{BH_Saxl}, the possibilities for $(L,J)$ are recorded in  \cite[Tables 2 and 3]{BH_Saxl}, together with the precise value of $r(L)$. Therefore, it is a routine exercise to read off the cases with $r(L) \leqs 4$, all of which are listed in Table \ref{tab:r(L)}. 

Next suppose $L \in \mathcal{G} \setminus \mathcal{L}$ and $\mathcal{P}(L) > 3/4$, in which case $4|J|^2r(L) > 3|L|$. We will establish the following claim, which extends \cite[Proposition 7.1]{BH_Saxl} and immediately implies that $r(L) \geqs 5$.

\vs

\noindent \emph{Claim. If $L \in \mathcal{G} \setminus \mathcal{L}$ and $\mathcal{P}(L) > 3/4$, then $16|J|^2 \leqs 3|L|$.}

\vs

To prove the claim, we consider each possibility for $T$ in turn. First assume $T = A_m$ is an alternating group. With the aid of {\sc Magma} \cite{Magma}, it is straightforward to verify the claim when $m\leqs 12$. Now assume $m>12$. As noted in the proof of \cite[Proposition 7.1]{BH_Saxl}, $m$ is a prime and $J = \mathrm{AGL}_1(m)\cap L$, which implies that 
\[
\frac{|J|^2}{|L|}\leqs\frac{m(m-1)}{(m-2)!}<\frac{3}{16}
\]
as required. The sporadic groups are also straightforward. Here the possibilities for $J$ can be read off from \cite{Wilson}, noting that we may exclude the cases in \cite[Table 2]{BH_Saxl} and \cite[Table 4]{B_sol} since we are assuming $\mathcal{P}(L) > 3/4$.

Next assume $T$ is an exceptional group of Lie type over $\mathbb{F}_q$. As noted in the proof of \cite[Proposition 7.1]{B_sol}, either $J = N_L(R)$ for some maximal torus $R$ or $T$ (see \cite[Table 5.2]{LSS_max}), or $(L,J)$ is one of the cases in \cite[Table 4]{BH_Saxl}. In addition, we may exclude the relevant cases in \cite[Table 2]{BH_Saxl}. 
The claim now follows by inspection. For example, suppose $T = {}^2B_2(q)$ and $J =N_L(R)$ is the normaliser of a maximal torus, where $q=2^f$ with $f \geqs 3$ odd. Here 
\[
|J| \leqs 4(q+\sqrt{2q}+1)\log q,\;\; |L| \geqs |T| = q^2(q^2+1)(q-1)
\]
and the claim follows if $f \geqs 5$. On the other hand, if $f=3$ then the condition $\mathcal{P}(L) > 3/4$ implies that $J\cap T = 7{:}2$ or $5{:}4$, whence $|J| \leqs 60$, $|L| \geqs 29120$ and once again the desired bound holds.

To complete the proof of the claim, we may assume $T$ is a classical group over $\mathbb{F}_q$ and we note that the possibilities for $L$ and $J$ are recorded in  \cite[Table 5]{BH_Saxl}. In each case, the precise structure of $J$ is given in \cite[Chapter 4]{KL} and the claim follows by inspection. For example, suppose $T = \LL_3^\epsilon(q)$ and $J$ is of type $\GL_1^\epsilon(q)\wr S_3$. If $q \geqs 19$, then the bounds $|L|>\frac{1}{6}q^{8}$ and $|J|\leqs 12(q+1)^2\log q$ are sufficient. For the remaining groups with $q \leqs 17$, excluding the cases with $\mathcal{P}(L) \leqs 3/4$ recorded in \cite[Table 3]{BH_Saxl}, we can verify the bound by working with the exact orders of $L$ and $J$. All of the other cases are very similar and we omit the details. This justifies the claim and we have now completed the proof of the proposition for the groups in $\mathcal{G} \setminus \mathcal{L}$.

Finally, let us assume $L \in \mathcal{L}$, so $L$ has socle $T = {\rm L}_2(q)$. Write $q=p^f$, where $p$ is a prime. If $q \leqs 81$ then it is a routine  exercise to determine all the groups with $r(L)\leqs 4$ using {\sc Magma} \cite{Magma} (see Section \ref{sss:reg1}) and one can check that all the relevant cases have been recorded in Table \ref{tab:r(L)}. For the remainder, we may assume $q >81$. 
	
First assume $J$ is of type $\GL_1(q)\wr S_2$. Here \cite[Lemma 4.7]{B_sol} implies that $b(L) = 2$ if and only if $\PGL_2(q)$ is not a proper subgroup of $L$. If $L = \PGL_2(q)$ then $r(L) = 1$ (see \cite[Example 2.5]{BG_Saxl}) and so this case is recorded in Table \ref{tab:r(L)}. Now assume $L\cap\PGL_2(q) = T$ and $q$ is odd. By \cite[Lemmas 4.3 and 4.4]{BH_Saxl}, we have $r(L)\leqs 4$ only if $r(\PSigmaL_2(q))\leqs 4$. By arguing as in the proof of \cite[Proposition 4.10]{BH_Saxl}, we see that $r(\PSigmaL_2(q)) = m/2f$, where $m$ is the number of non-squares in $\mathbb{F}_q$ that are not contained in any proper subfield of $\mathbb{F}_q$. Since every generator of the multiplicative group $\mathbb{F}_{q}^{\times}$ has this property, it follows that $m \geqs \phi(q-1)$, where $\phi$ is Euler's totient function. In particular, $r(L) \leqs 4$ only if $\phi(q-1) \leqs 8f$. By applying the lower bound on $\phi(q-1)$ in 
\cite[Lemma 4.9]{BH_Saxl} (the original reference is \cite[Theorem 15]{RS_phi}), we find that $\phi(q-1)>8f$ for every prime-power $q$ with $q>81$, whence $r(L) \geqs 5$ and no additional cases arise. 
	
Finally, suppose $J$ is of type $\GL_1(q^2)$, so $b(L) = 2$ if and only if $L$ does not contain $\PGL_2(q)$ (see \cite[Lemma 4.8]{B_sol}). Therefore, we may assume $q$ is odd and by combining Lemmas 4.12 and 4.13 in \cite{BH_Saxl}, we observe that $r(L)\leqs 4$ only if $r(\PSigmaL_2(q))\leqs 4$. By arguing as in the proof of \cite[Proposition 4.18]{BH_Saxl}, we deduce that $r(L) \leqs 4$ only if $\phi(q^2-1)\leqs 8f(q+1)$. But one can check that the bound in \cite[Lemma 4.9]{BH_Saxl} yields $\phi(q^2-1)> 8f(q+1)$ for every prime-power $q$ with $q>81$, and so once again we conclude that $r(L) \geqs 5$. This completes the proof of the proposition.
\end{proof}

\begin{rem}\label{r:LJ}
Let $(L,J)$ be one of the cases recorded in Table \ref{tab:r(L)}. In the fourth column of the table we give $r(T)$, which is the number of regular suborbits of $T$ on $\Gamma = L/J$. Note that each regular suborbit of $L$ is a union of $|L:T|$ regular suborbits of $T$, so $r(T) \geqs r(L) \cdot |L:T|$. If $L = {\rm PGL}_2(q)$ and $J$ is of type ${\rm GL}_1(q) \wr S_2$ then 
\[
r(L)=1,\;\; r(T) = \left\{\begin{array}{ll}
1 & \mbox{$q$ even} \\
(q+a)/4 & \mbox{$q$ odd} 
\end{array}\right.
\]
where $a=7$ if $q \equiv 1 \imod{4}$, otherwise $a=5$ (see the proof of \cite[Lemma 4.7]{B_sol}). For each of the remaining cases in Table \ref{tab:r(L)}, we can compute the exact value of $r(T)$ by proceeding as in Section \ref{sss:reg1}. This information will be useful in Section \ref{s:gen} (see the proof of Theorem \ref{thm:P=C_2_sol_gen}, for example). 
\end{rem}

As a corollary we obtain the following result, which establishes Theorem \ref{thm:wr_base-two}.

\begin{thm}\label{t:b2}
Let $G = L \wr P$ be a product type primitive group with soluble point stabiliser $J \wr P$. If $b(L)=2$, then either 
\begin{itemize}\addtolength{\itemsep}{0.2\baselineskip}
\item[{\rm (i)}] $b(G)=2$; or  
\item[{\rm (ii)}] $b(G) = 3$, $r(L) < D(P)$ and $(L,J,r(L))$ is one of the cases in Table \ref{tab:r(L)}.
\end{itemize}
\end{thm}

\begin{proof}
We may assume $b(G) \geqs 3$. Then Lemma \ref{l:plus1} implies that $b(G) = 3$ and $r(L) < D(P)$, where $D(P) \leqs 5$ by Theorem \ref{t:sol_dist}. Now apply Proposition \ref{prop:r(L)<=4}.
\end{proof}

\subsection{The general case.}\label{ss:gen}

In this section, we complete the proof of Theorem \ref{thm:reg(L,b(L))} and we establish Theorem \ref{thm:wr_base}. 
As before, $G = L\wr P$ is a product type primitive group with soluble point stabiliser $J \wr P$. Here $L \in \mathcal{S}$ and our first aim is to determine all the groups with ${\rm reg}(L,b(L)) \leqs 4$. 

Set $b = b(L)$ and $r = {\rm reg}(L,b)$.  Let $\mathcal{P}$ be the probability that a random $b$-tuple of points in $\Gamma$ is a base for $L$. It is straightforward to show that 
\begin{equation}\label{e:Qdef}
\mathcal{P} = \frac{|J|^br}{|L|^{b-1}},
\end{equation}
which is a generalisation of the expression in \eqref{e:Q} for the special case $b=2$. This immediately yields the following observation.

\begin{lem}\label{l:r5}
If $\mathcal{P} > 4|J|^b/|L|^{b-1}$ then $r \geqs 5$.
\end{lem}

In general, it is difficult to compute $\mathcal{P}$ precisely, but there is a method to obtain an upper bound on the complementary probability $\mathcal{Q} = 1 - \mathcal{P}$, which will typically be sufficient for our purposes. In order to do this, let $x_1, \ldots, x_m$ be a complete set of representatives of the conjugacy classes of elements in $L$ of prime order and let 
\[
{\rm fpr}(x_i) = \frac{|C_{\Gamma}(x_i)|}{|\Gamma|} = \frac{|x_i^L \cap J|}{|x_i^L|}
\]
be the \emph{fixed point ratio} of $x_i$ on $\Gamma$, where $C_{\Gamma}(x_i)$ is the set of points in $\Gamma$ fixed by $x_i$. Then it is straightforward to show that $\mathcal{Q} \leqs \widehat{\mathcal{Q}}$, where
\begin{equation}\label{e:whatQ}
\widehat{\mathcal{Q}} = \sum_{i=1}^m|x_i^L|\cdot {\rm fpr}(x_i)^b.
\end{equation}
This was originally observed and applied by Liebeck and Shalev in their proof of a conjecture of Cameron and Kantor on bases for almost simple primitive groups (see \cite[Theorem 1.3]{LSh}) and it provides a powerful approach for bounding the probability $\mathcal{Q}$. In particular, we deduce that $r \geqs 5$ if 
\begin{equation}\label{e:what}
\what{\mathcal{Q}} < 1 - \frac{4|J|^b}{|L|^{b-1}}.
\end{equation}

Three special cases arise in the proof of Proposition \ref{prop:2} below and it is convenient to handle them separately from the main argument. Note that in the following lemma, $J$ is the normaliser of a non-split maximal torus of $T$.

\begin{lem}\label{l:psl2}
Suppose $T = {\rm L}_2(q)$ and $J$ is of type ${\rm GL}_1(q^2)$, where $q \geqs 11$. Then $b(L) \leqs 3$, with equality if and only if ${\rm PGL}_2(q) \leqs L$. Moreover, if $b(L)=3$ then ${\rm reg}(L,3) \geqs 5$.
\end{lem}

\begin{proof}
The base size of $L$ is determined in \cite[Lemma 4.8]{B_sol} and so for the remainder we may assume ${\rm PGL}_2(q) \leqs L$. Write $q=p^f$ with $p$ a prime and note that $J \cap {\rm PGL}_{2}(q) = D_{2(q+1)}$. It suffices to verify the inequality in \eqref{e:what} (with $b=3$), so we need to consider the contributions to $\what{\mathcal{Q}}$ from the elements $x \in L$ of prime order; the argument below closely follows the proof of \cite[Lemma 4.6]{B_sol}. Note that ${\rm fpr}(x) = 0$ if $x^L \cap J$ is empty, so we are only interested in the relevant $L$-classes that meet $J$. Let $x \in J$ be an element of prime order $r$ and let $i_r(X)$ be the number of elements of order $r$ in $X$.

First assume $r=2$, so $x$ is either semisimple or unipotent (according to the parity of $p$) since ${\rm fpr}(x) = 0$ if $x$ is an involutory field automorphism. Then 
\[
|x^L\cap J| \leqs i_2(D_{2(q+1)}) \leqs q+2 = a, \;\; |x^L| \geqs \frac{1}{2}q(q-1) = b,
\]
so the contribution to $\what{\mathcal{Q}}$ from involutions is at most $\a_1 = b(a/b)^3$. 

Now suppose $r$ is odd, so either $r$ divides $q+1$ and $x$ is semisimple, or $q=q_0^r$ is an $r$-th power and $x$ is a field automorphism. If $x$ is semisimple, then $|x^T \cap J| = 2$, $|x^{T}| = q(q-1)$ and we note that $L$ has $(r-1)/2$ distinct $T$-classes of such elements. Therefore, the combined contribution to $\widehat{\mathcal{Q}}$ from semisimple elements of odd order is at most
\[
\sum_{r\in\pi}\frac{1}{2}(r-1)\cdot\frac{8}{q^2(q-1)^2}<\frac{4\log (q+1)}{q(q-1)^2} = \alpha_2,
\]
where $\pi$ is the set of odd prime divisors of $q+1$ (here we are using the fact that $|\pi|$ is at most $\log(q+1)$, recalling that all logarithms in this paper are in base $2$).

Finally, suppose $q=q_0^r$ and $x$ is a field automorphism of order $r$. Here
\[
|x^L \cap J| = \frac{q+1}{q_0+1},\;\; |x^L| = \frac{q(q^2-1)}{q_0(q_0^2-1)}
\]
and we note that there are $r-1$ distinct $T$-classes of field automorphisms of order $r$ in ${\rm Aut}(T)$. If $q_0 = 2$ then $q = 2^r$ and the contribution from field automorphisms is 
\[
(r-1) \cdot \frac{4(2^r+1)}{3.2^{2r}(2^r-1)^2} < 2^{-r} = q^{-1}.
\]
And for $q_0 \geqs 3$ we get
\[
\sum_{r \in \pi'} (r-1) \cdot  \frac{q_0^2(q_0-1)^2}{q^2(q-1)^2}\cdot \frac{q+1}{q_0+1} < \sum_{r \in \pi'} (r-1) \cdot 3q^{-3\left(1-\frac{1}{r}\right)} < q^{-1}\log\log q = \a_3,
\]
where $\pi'$ is the set of odd prime divisors of $f=\log_pq$.

By combining the above estimates, we conclude that $\what{\mathcal{Q}} \leqs \a_1+\a_2+\a_3$ and it is straightforward to check that the bound in \eqref{e:what} holds for all $q \geqs 11$.
\end{proof}

\begin{lem}\label{l:2b2}
Suppose $T = {}^2B_2(q)$ and $J \cap T = [q^{2}]{:}C_{q-1}$ is a Borel subgroup. Then $b(L)=3$ and either $G = {}^2B_2(8){:}3$ and ${\rm reg}(L,3) = 2$, or ${\rm reg}(L,3) \geqs 5$.
\end{lem}

\begin{proof}
Here $q = 2^f$, $f \geqs 3$ is odd and $|\Gamma| = q^2+1$. In addition, $b(L)=3$ by \cite[Theorem 1.2]{B_sol}. The cases with $q \leqs 2^7$ can be checked directly using {\sc Magma} (see Section \ref{sss:reg2}) and so we may assume $f \geqs 9$. As before, it suffices to show that the inequality in \eqref{e:what} is satisfied (with $b=3$). Set $J_0 = J \cap T$ and let $\chi$ be the permutation character $1_{J_0}^T$, which can be expressed as the sum of the trivial and Steinberg characters of $T$. The character table of $T$ is given in \cite{Suz}. 
	
First let $x \in T$ be an element of prime order $r$. If $r=2$ then $\chi(x) = 1$, so ${\rm fpr}(x) = 1/(q^2+1)$ and we have $|x^L| = (q^2+1)(q-1)$. Similarly, if $r$ divides $q-1$ then $\chi(x) = 2$, $|x^L| = q^2(q^2+1)$ and we note that there are at most $(q-2)/2$ distinct $T$-classes of such elements. Therefore, the contribution to $\what{\mathcal{Q}}$ from unipotent and semisimple elements is at most 
	\[
	\a_1 = \frac{q-1}{(q^2+1)^2} + \frac{1}{2}(q-2)\cdot \frac{8q^2}{(q^2+1)^2}.
	\]

Finally, suppose $x \in L$ is a field automorphism of prime order $r$ and note that $r$ is odd since $f$ is odd. Then 
\[
|x^T| = \frac{q^2(q^2+1)(q-1)}{q^{2/r}(q^{2/r}+1)(q^{1/r}-1)} = f(q,r)
\] 	
and $C_{J_0}(x)$ is a Borel subgroup of $C_T(x) = {}^2B_2(q^{1/r})$, so 
\[
|x^{T}\cap J_0x| = \frac{q^2(q-1)}{q^{2/r}(q^{1/r}-1)} = g(q,r).
\]
There are $r-1$ distinct $T$-classes of field automorphisms of order $r$ in ${\rm Aut}(T)$, so the combined contribution to $\what{\mathcal{Q}}$ from field automorphisms is
\begin{equation}\label{e:beta}
\b = \sum_{r \in \pi}(r-1) \cdot g(q,r)^3f(q,r)^{-2},
\end{equation}
where $\pi$ is the set of prime divisors of $f = \log q$. Set 
\begin{equation}\label{e:eqr}
e(q,r) = (r-1) \cdot g(q,r)^3f(q,r)^{-2}.
\end{equation}

If $q = 2^9$ or $2^{11}$ then it is straightforward to check that $\b<1/25$. Now assume $q \geqs 2^{13}$. If $f=r$ then $\b = e(q,r)$ and one checks that this is less than $q^{-1/2}$. Now assume $f$ is composite, so $q^{1/r} \geqs 8$ for each $r \in \pi$. Here  $f(q,r)>q^{5(1-1/r)}$ and $g(q,r) < 2q^{3(1-1/r)}$, which implies that $e(q,r) < 8(r-1)q_0^{1-r} < 4q^{-1/2}$. Since $|\pi|<\log\log q$, we deduce that
\[
\b < 4q^{-1/2}\log\log q.
\]

By combining the above estimates, we conclude that $\what{\mathcal{Q}} \leqs \a_1+\a_2$ for $q \geqs 2^9$, where $\a_2 = 1/25$ if $q \in \{2^9,2^{11}\}$, otherwise $\a_2 = 4q^{-1/2}\log\log q$. It is now routine to verify that the bound in \eqref{e:what} is satisfied for all $q \geqs 2^9$. 
\end{proof}

Note that in the statement of the next lemma we assume $q \geqs 27$. Indeed, if $q=3$ then $T = {}^2G_2(q)' \cong {\rm L}_2(8)$ and $J \cap T$ corresponds to a Borel subgroup of ${\rm L}_2(8)$ (if $G = T$, then $b(L)=3$ and ${\rm reg}(L,3) = 1$, otherwise $b(L)=4$ and ${\rm reg}(L,4) = 2$). 

\begin{lem}\label{l:2g2}
Suppose $T = {}^2G_2(q)$ and $J \cap T = [q^{3}]{:}C_{q-1}$ is a Borel subgroup, where $q \geqs 27$. Then $b(L)=3$ and ${\rm reg}(L,3) \geqs 5$.
\end{lem}

\begin{proof}
Here $q=3^f$, $f \geqs 3$ is odd, $|\Gamma| = q^3+1$ and $b(L)=3$ by \cite[Theorem 1.2]{B_sol}. The case $q=27$ can be checked directly using {\sc Magma} and so we may assume $f \geqs 5$. We now proceed as in the proof of the previous two lemmas, working with fixed point ratio estimates to derive a suitable upper bound on $\what{\mathcal{Q}}$ which allows us to verify the inequality in \eqref{e:what} (with $b=3$). As before, set $J_0 = J \cap T$ and let $\chi = 1_{J_0}^T$ be the permutation character. Once again, $\chi$ is the sum of the trivial and Steinberg characters of $T$ (the character table of $T$ is presented in \cite{Ward}). 

First let $x \in T$ be an element of prime order $r$. If $r=3$ then $\chi(x) = 1$ and thus ${\rm fpr}(x) = 1/(q^3+1)$. In addition, we calculate that there are precisely $(q^3+1)(q^2-1)$ elements in $T$ of order $3$ (forming three distinct conjugacy classes). Next assume $r$ divides $q-1$. If $r=2$ then $|x^L| = q^2(q^2-q+1)$ (there is a unique class of involutions in $T$) and we have $\chi(x) = q+1$, so ${\rm fpr}(x) = 1/(q^2-q+1)$. Now suppose $r$ is an odd prime divisor of $q-1$. Here $\chi(x) = 2$, so ${\rm fpr}(x) = 2/(q^3+1)$ and $|x^T| = q^3(q^3+1)$. Since there are at most $(q-3)/2$ distinct $T$-classes of such elements, we conclude that the contribution to $\what{\mathcal{Q}}$ from unipotent and semisimple elements is at most 
	\[
	\a_1 = \frac{q^2-1}{(q^3+1)^2} + \frac{q^2}{(q^2-q+1)^2} + \frac{1}{2}(q-3)\cdot \frac{8q^3}{(q^3+1)^2}.
	\]

Now assume $x \in L$ is a field automorphism of prime order $r$, so  $r$ is odd and we have
\[
|x^T| = \frac{q^3(q^3+1)(q-1)}{q^{3/r}(q^{3/r}+1)(q^{1/r}-1)} = f(q,r)
\] 	
and 
\[
|x^T \cap J_0x| = \frac{q^3(q-1)}{q^{3/r}(q^{1/r}-1)} = g(q,r). 
\]
Therefore, the combined contribution to $\what{\mathcal{Q}}$ from field automorphisms is $\beta$, as defined in \eqref{e:beta}. Define $e(q,r)$ as in \eqref{e:eqr}.

If $f=r$ then $\b = e(q,r)< q^{-1}$ for all $q \geqs 3^5$. Now assume $f$ is composite, so $q^{1/r} \geqs 27$ for each $r \in \pi$. Then one checks that $f(q,r)>q^{7(1-1/r)}$ and $g(q,r) < 2q^{4(1-1/r)}$, which implies that $e(q,r) < 8(r-1)q_0^{-2(r-1)} < 4q^{-1}$. Since $|\pi|<\log\log q$, we conclude that $\b < \a_2 = 4q^{-1}\log\log q$ for all $q \geqs 3^5$.

Therefore, $\what{\mathcal{Q}} \leqs \a_1+\a_2$ and it is now straightforward to verify the bound in \eqref{e:what}. 
\end{proof}

We are now in a position to prove the following result. When combined with Proposition \ref{prop:r(L)<=4}, this completes the proof of Theorem 
\ref{thm:reg(L,b(L))}. Recall that Table \ref{tab:reg(L)} is given in Section \ref{s:tab}.

\begin{prop}\label{prop:2}
Let $L \leqs {\rm Sym}(\Gamma)$ be an almost simple primitive group with socle $T$ and soluble point stabiliser $J$. If $b(L) \geqs 3$ then ${\rm reg}(L,b(L)) \leqs 4$ if and only if $(L,J)$ is one of the cases in Table \ref{tab:reg(L)}.
\end{prop}

\begin{proof}
Set $b=b(L)$, $r = {\rm reg}(L,b)$ and recall that $b \leqs 5$ by the main theorem of \cite{B_sol}. The proof of \cite[Theorem 8.2]{B_sol} gives $r \geqs 5$ if $b=5$, so we may assume $b \in \{3,4\}$ and we note that the possibilities for $L$ and $J$ are recorded in \cite[Tables 4--7]{B_sol}. Recall that $r \geqs 5$ if \eqref{e:what} holds, where $\what{\mathcal{Q}}$ is defined in \eqref{e:whatQ}.

For most of the cases appearing in \cite[Tables 4--7]{B_sol}, an explicit upper bound on $\what{\mathcal{Q}}$ is given in \cite{B_sol} and we can usually use this to verify the inequality in \eqref{e:what}. However, this approach is not always effective because the given upper bound on $\widehat{\mathcal{Q}}$ is either too large, or is not defined. As explained below, in these remaining cases we will typically use {\sc Magma} to directly compute $r$, implementing the approach described in Section \ref{sss:reg2}. We divide the remainder of the proof into three cases.

\vs

\noindent \emph{Case 1. $b(L)=4$.}

\vs

First assume $b = 4$. By inspecting the relevant tables in \cite{B_sol}, we see that $T = \LL_2(q)$ with $J$ of type $P_1$ (a Borel subgroup of $L$) is the only infinite family that arises. Let us first consider this special case. For $q>32$, an explicit upper bound on $\widehat{\mathcal{Q}}$ is presented as a function of $q$ in the proof of \cite[Lemma 4.4]{B_sol} and it is a routine exercise to check that the bound in \eqref{e:what} is satisfied. The remaining groups with $q \leqs 32$ can be handled using {\sc Magma}, which allows us to compute $r$ precisely. In particular, the groups with $r \leqs 4$ are recorded in Table \ref{tab:reg(L)}. We can apply the same computational approach to handle all the remaining groups with $b=4$ appearing in \cite[Tables 4--7]{B_sol}, considering each group in turn.
	
\vs

\noindent \emph{Case 2. $b(L)=3$, $J$ non-parabolic.}

\vs

To complete the proof, we may assume $b = 3$. We begin by assuming $L$ is not a group of Lie type in a parabolic action, so either 
\begin{itemize}\addtolength{\itemsep}{0.2\baselineskip}
\item[{\rm (a)}] $T = \LL_2(q)$ and $J$ is of type $\GL_1(q)\wr S_2$ or $\GL_1(q^2)$; or 
\item[{\rm (b)}] $(L,J)$ is one of a finite number of sporadic cases in \cite[Tables 4 and 7]{B_sol}. 
\end{itemize}

First let us consider the cases in (a), noting that the precise base size of $L$ is recorded in \cite[Lemmas 4.7 and 4.8]{B_sol}. The groups with $q \leqs 37$ can be handled using {\sc Magma} and we find that $r \leqs 4$ if and only if $(L,J,r)$ is one of the following:
\[
({\rm L}_2(4), D_{10},2), \, ({\rm L}_2(4).2, 5{:}4,1), \, ({\rm L}_2(4).2, D_{12},4), \, ({\rm PGL}_2(5), D_{12},4),
\]
all of which are recorded in Table \ref{tab:reg(L)} with $L = A_5$ or $S_5$ (see Remark \ref{r:cases}). If $J$ is of type ${\rm GL}_1(q) \wr S_2$ and $q > 37$, then the proof of \cite[Lemma 4.6]{B_sol} yields the upper bound $\widehat{\mathcal{Q}}<2q^{-1/2}$ and we deduce that \eqref{e:what} holds. For $J$ of type ${\rm GL}_1(q^2)$, we refer the reader to Lemma \ref{l:psl2}.

Next let us turn to the groups in (b) above. Here we apply computational methods, after  first dividing the groups into two subcollections according to the size of $\Gamma$. By inspection, one can check that $|\Gamma| \geqs 5\times 10^6$ if and only if $T = \POmega_8^+(3)$ and $J$ is of type $\mathrm{O}_4^+(3)\wr S_2$, or $T \in \{{\rm Fi}_{22},{\rm Fi}_{23}\}$ and $J$ is the $3$-local subgroup of $L$ recorded in \cite[Table 4]{B_sol}. Here we can use {\sc Magma} to compute $\widehat{\mathcal{Q}}$ precisely, which then allows us to verify the bound in \eqref{e:what} (see \cite[Section 2.2]{BH_Saxl} for further details). In each of the remaining cases, we can apply the usual approach to compute $r$, as explained in Section \ref{sss:reg2}.

\vs

\noindent \emph{Case 3. $b(L)=3$, $J$ parabolic.}

\vs

For the remainder of the proof, we may assume $L$ is a group of Lie type over $\mathbb{F}_q$, $J$ is a maximal parabolic subgroup and $b = 3$. Write $q=p^f$, where $p$ is a prime, and let $\phi$ be a field automorphism of $T$ of order $f$.

First assume $L$ is an exceptional group. Here the possibilities for $(L,J)$ are recorded in \cite[Table 5]{B_sol} and by inspection we see that one of the following holds:
\begin{itemize}\addtolength{\itemsep}{0.2\baselineskip}
	\item[{\rm (a$'$)}] $T = G_2(q)$ and $J\cap T = [q^6]{:}C_{q-1}^2$, where $p = 3$ and $L\not\leqs\langle T,\phi\rangle$.
	\item[{\rm (b$'$)}] $T = {^2}B_2(q)$ and $J\cap T = [q^{2}]{:}C_{q-1}$.
	\item[{\rm (c$'$)}] $T = {^2}G_2(q)$ and $J\cap T = [q^{3}]{:}C_{q-1}$, where $q\geqs 27$.
	\item[{\rm (d$'$)}] $(L,J)$ is one of a finite number of sporadic cases in \cite[Table 5]{B_sol} with $q \leqs 3$.
\end{itemize}

Consider case (a$'$). If $q\geqs 27$, then the explicit upper bound on $\widehat{\mathcal{Q}}$ presented in the proof of \cite[Lemma 5.9]{B_sol} is sufficient, while the groups with $q \in \{3,9\}$ can be handled directly using {\sc Magma} (note that if $q=9$ then we can construct $J$ by observing that $J = N_L(K)$ for some subgroup $K<T$ of order $9^6$). For (b$'$) and (c$'$) we refer the reader to Lemmas \ref{l:2b2} and \ref{l:2g2}. The cases in (d$'$) can all be handled computationally using {\sc Magma}. First assume $L = F_4(2).2$ and $J = [2^{22}].S_3^2.2$, in which case $|\Gamma| = 21928725$. Here we construct $L$ as a permutation group of degree $139776$ and we use the fact that $J = N_L(K)$ with $|K|=2^{22}$ to construct $J$ (here the function \texttt{MaximalSubgroups} is not effective). We then compute $\widehat{\mathcal{Q}}$ and we check that \eqref{e:what} holds (given the size of $\Gamma$, this appears to be the most efficient way to handle this case). We use a similar method in the case where $T = {}^3D_4(3)$ and $T\cap J = [3^{11}]{:}(26\circ\SL_2(3)).2$. All of the remaining cases in (d$'$) can be handled in the usual fashion and we can compute $r$ precisely. In this way, we deduce that $r \leqs 4$ if and only if $L = G_2(3)$ and $J = [3^5]{:}\GL_2(3)$, in which case $r = 4$.  

Finally, let us assume $L$ is a classical group and $J$ is a parabolic subgroup. By inspecting \cite[Table 6]{B_sol}, we see that one of the following holds:  
\begin{itemize}\addtolength{\itemsep}{0.2\baselineskip}
	\item[{\rm (a$''$)}] $T = \LL_2(q)$ and $J$ is of type $P_1$.
	\item[{\rm (b$''$)}] $T = \LL_3(q)$ and $J$ is of type $P_{1,2}$.
	\item[{\rm (c$''$)}] $T = \UU_3(q)$ and $J$ is of type $P_1$.
	\item[{\rm (d$''$)}] $T = \PSp_4(q)$ and $J\cap T = [q^4]{:}C_{q-1}^2$, where $q \geqs 4$ is even and $L\not\leqs\langle T,\phi\rangle$.
	\item[{\rm (e$''$)}] $(L,J)$ is one of a finite number of sporadic cases in \cite[Table 6]{B_sol} with $q \leqs 3$.
\end{itemize}

First consider case (a$''$). Here \cite[Lemma 4.5]{B_sol} implies that $b(L)=3$ if and only if $L\leqs \PGL_2(q)$, or $q=p^f$ is odd, $f$ is even and $L = \langle T,\delta\phi^{f/2}\rangle = T.2$, where ${\rm PGL}_2(q) = \la T, \delta \ra$. In other words,  either $L$ is sharply $3$-transitive on $\Gamma$ and thus $r=1$, or $q$ is odd, $L = T$ and $r=2$.

In cases (b$''$), (c$''$) and (d$''$), an explicit upper bound on $\widehat{\mathcal{Q}}$ is presented in the proofs of \cite[Lemmas 5.6--5.8]{B_sol}. Using this bound, one can check that \eqref{e:what} holds for $q>128, 32, 32$, respectively. Consider case (c$''$) for example. For $q>10^4$, the bound 
\[
\widehat{\mathcal{Q}}<8q^{-1/2}\log\log q+4q^{-1}+q^{-3}
\]
from the proof of \cite[Lemma 5.7]{B_sol} is sufficient. Similarly, for $9 \leqs q \leqs 10^4$, we can use a more accurate upper bound on $\widehat{\mathcal{Q}}$ in \cite{B_sol} to reduce our analysis to the groups with 
\[
q\in \{3,4,5,7,8,9,16,27,32\}
\]
and at this point, we can use {\sc Magma} to compute $r$ precisely (here it is convenient to note that the standard permutation representation of $L$ in {\sc Magma} corresponds to the action of $L$ on $\Gamma$). We find that there are several cases with $r \leqs 4$, all of which have been listed in Table \ref{tab:reg(L)}. Cases (b$''$) and (d$''$) can be handled in the same way. Similarly, we can use {\sc Magma} to compute $r$ for each group in case (e$''$). 
\end{proof}

This completes the proof of Theorem \ref{thm:reg(L,b(L))}. By combining Theorems \ref{thm:reg(L,b(L))} and \ref{thm:wr_base-two} with Lemma \ref{l:plus1}, we obtain Theorem \ref{thm:wr_base} as an immediate corollary.

\section{Regular suborbits}\label{s:reg}

In this section, we will derive an expression for the number of regular suborbits of a product type primitive group of the form $G = L\wr P$ acting on $\O = \Gamma^k$, which we denote by $r(G)$. This is the content of Theorem \ref{thm:r(W)_cal}, which yields Corollary \ref{c:unique} as an immediate application. We then go on to establish Theorem \ref{t:new}, which describes the groups $G = L \wr P$ with $r(G) = 1$ and $P \leqs S_k$ primitive. 
The latter result relies on our classification of the primitive groups $P \leqs S_k$ with a unique regular orbit on the power set of $[k]$ (see  Corollary \ref{c:unique_power}), which extends earlier work of Seress \cite{S_dist} and Dolfi \cite{D_dist}.
	
The proof of Theorem \ref{thm:BC} in \cite{BC} involves a specific construction, which can be used to describe the pairs of points in $\Omega$ that form a base for $G = L \wr P$. This yields the following result, which is \cite[Lemma 2.8]{BG_Saxl}. The proof of this lemma was omitted in \cite{BG_Saxl}, but here we give the details because the argument will be needed in our proof of Theorem \ref{thm:r(W)_cal} below.

\begin{lem}
	\label{l:BG_l:2.8}
	Let $G = L \wr P$ be a product type primitive group acting on $\O = \Gamma^k$ and let $\alpha = (\alpha_1,\dots,\alpha_k)$ and $\beta = (\beta_1,\dots,\beta_k)$ be elements in $\Omega$. Define a partition $\Pi$ of $\{1,\dots,k\}$ such that $i$ and $j$ are in the same part if and only if $(\alpha_i,\beta_i)$ and $(\alpha_j,\beta_j)$ are in the same $L$-orbit on $\Gamma^2$. Then $\{\alpha,\beta\}$ is a base for $G$ if and only if each $\{\alpha_i,\beta_i\}$ is a base for $L$ and $\Pi$ is a distinguishing partition for $P$.
\end{lem}

\begin{proof}
	Let $\mathcal{A}$ be the set of $k\times 2$ arrays of elements in $\Gamma$ such that
	\begin{itemize}\addtolength{\itemsep}{0.2\baselineskip}
		\item[(A1)] each row of the array is an ordered base for $L$; and
		\item[(A2)] the partition of $\{1,\dots,k\}$ with respect to $L$-orbits on rows is a distinguishing partition for $P$.
	\end{itemize}	
Let $\mathcal{B}$ be the set of ordered pairs of bases for $G$ and let $A$ be an array in $\mathcal{A}$. We claim that the two columns of $A$ form a base for $G$, so there is a natural map $f:\mathcal{A}\to\mathcal{B}$.

To justify the claim, let $\alpha = (\alpha_1,\dots,\alpha_k)$ and $\beta = (\beta_1,\dots,\beta_k)$ be the two columns of $A$ and suppose $x\in G_\alpha\cap G_\beta$. Write $x = z\sigma$, where $z = (z_1,\dots,z_k)\in L^k$ and $\sigma\in P$. If $i^\sigma = j$, then $(\alpha_i,\beta_i)^{z_i} = (\alpha_j,\beta_j)$ and thus $(\alpha_i,\beta_i)$ and $(\alpha_j,\beta_j)$ are in the same $L$-orbit. It follows that $\sigma$ fixes the partition given by (A2), so $\sigma = 1$ since this is a distinguishing partition for $P$. Hence, $z_i\in L_{\alpha_i}\cap L_{\beta_i} = 1$ by (A1) and thus $x = 1$ as required. 

In order to prove the lemma, it suffices to show that $f$ is a bijection. It is clear that $f$ is injective. To show that $f$ is surjective, let $A$ be a $k \times 2$ array of elements in $\Gamma$ and assume the two columns of $A$ form a base for $G$. We need to show that $A$ satisfies the conditions labelled (A1) and (A2) above. 
	
Suppose there is a row of $A$ that does not form an ordered base for $L$. There there exists a non-trivial element in $L$ fixing both entries, which implies that there is a non-trivial element in $L^k$ stabilising the two columns of $A$ pointwise. But this is  incompatible with the fact that the two columns of $A$ form a base for $G$, whence (A1) is satisfied. 

Now let us turn to (A2). As we have just noted, each row of $A$ is an ordered base for $L$ of size $2$. Now $L$ acts on the set of such bases, so we can use the rows of $A$ to construct a partition $\Pi$ of $\{1,\dots,k\}$, where $i$ and $j$ are in the same part if and only if the $i$-th and $j$-th rows of $A$ are in the same $L$-orbit. We need to show that $\Pi$ is a distinguishing partition for $P$. Without loss of generality, we may assume that any two rows of $A$ in the same $L$-orbit are equal (with this assumption, note that the two columns of $A$ still form a base for $G$). Let $\sigma\in P$ and note that $\sigma$ permutes the rows of $A$. Now any two rows of $A$ are equal if and only if they are in the same part of $\Pi$, so if $\s$ fixes $\Pi$ then 
$\s$ must fix every entry of $A$. In particular, this implies that $\sigma$ fixes the two columns of $A$, which form a base for $G$. Hence, $\sigma = 1$ and thus $\Pi$ is a distinguishing partition. Therefore, property (A2) also holds and we conclude that $f$ is surjective.
\end{proof}

We will now use this lemma to calculate $r(L\wr P)$. In order to do this, we need the following definition.

\begin{defn}\label{d:tm}
Let $P \leqs S_k$ be a permutation group of degree $k$ and let $1 \leqs m \leqs k$ be an integer. We will write $t_m$ for the number of (unordered) distinguishing partitions for the action of $P$ on $[k]$ into $m$ non-empty parts.
\end{defn}

\begin{rem}\label{r:tmm}
Note that $t_m>0$ if and only if $m \geqs D(P)$. In general, it is rather difficult to compute $t_m$ precisely, but this is possible in some special cases. For example, if $P=S_k$ then $D(P)=k$ and $t_k=1$ since the partition of $[k] = \{1, \ldots, k\}$ into singletons is clearly the only distinguishing partition for $P$. Similarly, if $P = A_k$ then $D(P)=k-1$ and we have $t_{k-1} = k(k-1)/2$. As discussed in Section \ref{s:intro}, if $m \geqs D(P)$ then  
\begin{equation}\label{e:bds}
\frac{|P|}{m!} \leqs t_m \leqs S(k,m),
\end{equation}
where $S(k,m)$ denotes the number of partitions of $[k]$ into $m$ non-empty parts (a Stirling number of the second kind).
\end{rem}

The following result is Theorem \ref{thm:r(W)_cal}.

\begin{thm}\label{t:reg}
	Let $G = L \wr P$ be a product type primitive group acting on $\O = \Gamma^k$. Then 
	\[
	r(G) =\frac{1}{|P|}\sum_{m=D(P)}^km!{r(L)\choose m}t_m.
	\]
\end{thm}

\begin{proof}
Set $r = r(L)$ and $D = D(P)$. If $r < D$ then Theorem \ref{thm:BC} gives $b(G) \geqs 3$, so $r(G) = 0$ and the result follows (note that each summand in the given expression is $0$ in this situation). Now assume $r \geqs D$ and adopt the notation in the proof of Lemma \ref{l:BG_l:2.8}, where we showed that $|\mathcal{A}| = |\mathcal{B}|$. In view of \eqref{e:Q}, we see that $|\mathcal{B}| = |G|r(G)$ and so we just need to count the number of $k \times 2$ arrays satisfying (A1) and (A2).
	
Let $A$ be an arbitrary $k\times 2$ array satisfying (A1) and (A2), where $\Pi$ is the partition of $\{1, \ldots, k\}$ corresponding to (A2). Let $m$ be the number of parts in $\Pi$, so $D \leqs m \leqs k$ and there are precisely $t_m$ possibilities for $\Pi$ in total. 
Each part comprising $\Pi$ corresponds to a distinct regular $L$-orbit on $\Gamma^2$. Since there are $r$ such orbits, it follows that there are $m!{r\choose m}$ different ways to label the rows of $A$ by regular $L$-orbits on $\Gamma^2$. In addition, since each of these $L$-orbits has length $|L|$, there are $|L|^k$ possibilities for $A$ with respect to each choice of labelling of rows by regular $L$-orbits. To summarise, we deduce that 
	\[
	|G|r(G) = |\mathcal{A}| = \sum_{m=D}^km!{r\choose m}t_m|L|^k,
	\]
	and the result follows since $|G| = |L|^k|P|$.
\end{proof}

\begin{rem}\label{r:trans}
Notice that Theorem \ref{t:reg} holds for any group of the form $G = L \wr P$ with its product action on $\O = \Gamma^k$, where $L \leqs {\rm Sym}(\Gamma)$ is transitive. In particular, there is no need to impose any conditions on $P$ (for example, if $P=1$ then $r(G) = r(L)^k$).
\end{rem}

We now present several corollaries of Theorem \ref{thm:r(W)_cal}. Further applications will be discussed in the next section.

\begin{cor}
	\label{cor:val_P=Sk}
	Let $G = L \wr P$ be a product type primitive group acting on $\O = \Gamma^k$, where $P=S_k$. Then $r(G) = {r(L) \choose k}$.
\end{cor}

\begin{proof}
This is an immediate application of Theorem \ref{thm:r(W)_cal}, noting that for $P=S_k$ we have $D(P)=k$ and $t_k=1$. 
\end{proof}

The next corollary is \cite[Proposition 3.5]{BG_Saxl}; here we give a short proof, as an application of Theorem \ref{thm:r(W)_cal}.

\begin{cor}
	\label{cor:val_P=Cp}
	Let $G = L \wr P$ be a product type primitive group acting on $\O = \Gamma^k$, where $P = C_k$ and $k$ is a prime. Then $r(G) = (r(L)^k-r(L))/k$. 
	\end{cor}

\begin{proof}	
Set $r = r(L)$ and note that $D(P) = 2$, so Theorem \ref{thm:BC} implies that $r(G) \geqs 1$ if and only if $r \geqs 2$. Now any partition of $[k]$ into at least two parts is a distinguishing partition for $P$, so for $m \geqs 2$ we observe that $t_m$ coincides with the total number of partitions of $[k]$ into $m$ parts. In other words, $t_m = S(k,m)$ is a Stirling number of the second kind. Therefore,
	\[
	\sum_{m = 2}^k m!{r\choose m}t_m = \sum_{m = 2}^k m!{r\choose m}S(k,m) = r^k-r,
	\] 
where the final equality follows from a basic property of Stirling numbers of the second kind (see \cite[p.75]{Stanley}, for example). By applying Theorem \ref{thm:r(W)_cal}, we conclude that $r(G)k = r^k-r$.
\end{proof}

For the remainder of Section \ref{s:reg}, we focus our attention on the groups with $r(G)=1$. As a starting point, we establish Corollary \ref{c:unique}.

\begin{cor}\label{c:reg}
	Let $G = L \wr P$ be a product type primitive group. Then $r(G) = 1$ if and only if $r(L) = D(P)$ and $t_{D(P)} = |P|/D(P)!$.
\end{cor}

\begin{proof}
Set $r = r(L)$ and $D = D(P)$. If $r = D$ and $t_D = |P|/D!$ then Theorem \ref{thm:r(W)_cal} yields $r(G)=1$. On the other hand, if $r(G)=1$ then $r \geqs D$ by Theorem \ref{thm:BC} and we have $D!t_D\geqs |P|$ (see \eqref{e:bds}), whence 	
	\[ 
	1 = r(G) = \frac{1}{|P|}\sum_{m = D}^km!{r\choose m}t_m\geqs \frac{1}{|P|}D!{r\choose D}t_D \geqs 1.
	\]
Therefore $r = D$ and $t_D = |P|/D!$ as required.
\end{proof}

In view of the corollary, we are interested in understanding the condition 
\begin{equation}\label{e:tdp}
t_{D(P)} = \frac{|P|}{D(P)!}.
\end{equation}
As previously noted (see Remark \ref{r:tmm}), $P=S_k$ has this property, but $P = A_k$ does not. Our proof of Theorem \ref{t:new} relies on the classification of the primitive groups with this property, which is stated as Proposition \ref{p:cond_5} below.

Fix an integer $m \in \{1, \ldots, k\}$ and let 
\[
\mathcal{P}_m([k]) = \left\{ (\pi_1, \ldots, \pi_m) \,:\, \mbox{$\pi_i \subseteq [k]$, $\pi_i \cap \pi_j = \emptyset$ for $i\ne j$, $\bigcup_i \pi_i = [k]$} \right\} 
\]
be the set of ordered partitions of $[k]$ into $m$ parts, where some of the parts are allowed to be the empty set. Note that we may identify $\mathcal{P}_2([k])$ with the power set of $[k]$. Then $P$ acts naturally on $\mathcal{P}_m([k])$ via
\[
(\pi_1,\ldots,\pi_{m})^{\sigma} = (\pi_1^{\sigma},\ldots,\pi_{m}^{\sigma})
\]
and we see that $(\pi_1,\ldots,\pi_{m})$ is in a regular $P$-orbit if and only if $\{\pi_1,\dots,\pi_{m}\}$ is a distinguishing partition for $P$. The following lemma, which is an immediate consequence of this observation, will be useful.

\begin{lem}\label{l:unique}
Let $P \leqs S_k$ be a transitive group of degree $k$. Then $t_{D(P)} = |P|/D(P)!$ if and only if $P$ has a unique regular orbit on $\mathcal{P}_{D(P)}([k])$.
\end{lem}

Suppose \eqref{e:tdp} holds and $\{\pi_1, \ldots, \pi_{D(P)}\}$ is a distinguishing partition for $P$. Then the unique regular $P$-orbit on $\mathcal{P}_{D(P)}([k])$ is represented by 
$(\pi_1, \ldots, \pi_{D(P)})$ and all $D(P)!$ rearrangements of this ordered partition are in the same $P$-orbit. Therefore, the $\pi_i$ all have the same size and thus $D(P)$ must divide $k$.

Since we are focussing on the case where $P \leqs S_k$ is primitive, we have a special interest in the groups with $D(P) = 2$ (see Theorem \ref{t:prim}). Here Lemma \ref{l:unique} implies that \eqref{e:tdp} holds if and only if $P$ has a unique regular orbit on the power set of $[k]$ and we will classify the primitive groups with this property (see Corollary \ref{c:unique_power} below). This can be viewed as a natural extension of the main theorem of \cite{S_dist} (stated here as Theorem \ref{t:prim}), which determines the primitive groups $P \leqs S_k$ with a regular orbit on the power set of $[k]$. It also extends earlier work of Dolfi \cite{D_dist}, where the primitive groups with a unique regular orbit on $\mathcal{P}_3([k])$ or $\mathcal{P}_4([k])$ are classified.

For the remainder of this section, define 
\begin{equation}\label{e:X}
X = \{\Lambda \subseteq [k] \, :\, |\Lambda| \ne k/2\}
\end{equation}
and note that $P$ has a natural action on $X$. Of course, if $k$ is odd then $X$ coincides with the power set of $[k]$.

\begin{lem}\label{l:X_unique}
If $P$ has a regular orbit on $X$, then $t_2>|P|/2$.
\end{lem}

\begin{proof}
If $\Lambda$ is in a regular orbit of $P$ on $X$, then its complement $[k]\setminus \Lambda$ is also contained in a regular $P$-orbit. Therefore, $P$ has at least two regular orbits on $\mathcal{P}_2([k])$ and thus $t_2>|P|/2$ by Lemma \ref{l:unique}.
\end{proof}

Clearly, if $D(P)=2$ and $k$ is odd, then $P$ has a regular orbit on $X$ and thus $t_2>|P|/2$ by the lemma. Therefore, we are interested in the case where $D(P)=2$ and $k$ is even, so 
\[
|X| = 2^k-{k\choose k/2}.
\]

\begin{lem}
	\label{l:|X|}
	We have $|X| \geqs 2^{k-1}$.
\end{lem}

\begin{proof}
We may assume $k$ is even, so it suffices to show that
	\[
	{k\choose k/2}\leqs 2^{k-1}.
	\]
	To do this, we will use the following bounds on $m!$ (valid for all $m \geqs 1$), which are a consequence of Stirling's approximation (see \cite{R_Stirling}):
	\[
	\sqrt{2 \pi m}\left(\frac{m}{e}\right)^{m} e^{1/(12m +1)}<m!<\sqrt{2 \pi m}\left(\frac{m}{e}\right)^{m} e^{1/12m}.
	\]
	Therefore, for $k \geqs 4$ we have
	\[
	{k\choose k/2} = \frac{k!}{(k/2)!^2} < \frac{\sqrt{2\pi k}\left(\frac{k}{e}\right)^ke^{1/12k}}{\left(\sqrt{\pi k}\left(\frac{k}{2e}\right)^{k/2}e^{1/(6k+1)}\right)^2} = \left(\frac{2\sqrt{2}}{\sqrt{\pi k}}e^{\frac{1}{12k}-\frac{2}{6k+1}}\right)2^{k-1} < 2^{k-1}
	\]
and the result follows, noting that the case $k = 2$ is clear. 
\end{proof}

Let $\mu(P)$ be the \emph{minimal degree} of $P$, which is the minimal number of points in $[k]$ moved by a non-identity element of $P$. 

\begin{lem}
	\label{l:mindeg}
If $k$ is even and $|P|<2^{\mu(P)/2-1}$, then $P$ has a regular orbit on $X$.
\end{lem}

\begin{proof}
We follow the proof of the main theorem of \cite{CNS_dist}. Suppose $P$ has no regular orbit on $X$, which means that each set in $X$ is fixed (setwise) by some prime order element of $P$. Therefore,
\[
|X| = \left| \bigcup_{\s \in \mathcal{R}} C_X(\s) \right| \leqs \sum_{\s \in \mathcal{R}} |C_{X}(\s)|,
\]
where $\mathcal{R}$ is the set of prime order elements in $P$ and $C_X(\s)$ is the set of fixed points of $\s$ on $X$. If $\s \in P$ has prime order $r$, then $\s$ has cycle-shape $(r^m,1^{k-mr})$ on $[k]$ for some $m \geqs 1$ and it is easy to see that 
\[
|C_X(\s)| \leqs 2^{k-m(r-1)} \leqs 2^{k-(r-1)\mu(P)/r} \leqs 2^{k-\mu(P)/2}
\]
since $\mu(P) \leqs mr$ and $r \geqs 2$. By applying Lemma \ref{l:|X|} we deduce that
\[
2^{k-1} \leqs |X| \leqs 2^{k-\mu(P)/2}|P|
\]
and thus $|P| \geqs 2^{\mu(P)/2-1}$. 
\end{proof}

We will use Lemma \ref{l:mindeg} and the O'Nan-Scott theorem to prove the following result. 

\begin{prop}\label{p:ddagger}
Let $P \leqs S_k$ be a primitive group with $D(P) = 2$. Then $P$ has no regular orbit on $X$ if and only if $(k,P) = (2,S_2)$ or $(16,2^4{:}{\rm O}_4^-(2))$.
\end{prop}

We then obtain the following as an immediate corollary.

\begin{cor}\label{c:unique_power}
Let $P \leqs S_k$ be a primitive group. Then $P$ has a unique regular orbit on the power set of $[k]$ if and only if $(k,P) = (2,S_2)$.
\end{cor}

\begin{proof}
Suppose $P$ has a unique regular orbit on the power set of $[k]$, so $D(P) = 2$ and $P$ has no regular orbit on $X$ by Lemma \ref{l:X_unique}. If $(k,P)=(16, 2^4{:}{\rm O}_4^-(2))$ then one checks that $P$ has two regular orbits on the power set of $[k]$. Now apply Proposition \ref{p:ddagger}.
\end{proof}

We now focus on the proof of Proposition \ref{p:ddagger}, considering each family of primitive groups in turn (see Table \ref{tab:prim}).

\begin{lem}
	\label{l:prim_nr}
	Let $P \leqs S_k$ be a primitive group of type III, IV or V. Then $P$ has a regular orbit on $X$.
\end{lem}

\begin{proof}
First assume that $P$ is of type III, so $P \leqs S^m.(\Out(S)\times S_m)$ is a diagonal type group, $S$ is a non-abelian simple group and $k = s^{m-1}$, where $s=|S|$ and  $m \geqs 2$. Now $|P| < s^{m+1}m!$ and \cite[Theorem 4]{BG_mindeg} gives  
$\mu(P) \geqs 2k/3$, so in view of Lemma \ref{l:mindeg}, it suffices to show that
\[
f(s,m) := \frac{2^{\frac{1}{3}s^{m-1}-1}}{s^{m+1}m!}>1.
\]
If we fix $s$, then it is easy to check that $f(s,m)$ is an increasing function of $m$, whence
\[
f(s,m) \geqs f(s,2) = \frac{2^{\frac{1}{3}s-1}}{2s^3} > 1
\]
for all $s \geqs 60$ and the result follows.

Next assume $P \leqs S_m \wr S_t$ is a product type primitive group of degree $k = m^t$, where $m\geqs 5$ and $t\geqs 2$. As explained in the proof of \cite[Lemma 4]{S_dist}, there exists a subset $\Lambda\subseteq\{1,\dots,k\}$ such that $P_{\{\Lambda\}} = 1$ and
	\[
	|\Lambda| = \ell+\sum_{i=1}^t (m-2)^{i-1}m^{t-i},
	\]
	where $\ell=3m-5$ if $t \geqs 3$, otherwise $\ell = 2m-3$.  
If $m$ is odd then $k = m^t$ is odd and thus $\Lambda\in X$. On the other hand, if $m$ is even, then $k/2 = m^t/2$ is even and $|\Lambda|$ is odd, so $\Lambda\in X$ once again. 

Finally, suppose $P$ is a primitive group of type V, so $P = S^{m}.Q$ is a twisted wreath product, where $S$ is a non-abelian simple group and $Q \leqs S_m$ is transitive. Here we can embed $P$ in a product type primitive group $R = S^2 \wr S_m \leqs {\rm Sym}([k])$ and the result follows since we have already shown that $R$ has a regular orbit on $X$ (see \cite[Remark 2(ii)]{LPS} for the containment of $P$ in $R$). 
\end{proof}

Next, we turn to the primitive groups of affine type. The following result extends \cite[Lemma 7]{S_dist}.

\begin{lem}
	\label{l:AGL}
Consider the natural action of $P = \mathrm{AGL}_d(p)$ on $k=p^d$ points, where $d \geqs 1$, $p$ is a prime and $(d,p) \ne (1,2)$. Then either $P$ has a regular orbit on $X$, or $D(P)>2$ and one of the following holds:
	\begin{itemize}\addtolength{\itemsep}{0.2\baselineskip}
		\item[\rm (i)] $p = 2$ and $d\in\{2,3,4,5\}$.
		\item[\rm (ii)] $d = 1$ and $p\in\{3,5,7\}$.
		\item[\rm (iii)] $d = 2$ and $p = 3$.
	\end{itemize}
\end{lem}

\begin{proof}
First assume $p$ is odd, so $k=p^d$ is also odd. Then as previously noted, $P$ has a regular orbit on $X$ if and only if $D(P) = 2$, and the groups with $D(P)>2$ can be read off from \cite[Lemma 7]{S_dist}. For the remainder, we may assume $p=2$. 
	
	As noted in the proof of \cite[Lemma 7]{S_dist}, we have $\mu(P) = 2^{d-1}$ and thus
\[
|P|<2^{d^2+d-1}<2^{\mu(P)/2-1}
\]
for $d \geqs 9$. Therefore, Lemma \ref{l:mindeg} implies that $P$ has a regular orbit on $X$ if $d \geqs 9$. For $d = 6,7,8$, we can use {\sc Magma} to construct $P$ as a permutation group on $[k]$ and by random search we can find a subset $\Lambda \in X$ with $P_{\{\Lambda\}} = 1$ and $|\Lambda| = 16, 16, 17$, respectively. Finally, if  $d \in \{3,4,5\}$ then $D(P)>2$ by \cite[Lemma 7]{S_dist}, and similarly $D(P)=4$ when $d=2$. 
\end{proof}

We can now complete the analysis of primitive groups of affine type.

\begin{lem}
	\label{l:affine}
Let $P \leqs {\rm AGL}_d(p)$ be a primitive affine group of degree $k=p^d$, where $d \geqs 1$, $p$ is a prime and $D(P)=2$. Then $P$ has no regular orbit on $X$ if and only if one of the following holds:
\begin{itemize}\addtolength{\itemsep}{0.2\baselineskip}
\item[\rm (i)] $k=2$ and $P=S_2$.
\item[\rm (ii)] $k=16$ and $P = 2^4{:}\mathrm{O}_4^-(2)$.
\end{itemize}
\end{lem}

\begin{proof}
We may assume $p=2$ and $d \geqs 2$. If $d \geqs 6$ then Lemma \ref{l:AGL} implies that $\mathrm{AGL}_d(2)$, and hence $P$, has a regular orbit on $X$. Therefore, we may assume $d \in \{2,3,4,5\}$ and $P < {\rm AGL}_d(2)$.

If $d=5$ then $P = 2^5{:}31$ or $2^5{:}(31{:}5)$ and in both cases we can use {\sc Magma} to find a subset of $[32]$ of size $5$ with trivial setwise stabiliser in $P$. Next assume $d=4$. We can use the {\sc Magma} database of primitive groups to construct each possibility for $P$; there are $19$ such groups, up to permutation isomorphism, and $15$ with $D(P)=2$. In all but one of these cases, we can use random search to find a set in $X$ with trivial setwise stabiliser in $P$. The exception is the group $P = 2^4{:}{\rm O}_4^{-}(2)$ recorded in case (ii). Here $k=16$, $D(P)=2$ and every subset of $[16]$ with trivial setwise stabiliser has size $8$, so this is a genuine exception. Finally, if $d \in \{2,3\}$ then $P = 2^3{:}7$ is the only group with $D(P)=2$ and it is easy to check that there is a subset of size $3$ with trivial setwise stabiliser in $P$. 
\end{proof}

Finally, we deal with the case where $P$ is an almost simple primitive group. 

\begin{lem}\label{l:as}
Let $P \leqs S_k$ be an almost simple primitive group with $D(P) = 2$. Then $P$ has a regular orbit on $X$.
\end{lem}

\begin{proof}
Let $R$ be the socle of $P$ and write $Q$ for the stabiliser of a point in $\Omega = [k]$. For the convenience of the reader, we divide the proof into several cases.

\vs

\noindent \emph{Case 1. $R$ is an alternating group.}

\vs

First assume $R = A_m$ is an alternating group. The case $m=6$ can be handled using {\sc Magma}, so we may assume $m \ne 6$ and thus $P = S_m$ or $A_m$. There are now three cases to consider, according to the action of $Q$ on $\{1, \ldots, m\}$.

First assume $Q$ is intransitive, in which case we may identify $\O$ with the set of $t$-element subsets of $\{1, \ldots, m\}$ for some $2 \leqs t < m/2$ (note that $t \ne 1$ since we are assuming $D(P)=2$). Suppose $m \geqs t+5$. If $t \geqs 4$ then the proof of \cite[Lemma 9]{S_dist} shows that there exists a subset $\Lambda$ of $\O$ such that  $P_{\{\Lambda\}} = 1$ and $|\Lambda| = 2(m-t+1) <k/2$, so $P$ has a regular orbit on $X$. Similarly, if $t=2,3$ then we can take $|\Lambda| = m-t+1$. On the other hand, if $m<t+5$ then $(m,t) = (7,3)$ or $(5,2)$, noting that $P=A_5$ in the latter case (since $D(P)>2$ when $P = S_5$). Here it is straightforward to check that $\O$ has a subset of size $4$ with trivial setwise stabiliser in $P$. 

Next assume $Q$ acts primitively on $\{1, \ldots, m\}$. By the main theorem of \cite{PS_order} we have $|Q|<4^m$ and \cite[Theorem 4]{BG_mindeg} implies that $\mu(P) \geqs 2k/3$. By combining these bounds and applying Lemma \ref{l:mindeg}, noting that $k = |P:Q|$, we deduce that $P$ has a regular orbit on $X$ if $m \geqs 14$. For $7 \leqs m \leqs 13$ we can use {\sc Magma} to determine the possibilities for $Q$ and it is routine to check that $|P| < 2^{k/3-1}$ unless $(k,P,Q) = (15,A_8,{\rm AGL}_3(2))$ or $(15,A_7, {\rm L}_3(2))$. In the former case we have $D(P)>2$, while $D(P)=2$ in the latter and the result follows since $k=15$ is odd. Finally, if $m=5$ then one can check that $D(P)>2$.
	
To complete the argument when $R = A_m$, we may assume $Q$ acts imprimitively on $\{1, \ldots, m\}$. Here we may identify $\O$ with the set of partitions of $\{1, \ldots, m\}$ into $b$ sets of size $a$, where $a,b \geqs 2$ and $m=ab \geqs 8$ (we have already considered the case $m=6$). Therefore, $k = m!/(a!^bb!)$ and the main theorem of \cite{GM_mindeg} gives $\mu(P) \geqs k/2$, so it suffices to show that $|P| < 2^{k/4-1}$ (see Lemma \ref{l:mindeg}). First assume $m=8$. If $(a,b) = (2,4)$ then it is easy to check that $|P| < 2^{k/4-1}$. On the other hand, if $(a,b) = (4,2)$ then $D(P)=2$ and the result follows since $k=35$ is odd. Now assume $m \geqs 9$. Here $m! < 2^{m^2/4-1}$ and so it suffices to show that $k \geqs m^2$. This is clear if $a \geqs 3$ since $k \geqs \binom{m}{3} > m^2$ for $m \geqs 9$. Finally, for $a=2$ it remains to show that
\[
f(b) := \frac{(2b)!}{b!2^b4b^2}>1
\]
for all $b \geqs 5$. It is easy to verify that this is an increasing function, so $f(b) \geqs f(5)>1$ and the result follows.

\vs

\noindent \emph{Case 2. $R$ is a sporadic group.}

\vs

Next assume $R$ is a sporadic simple group. First observe that $D(P)>2$ if $R = {\rm M}_{22}$ and $k=22$, so this case does not arise and thus $\mu(P) \geqs 2k/3$ by \cite[Theorem 4]{BG_mindeg}. Therefore, it suffices to show that $|P| < 2^{k/3-1}$. Let $\ell$ be the minimal index of a core-free subgroup of $P$, which can be read off from \cite{Wilson}. If $P$ is not a Mathieu group, then it is straightforward to show that $|P| < 2^{\ell/3-1}$ and the result follows. On the other hand, if $P$ is a Mathieu group then the cases with $D(P)>2$ are determined in \cite[Lemma 12]{S_dist}; by excluding these groups, it is easy to check that $|P| < 2^{k/3-1}$ as required.

\vs

\noindent \emph{Case 3. $R$ is a group of Lie type.}

\vs

For the remainder, we may assume $R$ is a simple group of Lie type over $\mathbb{F}_q$, where $q=p^f$ and $p$ is a prime. As before, let $\ell$ be the minimal index of a core-free subgroup of $P$ and note that $\ell$ is recorded in \cite[Table 4]{GMPS_deg}. 

First assume $R$ is an exceptional group of Lie type. By \cite[Proposition 2]{LS_bound} and \cite[Theorem 4]{BG_mindeg}, we have $|P| < k^5$ and $\mu(P) \geqs 2k/3$, so it suffices to show that $k^5 < 2^{k/3-1}$. The latter bound holds for $k \geqs 104$ and by inspecting \cite[Table 4]{GMPS_deg} we reduce to the case where $R = {}^2B_2(8)$ and $Q$ is a Borel subgroup of $P$. Here $k = 65$ and $|P| \leqs 3|R|<2^{k/3-1}$, so once again Lemma \ref{l:mindeg} implies that $P$ has a regular orbit on $X$.
	
	Finally, let us assume $R$ is a classical group. Due to isomorphisms between some of the low dimensional groups, we may assume $R$ is one of the following:
	\[
	{\rm L}_n(q), \, n \geqs 2; \; {\rm U}_n(q), \, n \geqs 3; \; {\rm PSp}_n(q), \, n \geqs 4; \; {\rm P\O}_n^{\e}(q), \, n \geqs 7.
	\]
	We may also assume that $R$ is not isomorphic to an alternating group.
	By \cite[Corollary 1]{GM_mindeg} we have $\mu(P) \geqs 3k/7$. 
	
First assume $n \geqs 4$ and observe that $|P| < q^{n^2}$.  By carefully inspecting \cite[Table 4]{GMPS_deg} we see that $\ell > q^{n-2}$ and thus $\mu(P) > \frac{3}{7}q^{n-2}$. Now, if $n \geqs 12$ or $q \geqs 19$ then
\[
q^{n^2} < 2^{\frac{3}{14}q^{n-2}-1}
\]
and we deduce that $|P| < 2^{\mu(P)/2-1}$, which implies that $P$ has a regular orbit on $X$ (in fact, if $q \geqs 3$ then the same bound holds for all $n \geqs 8$). This leaves us with finitely many groups to consider. In each of these remaining cases, it is routine to check that $|P| < 2^{3\ell/14-1}$ with the exception of the following possibilities for $(R,k)$: 
\[
({\rm L}_4(5),156), \, ({\rm L}_4(3),40), \, ({\rm U}_4(3),112), \, ({\rm U}_4(2),40), \, ({\rm U}_4(2),36),\]
\[
({\rm Sp}_6(2),36), \, (\O_8^{+}(2),120), \, (\O_8^{-}(2),136).
\]
(Here we also exclude the relevant groups with $D(P)>2$, as recorded in \cite[Lemma 12]{S_dist}, together with the groups where $D(P)=2$ and $k$ is odd.) In each of these cases, we can use {\sc Magma} to construct $P$ as a permutation group on $[k]$ and then find a subset in $X$ by random search with trivial setwise stabiliser in $P$.
	
Finally, let us assume $n \in \{2,3\}$. If $R = {\rm L}_3(q)$ then $\ell = q^2+q+1$ and 
\[
|P| \leqs 2q^3(q^2-1)(q^3-1)\log q < 2^{3\ell/14-1}
\]
for $q \geqs 13$. The remaining groups with $q<13$ can be dealt with using {\sc Magma} as above. Similarly, if $R = {\rm U}_3(q)$ then the problem is quickly reduced to the groups with $q \leqs 5$, each of which can be handled in the usual fashion with the aid of {\sc Magma}. Finally, suppose $R = {\rm L}_2(q)$. If $q \geqs 113$ then $\ell=q+1$ and one can check that
\[
|P| \leqs q(q^2-1)\log q < 2^{3\ell/14-1},
\]
which implies that $P$ has a regular orbit on $X$. The remaining groups with $q<113$ can be handled using {\sc Magma}; either the bound $|P| < 2^{3k/14-1}$ is satisfied and we conclude via Lemma \ref{l:mindeg}, or we construct $P$ as a permutation group on $[k]$ and then use random search to find a set in $X$ with trivial setwise stabiliser in $P$. 
\end{proof}

This completes the proof of Proposition \ref{p:ddagger} and we are now in a position to classify the primitive groups $P$ with $t_{D(P)} = |P|/D(P)!$.

\begin{prop}
	\label{p:cond_5}
Let $P \leqs S_k$ be a primitive group. Then $t_{D(P)} = |P|/D(P)!$ if and only if $P = S_k$, or $(k,P,D(P)) = (6,A_5,3)$, $(9,\PGammaL_2(8),3)$ or $(8,\mathrm{AGL}_3(2),4)$.
\end{prop}

\begin{proof}
Set $D = D(P)$. As previously noted, if $P = S_k$ then $D = k$ and $t_D = |P|/D!$, whereas $t_D > |P|/D!$ if $P = A_k$. Now assume $P \ne A_k, S_k$ and recall that $D \leqs 4$ (see Theorem \ref{t:prim}). If $D=2$ then Corollary \ref{c:unique_power} applies, so we may assume $D \in \{3,4\}$. By Lemma \ref{l:unique}, $t_D = |P|/D!$ if and only if $P$ has a unique regular orbit on $\mathcal{P}_D([k])$ and these groups can be read off by inspecting parts (a) and (c) in \cite[Lemma 1]{D_dist}. 
\end{proof}

Finally, observe that Theorem \ref{t:new} now follows by combining Corollary \ref{c:unique} and Proposition \ref{p:cond_5}.

\section{Saxl graphs}\label{s:Saxl}

In this section we use Theorem \ref{thm:r(W)_cal} to study the valency and connectedness properties of the Saxl graphs of base-two product type primitive groups. 

Let $G \leqs {\rm Sym}(\O)$ be a finite transitive permutation group with $b(G) = 2$ and point stabiliser $H$. Recall that the vertices of the \emph{Saxl graph} of $G$, denoted $\Sigma(G)$, are labelled by the elements of $\O$, with $\a$ and $\b$ joined by an edge if and only if $\{\a,\b\}$ is a base for $G$. Then $\Sigma(G)$ is vertex-transitive with valency ${\rm val}(G) = r|H|$, where $r = r(G)$ is the number of regular suborbits of $G$. In addition, it is easy to see that  $\Sigma(G)$ is connected if $G$ is primitive. This concept was initially introduced and studied by Burness and Giudici in \cite{BG_Saxl}.

\subsection{Valency}\label{ss:val}

Recall that a graph is \emph{Eulerian} if it contains an Eulerian cycle, which is a cycle that uses each edge exactly once. A celebrated theorem of Euler asserts that a connected graph is Eulerian if and only if the degree of every vertex is even. In particular, the Saxl graph of a base-two primitive group $G$ is Eulerian if and only if ${\rm val}(G)$ is even.

A complete classification of the finite primitive groups with an Eulerian Saxl graph remains out of reach and some genuine exceptions have been identified. For example, if $G = {\rm M}_{23}$ and $H = 23{:}11$, then the action of $G$ on $G/H$ is primitive with $b(G)=2$ and the corresponding Saxl graph is non-Eulerian (indeed, we compute $r=159$, so ${\rm val}(G) = r|H|$ is odd). The problem for almost simple primitive groups is studied in \cite[Proposition 3.2]{BG_Saxl} and subsequently extended in \cite[Theorem 4]{CH_val}. 

As another application of Theorem \ref{thm:r(W)_cal}, the following result establishes part (i) of Corollary \ref{cor:Saxl}. It can be viewed as an extension of \cite[Proposition 3.4]{BG_Saxl}.

\begin{prop}\label{p:euler}
Let $G=L\wr P$ be a base-two product type primitive group acting on $\O = \Gamma^k$. Then $\Sigma(G)$ is Eulerian.
\end{prop}

\begin{proof}
By Theorem \ref{thm:r(W)_cal} we have
	\begin{equation}\label{e:thm4}
	\sum_{m = D(P)}^k m!{r(L)\choose m}t_m = r(G)|P|
	\end{equation}
	and each summand on the left hand side of this equality is even since $D(P)\geqs 2$. Therefore, at least one of $r(G)$ or $|P|$ is even, so ${\rm val}(G) = r(G)|H|$ is even and the result follows.
\end{proof}

The finite transitive groups $G$ such that $\Sigma(G)$ has prime valency are determined in \cite[Proposition 3.1]{BG_Saxl}. For almost simple primitive groups, this is extended in \cite[Theorem 3]{CH_val}, which classifies all the groups of this form with the property that ${\rm val}(G)$ is a prime power (in each case, ${\rm val}(G)$ is a $2$-power). 

In our next result, which gives part (ii) of Corollary \ref{cor:Saxl},  
we extend the analysis to product type primitive groups of the form $G = L\wr P$. 

\begin{prop}
	\label{p:prime-power}
	Let $G = L\wr P$ be a base-two product type primitive group acting on $\O = \Gamma^k$ with point stabiliser $J \wr P$. Then ${\rm val}(G)$ is a prime power if and only if $L = {\rm M}_{10}$, $J = SD_{16}$, $P$ is a $2$-group and $t_2 \geqs 1$ is a $2$-power.
\end{prop}

\begin{proof}
Suppose ${\rm val}(G) = p^a$ with $p$ a prime. Then $J$ and $P$ are $p$-groups, so $J$ is soluble and thus $L$ is almost simple. By \cite[Proposition 6.3]{CH_val}, the possibilities for $(L,J)$ are recorded in \cite[Table 3]{CH_val} and we deduce that $p=2$. By Theorem \ref{thm:r(W)_cal}, we see that \eqref{e:thm4} holds and thus
\[ 
\sum_{m=D(P)}^km!{r(L) \choose m}t_m
\]
is a $2$-power. This observation immediately implies that $D(P) = 2$, otherwise each summand is divisible by $3$. In addition, $r(L) \equiv 2\imod{3}$ because the binomial coefficient $\binom{r(L)}{2}$ must be indivisible by $3$. By inspecting \cite[Table 3]{CH_val}, we now consider each possibility for $(L,J)$ in turn. 
	
First assume $L = \LL_2(q)$ and $J = D_{q-1}$, where $q \geqs 17$ is a Fermat prime. Then $q \equiv 1\pmod 4$ and so $r(L) = (q+7)/4$ by the proof of \cite[Lemma 4.7]{B_sol}. If we write $q = 2^m+1$ then $r(L) = 2^{m-2}+2$ and thus $r(L) \not\equiv 2\imod{3}$, so this case does not arise. Now suppose $L = \LL_2(q)$ and $J = D_{q+1}$ with $q = 2^m-1 \geqs 31$ a Mersenne prime. Here $q\equiv 3\pmod 4$ and $r(L) = (q-3)/4$ by the proof of \cite[Lemma 7.9]{BH_gen}. Therefore $r(L) = 2^{m-2}-1$ and once again we deduce that $r(L) \not\equiv 2\imod{3}$.
	
Finally, let us turn to the remaining cases in \cite[Table 3]{CH_val}. If $L = \PGL_2(q)$ and $J = D_{2(q-1)}$ with $q\geqs 17$ a Fermat prime, then $r(L) = 1$ (see Table \ref{tab:r(L)}) and thus $b(G) \geqs 3$. The case where $L = \PGL_2(q)$ and $J =  D_{2(q+1)}$ with $q \geqs 7$ a Mersenne prime can be immediately excluded since $b(L)=3$ by \cite[Theorem 2]{B_sol}. The handful of remaining possibilities can be checked using {\sc Magma} \cite{Magma}, implementing the approach presented in Section \ref{sss:reg1} to compute $r(L)$. In this way, we find that $r(L) \equiv 2\imod{3}$ if and only if $L = {\rm M}_{10}$ and $J = SD_{16}$, where we observe that $r(L)=2$.
	
We conclude that $L = {\rm M}_{10}$ with $J = SD_{16}$ is the only possibility. Here Theorem \ref{thm:r(W)_cal} implies that $|P|r(G) = 2t_2$ and thus $t_2$ is a $2$-power (note that $t_2=1$ if and only if $P=S_2$). This completes the proof of the proposition.
\end{proof}

\begin{rem}\label{r:t2}
There are genuine examples in Proposition \ref{p:prime-power}, where $P \leqs S_k$ is a transitive $2$-group and $t_2$ is a $2$-power. For example, we can take $(k,P) = (2,S_2)$ or $(4,C_2 \times C_2)$, where $t_2 = 1$  or $4$, respectively. By inspecting the {\sc Magma} database of transitive groups, we find that there are $165$ groups of degree $k \leqs 16$ with the desired property. For example,  $t_2$ is a $2$-power for $156$ of the $1427$ transitive $2$-groups of degree $16$.   
\end{rem}

\subsection{Connectedness}\label{ss:conn}

Let $G \leqs {\rm Sym}(\O)$ be a finite primitive group with $b(G) = 2$ and recall that the Saxl graph $\Sigma(G)$ is connected in this situation, so it is natural to consider its diameter. Given $\a \in \O$, let $\Sigma_G(\alpha)$ (or just $\Sigma(\a)$ if the corresponding group $G$ is clear from the context) be the set of neighbours of $\alpha$ in $\Sigma(G)$, so 
\[
\Sigma(\a) = \{ \b \in \O \, :\, G_{\a} \cap G_{\b} = 1\} 
\]
is the union of the regular $G_{\a}$-orbits on $\O$.

One of the main open problems concerning the connectedness of $\Sigma(G)$ is the following conjecture of Burness and Giudici (see \cite[Conjecture 4.5]{BG_Saxl}).

\begin{con}\label{con:BG}
Let $G \leqs {\rm Sym}(\O)$ be a finite primitive group with $b(G)=2$. Then any two vertices in $\Sigma(G)$ have a common neighbour.
\end{con}

In other words, the conjecture asserts that the following property is satisfied:
\begin{equation}\label{e:star}
\mbox{\emph{$\Sigma(\a) \cap \Sigma(\beta)$ is non-empty for all $\a,\b \in \O$.}} \tag{$\star$}
\end{equation}
In particular, this implies that $\Sigma(G)$ has diameter at most $2$ for every finite primitive group $G$ with $b(G)=2$ (this weaker assertion is stated as \cite[Conjecture 4.4]{BG_Saxl}).

\begin{rem}\label{r:saxl}
Let $G \leqs {\rm Sym}(\O)$ be a finite transitive group and let $\mathcal{P}(G) = {\rm val}(G)/|\O|$ be the probability that a random pair of points in $\O$ form a base for $G$. If $\mathcal{P}(G)>1/2$ then ${\rm val}(G) > |\O|/2$ and thus the neighbourhoods of any two points in $\O$ must have a non-empty intersection. Therefore, \eqref{e:star} holds if $\mathcal{P}(G)>1/2$ (this is a special case of \cite[Lemma 3.6(i)]{BG_Saxl}).
\end{rem}

As we highlighted in Section \ref{s:intro}, Conjecture \ref{con:BG} has been verified in various special cases. For example, every primitive group of degree at most $4095$ has the desired property (this has been  checked computationally, using the {\sc Magma} database of primitive groups). In addition, it has been verified for all almost simple primitive groups with socle ${\rm L}_2(q)$ or with soluble point stabilisers (see \cite{ChenDu2020Saxl} and \cite{BH_Saxl}). For the almost simple groups with socle an alternating group $A_n$, the proof is reduced in \cite[Section 5]{BG_Saxl} to the case where each point stabiliser acts imprimitively on $\{1, \ldots, n\}$. Computational methods are used in \cite[Section 6]{BG_Saxl} to verify the conjecture for many almost simple sporadic groups and there are asymptotic results of Fawcett (see \cite[Theorem 1.4]{Faw2} and \cite[Theorem 1.5]{Faw1}), which give the desired conclusion for almost all sufficiently large primitive groups of diagonal and twisted wreath type. Recent work of Lee and Popiel \cite{LP} shows that the conjecture holds for many affine groups of the form $G = VH$, where $H$ is a quasisimple sporadic group. 

This demonstrates that there is a growing body of evidence pointing towards the veracity  of Conjecture \ref{con:BG}. However, there are very few (if any) results in the literature for product type primitive groups.

Let $L \leqs {\rm Sym}(\Gamma)$ be a finite primitive group (of any type), let $P \leqs S_k$ be transitive with $k \geqs 2$ and consider the product action of $G = L\wr P$ on $\O = \Gamma^k$. If $L$ is non-regular on $\Gamma$, then $G \leqs {\rm Sym}(\O)$ is primitive and it is clear that $b(G)=2$ only if $b(L)=2$ for the action of $L$ on $\Gamma$. The following result reveals the relationship between property \eqref{e:star} for $G$ and $L$ in this setting.

\begin{lem}\label{l:BG_D(P)}
Let $G = L\wr P$ be a base-two primitive group in its product action on $\O = \Gamma^k$ and assume $G$ satisfies \eqref{e:star}. Then $\Sigma_L(\alpha)$ meets at least $D(P)$ regular $L_{\beta}$-orbits for all $\a,\b \in \Gamma$.
\end{lem}

\begin{proof}
Suppose there exists $\alpha,\beta\in\Gamma$ such that the given condition on $\Sigma(L)$ does not hold. We claim that the elements $\rho = (\alpha,\dots,\alpha)$ and $\s = (\beta,\dots,\beta)$ in $\O$ have no common neighbour in $\Sigma(G)$, which implies that $G$ does not satisfy \eqref{e:star}. 

Seeking a contradiction, suppose $(\gamma_1, \ldots, \gamma_k) \in \Sigma_G(\rho) \cap \Sigma_G(\sigma)$. For each $i$, Lemma \ref{l:BG_l:2.8} implies that $\{\a,\gamma_i\}$ and $\{\b,\gamma_i\}$ are bases for $L$, so $\gamma_i$ is contained in both $\Sigma_L(\a)$ and a regular $L_{\b}$-orbit. In addition, the lemma implies that the pairs $(\b,\gamma_i)$ for $1 \leqs i \leqs k$ represent at least $D(P)$ distinct $L$-orbits on $\Gamma^2$, whence $\{\gamma_1, \ldots, \gamma_k\}$ meets at least $D(P)$ distinct regular $L_{\b}$-orbits. But this means that $\Sigma_L(\alpha)$ meets at least $D(P)$ regular $L_{\beta}$-orbits and we have reached a contradiction.
\end{proof}

In order to state our next result, we define the following condition for a base-two group $G \leqs {\rm Sym}(\Omega)$:
\begin{equation}\label{e:stars}
\mbox{\emph{$\Sigma(\a)$ meets every regular $G_{\b}$-orbit for all   $\a,\b \in \O$.}} \tag{$\star\star$}
\end{equation}
Note that if this holds, then $|\Sigma(\a) \cap \Sigma(\b)| \geqs r(G)$ for all $\a,\b \in \O$, so it can be viewed as a stronger form of the condition in \eqref{e:star}. 

As a special case of Lemma \ref{l:BG_D(P)}, we obtain the following corollary.

\begin{cor}\label{c:BG_stars}
Let $L \leqs {\rm Sym}(\Gamma)$ be a base-two primitive group and consider the primitive product action of $G = L\wr S_{r(L)}$ on $\O = \Gamma^k$. If $G$ satisfies \eqref{e:star}, then $L$ satisfies \eqref{e:stars}.
\end{cor}

This observation leads us to propose the following strengthening of Conjecture \ref{con:BG}, which is stated as Conjecture \ref{conjecture:SBG} in Section \ref{s:intro}.

\begin{con}
	\label{conj:BG_stronger}
Let $G \leqs {\rm Sym}(\O)$ be a finite primitive permutation group with $b(G)=2$. Then \eqref{e:stars} holds.
\end{con}

As noted above, this asserts that any two vertices in $\Sigma(G)$ have at least $r(G)$ common neighbours, and it coincides with Conjecture \ref{con:BG} when $r(G)=1$. In fact, by taking $G = L \wr S_{r(L)}$ as in Corollary \ref{c:BG_stars}, where $L \leqs {\rm Sym}(\Gamma)$ is an arbitrary primitive group with $b(L)=2$, we deduce the following result.

\begin{prop}\label{p:equiv}
Conjectures \ref{con:BG} and \ref{conj:BG_stronger} are equivalent.
\end{prop}

\begin{rem}\label{r:evidence}
Let $G \leqs {\rm Sym}(\O)$ be a finite primitive group with $b(G)=2$ and point stabiliser $H$. Let $R$ be a set of $(H,H)$ double coset representatives in $G$ and set 
\[
S = \{x \in R \,:\, |HxH| = |H|^2\}.
\]
Then observe that \eqref{e:stars} holds if for all $x \in R$ and all $y \in S$, there exists $z \in HyH$ such that $H \cap H^z = H^x \cap H^z = 1$. This approach for verifying \eqref{e:stars} can be implemented in {\sc Magma} and by using the primitive groups  database, we have checked that it holds for every base-two primitive group of degree $n \leqs 4095$. 
\end{rem}

\begin{rem}\label{r:aff}
Let $G = V{:}H$ be a primitive affine group with $b(G)=2$, where $V = (\mathbb{F}_p)^d$ and $H = G_0 \leqs {\rm GL}_d(p)$ is the stabiliser of the zero vector. Set $r=r(G)$ and let $\Lambda_1, \ldots, \Lambda_r$ be the regular orbits of $H$ on $V$. In \cite[Lemma 2.3]{LP}, Lee and Popiel observe that $G$ satisfies \eqref{e:star} if and only if every non-zero vector in $V$ is of the form $v_1+v_2$, where $v_1$ and $v_2$ are contained in regular $H$-orbits. The same argument shows that \eqref{e:stars} holds if and only if for all non-zero vectors $v \in V$ and all $i \in \{1, \ldots, r\}$ we can write $v = u_i+w_i$ with $u_i \in \Lambda_i$ and $w_i \in \Lambda_j$ for some $j$. 
\end{rem}

The following proposition shows that \eqref{e:stars} holds for an infinite family of primitive groups $G$ with $r(G) \geqs 2$ (by inspecting \cite[Table 8.1]{Low-Dimensional}, we note that $G$ is primitive if $q$ is even, or if $q \geqs 13$). 

\begin{prop}\label{p:SBG_PSL}
Let $G = {\rm L}_2(q)$, where $q \geqs 4$, and consider the action of $G$ on $\O = G/H$, where $H$ is a subgroup of type $\GL_1(q)\wr S_2$. Then $G$ satisfies \eqref{e:stars}.
\end{prop}

\begin{proof}
First assume $q$ is even. As noted in \cite[Example 2.5]{BG_Saxl},
$G$ satisfies \eqref{e:star}, which coincides with \eqref{e:stars} since $r(G)=1$. 

For the remainder, we may assume $q$ is odd, so $H = D_{q-1}$ and \cite[Lemma 4.7]{B_sol} implies that $b(G) = 2$. As in \cite[Section 4.1]{BH_Saxl}, we may identify $\Omega$ with the set of unordered pairs of distinct $1$-dimensional subspaces of $V = \mathbb{F}_q^2$. 
	
Fix a basis $\{e_1,e_2\}$ for $V$ and set $\beta = \{\langle e_1\rangle,\langle e_2\rangle\} \in \O$. Note that if $x\in G_\beta$, then $x$ is the image (modulo scalars)  of an element $A \in {\rm SL}_2(q)$ of the form
	\[
	A=
	\begin{pmatrix}
	\lambda&\\
	&\lambda^{-1}
	\end{pmatrix} \mbox{ or }  
	\begin{pmatrix}
	&\lambda\\
	-\lambda^{-1}&
	\end{pmatrix}
	\] 
	for some $\lambda\in\mathbb{F}_q^\times$, where the matrices are presented with respect to the basis $\{e_1,e_2\}$. We first determine the regular $G_\beta$-orbits (as recorded in Remark \ref{r:LJ}, $G_{\b}$ has $(q+a)/4$ regular orbits, where $a = 7$ if $q \equiv 1 \imod{4}$, otherwise $a=5$). Fix an element $\gamma \in \O$. 
	
	Suppose $\gamma = \{\langle e_1\rangle,\langle e_1+se_2\rangle\}$ for some $s\in\mathbb{F}_q^\times$. Then an easy calculation shows that $\{\beta,\gamma\}$ is a base for $G$ and the regular $G_\beta$-orbit containing $\gamma$ is
	\[
	R_1 = \{\{\langle e_1\rangle,\langle e_1+ce_2\rangle\} \, : \, c\in\mathbb{F}_q^\times\}.
	\]
	Similarly, if $\gamma = \{\langle e_2\rangle,\langle e_1+se_2\rangle\}$ for some $s\in\mathbb{F}_q^\times$, then $\{\beta,\gamma\}$ is also a base for $G$ and 
	\[
	R_2 = \{\{\langle e_2\rangle,\langle e_1+ce_2\rangle\} \, :\, c\in\mathbb{F}_q^\times\}
	\]
	is the regular $G_\beta$-orbit containing $\gamma$.
		
	Now suppose $\gamma = \{\langle e_1+se_2\rangle,\langle e_1+te_2\rangle\}$, where $s,t\in\mathbb{F}_q^\times$ are distinct. By arguing as in the proof of \cite[Lemma 4.3]{BH_Saxl} we calculate that $\{\beta,\gamma\}$ is a base for $G$ if and only if $-st^{-1}$ is a non-square in $\mathbb{F}_q$. So let us assume $-st^{-1}$ is non-square and suppose $A\in\SL_2(q)$ fixes $\beta$. If $A = \mathrm{diag}(\lambda,\lambda^{-1})$, then 
\[
\gamma^A = \{\langle e_1+\lambda^{-2}se_2\rangle,\langle e_1+\lambda^{-2}te_2\rangle\},
\]
whereas 
	\[
	\gamma^A = \{\langle e_1-\lambda^{-2}s^{-1}e_2\rangle,\langle e_1-\lambda^{-2}t^{-1}e_2\rangle\}
	\]
	if $A = \begin{pmatrix}
	&\lambda\\
	-\lambda^{-1}&
	\end{pmatrix}$. 	It follows that the regular $G_\beta$-orbit containing $\gamma$ is
	\[
	R_{s,t}=\{\{\langle e_1+ \lambda^2se_2\rangle,\langle e_1+ \lambda^2te_2\rangle\}:\lambda\in\mathbb{F}_q^\times\}\cup \{\{\langle e_1- \lambda^2s^{-1}e_2\rangle,\langle e_1- \lambda^2t^{-1}e_2\rangle\}:\lambda\in\mathbb{F}_q^\times\}.
	\]
	
	We conclude that 
	\[
	\{ R_1, R_2, R_{s,t} \, : \, \mbox{$s,t \in \mathbb{F}_{q}^{\times}$, $s\ne t$ and $-st^{-1}$ is a non-square in $\mathbb{F}_q$}\}
	\]
	is the set of regular $G_\beta$-orbits. Note that we are not claiming that the orbits denoted $R_{s,t}$ are all distinct. 
	
Fix an element $\a \in \Omega$ with $\a \ne \b$ and let $\Sigma(\a)$ be the set of neighbours of $\a$ in the Saxl graph $\Sigma(G)$. In order to verify \eqref{e:stars}, we need to show that $\Sigma(\a)$ meets every regular $G_{\b}$-orbit. There are several cases to consider.

First assume $\alpha = \{\langle e_1\rangle, \langle e_1+be_2\rangle\}$ for some $b\in\mathbb{F}_q^\times$. It is easy to see that each $\gamma\in R_1\setminus\{\alpha\}$ is contained in $\Sigma(\alpha)$ and we also have $\{\langle e_2\rangle,\langle e_1+be_2\rangle\} \in R_2\cap \Sigma(\alpha)$. Now consider $R_{s,t}$, where $s,t\in\mathbb{F}_q^\times$ are distinct and $-st^{-1}$ is non-square. Note that either $-bs$ or $bt^{-1}$ is a square in $\mathbb{F}_q$, so there are two cases to consider. If $-bs = \mu^2$ is a square then $b = -\mu^2s^{-1}$ and 
\[
\{\langle e_1-\mu^2s^{-1}e_2\rangle,\langle e_1-\mu^2t^{-1}e_2\rangle\}\in R_{s,t}\cap \Sigma(\alpha).
\]
On the other hand, if $bt^{-1} = \mu^2$ then $b = \mu^2t$ and 
\[
\{\langle e_1+\mu^2se_2\rangle,\langle e_1+\mu^2te_2\rangle\}\in R_{s,t}\cap \Sigma(\alpha).
\]
Therefore, $\Sigma(\alpha)$ meets every regular $G_\beta$-orbit as required. 

A very similar argument applies when $\alpha = \{\langle e_2\rangle,\langle e_1+be_2\rangle\}$ for some $b\in\mathbb{F}_q^\times$ and we omit the details. 

Finally, let us assume $\alpha = \{\langle e_1+s_0e_2\rangle,\langle e_1+t_0e_2\rangle\}$, where $s_0,t_0\in\mathbb{F}_q^\times$ are distinct. Note that 
$\{\langle e_i\rangle,\langle e_1+s_0e_2\rangle\}\in R_i \cap \Sigma(\a)$ for $i = 1,2$, so it just remains to show that $\Sigma(\a)$ meets each $R_{s,t}$. As before, either $-s_0s$ or $s_0t^{-1}$ is a square in $\mathbb{F}_q^\times$ and we can repeat the above argument in order to construct an element in $R_{s,t}\cap \Sigma(\alpha)$.
\end{proof}

A weaker form of Conjecture \ref{con:BG} asserts that the Saxl graph of any base-two primitive group has diameter at most $2$ (see \cite[Conjecture 4.4]{BG_Saxl}). Our final result in this section highlights the relationship between this assertion and Conjecture \ref{conj:BG_stronger}. 

\begin{prop}\label{p:conj_equiv}
Let $L \leqs {\rm Sym}(\Gamma)$ be a finite base-two primitive group. If the Saxl graph of every finite base-two primitive group has diameter at most $2$, then either $r(L) = 1$, or $L$ satisfies \eqref{e:stars}.
\end{prop}

\begin{proof}
We may assume $r = r(L) \geqs 2$. Fix 
$\alpha,\beta\in\Gamma$. We need to show that $\Sigma(\alpha)$ meets all $r$  regular $L_{\beta}$-orbits on $\Gamma$. Set $\O = \Gamma^r$, $P=S_r$ and consider $G=L\wr P\leqs \mathrm{Sym}(\Omega)$, which is a base-two product type primitive group by Theorem \ref{thm:BC}. Fix $\rho = (\a, \ldots, \a)$ and $\sigma = (\b, \ldots, \b)$ in $\O$.

First observe that $D(P) =r \geqs 2$, so $\{\rho,\sigma\}$ is not a base for $G$ by Lemma \ref{l:BG_l:2.8}. Therefore, our hypothesis implies that there exists $\tau = (\gamma_1,\ldots,\gamma_r)\in\Omega$ such that both $\{\rho,\tau\}$ and $\{\sigma,\tau\}$ are bases for $G$ (otherwise the distance between $\rho$ and $\s$ in $\Sigma(G)$ is at least $3$). By applying Lemma \ref{l:BG_l:2.8} again, we deduce that each $\gamma_i$ is contained in $\Sigma_L(\a) \cap \Sigma_L(\b)$. Moreover, the $\gamma_i$ are contained in distinct regular $L_{\beta}$-orbits since the only distinguishing partition for the action of $P$ on $[r]$ is the partition into singletons. The result follows.
\end{proof}

\section{General product type groups}\label{s:gen}

Up to now, we have focussed on product type groups of the form $G = L \wr P$. As one might expect, the study of bases in the general setting $G \leqs L \wr P$ is more difficult and this is essentially unchartered territory. In this final section, we take the first steps in this direction by focussing on the special case where $P \leqs G$, which already turns out to be rather challenging. Our main results are Theorems \ref{thm:P=C_2_sol_gen} and \ref{thm:semiprimitive_sol_stab}, which describe the groups with $b(G)=2$ in certain families of product type groups with soluble point stabilisers. Note that Theorem \ref{thm:P=C_2_sol_gen} is stated as Theorem \ref{t:new2} in Section \ref{s:intro}.

\subsection{Preliminaries}

Let us fix the set-up and notation we will work with throughout this section. As before, $G \leqs L \wr P \leqs {\rm Sym}(\O)$ is a product type primitive group on $\O = \Gamma_1 \times \cdots \times \Gamma_k = \Gamma^k$, with socle $T_1 \times \cdots \times T_k = T^k$ and point stabiliser $H$. Here $k \geqs 2$, $P \leqs S_k$ is transitive and $L \leqs {\rm Sym}(\Gamma)$ is a primitive group with socle $T$ and point stabiliser $J$, which is either almost simple or diagonal type. In addition, we may assume $P$ is the permutation group on $[k] = \{1, \ldots, k\}$ induced by the conjugation action of $G$ on the set of $k$ factors of the socle $T^k$. In view of Remark \ref{r:kov}, we may (and will) also assume that $L$ is the group induced by $G$ on $\Gamma_1$. In particular, this means that both $L$ and $P$ are uniquely determined by $G$. As explained in Remark \ref{r:sol}, if $H$ is soluble then $J$ and $P$ are also soluble.

Since the case $G = L \wr P$ has been studied in Sections \ref{s:sol}-\ref{s:Saxl}, we will assume $G < L \wr P$. More importantly, we will also assume that $G$ contains $P$, which means that we adopt the following hypothesis for the remainder of this section (in particular, note that $L \ne T$).

\begin{hyp}\label{h:6}
$G \leqs {\rm Sym}(\O)$ is a product type primitive group with 
$T \wr P < G < L \wr P$.
\end{hyp}

In the setting of Hypothesis \ref{h:6}, we introduce some new notation. Given an element $g = (g_1, \ldots, g_k) \in L^k \cap G$, let $\tau(g)$ be the number of coordinates of $g$ that are contained in $T$ and set 
\begin{equation}\label{e:tau}
\tau(G) =  \max\{\tau(g) \,:\, g \in (L^k \cap G) \setminus T^k\} \in \{0, 1, \ldots, k-1\}.
\end{equation}
For example, if $L = \langle T,x\rangle$ and $G = \langle T^k,(x,\dots,x),P\rangle$, then $\tau(G) = 0$. Note that if $\tau(G) = k-1$ and $G$ satisfies Hypothesis \ref{h:6}, then $|L:T|$ is composite (indeed, if $|L:T|$ is a prime then we get $G = L \wr P$). 

As before, let $r(L)$ and $r(T)$ denote the number of regular suborbits of $L$ and $T$ on $\Gamma$, respectively (recall that $T$ acts transitively on $\Gamma$ since $L$ is  primitive).

We begin with the following result, which gives a sufficient condition for a group satisfying Hypothesis \ref{h:6} to admit a base of size $2$.

\begin{prop}
	\label{prop:|L:T|_general}
Assume Hypothesis \ref{h:6} and let $m \in \{1, \ldots, D(P)\}$ be minimal such that there exists a distinguishing partition $\{\pi_1,\dots,\pi_{D(P)}\}$ for $P$ with $|\bigcup_{i=1}^m\pi_i|>\tau(G)$. Then $b(G)=2$ if 
\begin{itemize}\addtolength{\itemsep}{0.2\baselineskip}
\item[{\rm (i)}] $r(L)\geqs m$; and
\item[{\rm (ii)}] $r(T)\geqs m(|L:T|-1)+D(P)$.
\end{itemize}
\end{prop}

\begin{proof}
Suppose the bounds in (i) and (ii) are satisfied. Set $D = D(P)$ and fix an element $\alpha = (\alpha_0,\dots,\alpha_0)\in\Omega$ for some $\a_0 \in \Gamma$. Since $r(L)\geqs m$, we may choose elements $\gamma_1,\ldots,\gamma_m$ in $\Gamma$ that are contained in distinct regular $L_{\alpha_0}$-orbits. Now each of these $L_{\alpha_0}$-orbits is a union of $|L:T|$ regular $T_{\alpha_0}$-orbits, so the bound in (ii) implies that we can find additional points $\gamma_{m+1}, \ldots, \gamma_D$ in $\Gamma$ such that each element in $\{\gamma_1, \ldots, \gamma_D\}$ is contained in a distinct regular $T_{\a_0}$-orbit. Define $\beta = (\beta_1,\dots,\beta_k) \in \O$, where  
$\beta_j = \gamma_i$ if $j\in\pi_i$. We claim that $\{\a,\b\}$ is a base for $G$.
	
Let $x\in G_{\alpha} \cap G_{\beta}$ and write $x = z\sigma$, where $z = (z_1,\ldots,z_k)\in L^k \cap G$ and $\sigma\in P$. Recall that $\tau(z)$ denotes the number of coordinates $z_i$ that are contained in $T$ and observe that $z \in T^k$ if $\tau(z) > \tau(G)$ (see \eqref{e:tau}). By Lemma \ref{l:BG_l:2.8}, we see that $\{\a,\b\}$ is a base for $T \wr P$, so it suffices to show that $\tau(z) > \tau(G)$.

Fix $j \in \pi_1$ and notice that $z_j\in L_{\alpha_0}$ since $x$ fixes $\a$. Next observe that the $j^\sigma$-th coordinate of $\beta^x$ is $\beta_j^{z_j} = \gamma_1^{z_j}$, which is equal to $\b_{j^{\sigma}} \in \{\gamma_1, \ldots, \gamma_D\}$ since $x$ fixes $\b$. By construction, none of the elements $\gamma_2,\dots,\gamma_D$ are contained in the $L_{\alpha_0}$-orbit of $\gamma_1$, whence $\b_{j^{\sigma}} = \gamma_1$ and thus $z_j \in L_{\alpha_0} \cap L_{\gamma_1} = 1$. In the same way, we deduce that 
$z_j = 1$ for all $j\in \bigcup_{i = 1}^m\pi_i$ and thus $\tau(z) \geqs |\bigcup_{i = 1}^m\pi_i| >\tau(G)$. The result follows.
\end{proof}

\begin{rem}
Note that Proposition \ref{prop:|L:T|_general} can be applied when $G = L \wr P$ and $L \ne T$. Here $\tau(G) = k-1$, so $m = D(P)$ and the proposition asserts that
$b(G) = 2$ if $r(L)\geqs D(P)$ and $r(T)\geqs |L:T|$. Here the condition $r(L)\geqs D(P)$ coincides with the one in Theorem \ref{thm:BC}, while the bound $r(T)\geqs |L:T|$ always holds when $b(L)=2$. So in some sense, Proposition \ref{prop:|L:T|_general} can be viewed as a generalisation of Theorem \ref{thm:BC} for base-two groups. However, it is worth noting that there are groups with $b(G)=2$ that do not satisfy the bounds labelled (i) and (ii) in the proposition. For example, the proof of Theorem \ref{thm:P=C_2_sol_gen} shows that there are groups $G$ satisfying Hypothesis \ref{h:6} with $k=m=2$, $r(L)=1$ and $b(G)=2$.
\end{rem}

Let $G$ be a group satisfying Hypothesis \ref{h:6} and note that $b(G)=2$ only if $b(T \wr P)=2$, so we are interested in the groups with $b(T)=2$. Moreover, in view of \cite[Theorem 2.13]{BC}, we may assume $r(T) \geqs D(P)$. Recall that if $H$ is soluble, then $L$ is almost simple with soluble point stabilisers and as a consequence of \cite[Theorem 2]{B_sol} we observe that $b(L) \in \{2,3\}$. So in the case where $G$ has soluble point stabilisers, as in the two main results of this section, there are two cases to consider according to the base size of $L$ on $\Gamma$. To simplify the analysis, we will focus on the groups with $b(L)=2$, which allows us to bring Proposition \ref{prop:|L:T|_general} into play. However, it is worth noting that there are groups of this form with $b(L)=3$ and $b(G)=2$. Indeed, we refer the reader to Example \ref{ex:bL3} at the end of the section for an infinite family of groups $G$ with soluble point stabilisers where we have $b(L)=3$ and $b(G)=2$.

\subsection{Base-two groups with $k=2$}

Our first main result is Theorem \ref{thm:P=C_2_sol_gen} below, which is stated as Theorem \ref{t:new2} in Section \ref{s:intro}. Here we assume $k=b(L)=2$, so $P = S_2$ and $D(P)=2$. If $r(L) \geqs 2$ then $b(L \wr P) = 2$ by Theorem \ref{thm:BC}, so we may as well assume $r(L)=1$. Define $\tau(G)$ as in \eqref{e:tau} and note that $\tau(G) \in \{0,1\}$. If $\tau(G) = 0$, then $m=1$ in Proposition \ref{prop:|L:T|_general} and we deduce that $b(G) = 2$ if $r(T)\geqs |L:T|+1$ (recall that the slightly weaker bound $r(T) \geqs |L:T|$ always holds when $b(L)=2$). On the other hand, if $\tau(G)=1$ then $m=2$ and thus Proposition \ref{prop:|L:T|_general} is not useful when $r(L)=1$. Also recall that $|L:T|$ is composite if $\tau(G)=1$ (otherwise $G = L \wr P$).

\begin{thm}\label{thm:P=C_2_sol_gen}
Assume Hypothesis \ref{h:6}, where $k=b(L)=2$, $H$ is soluble and $J$ is a point stabiliser in $L$. Then $b(G) \leqs 3$, with equality if and only if $|L \wr P :G|=2$ and one of the following holds:
	\begin{itemize}\addtolength{\itemsep}{0.2\baselineskip}
		\item[\rm (i)] $(L,J) = (\mathrm{M}_{10},5{:}4)$ or $(\mathrm{J}_2.2,5^2{:}(4\times S_3))$.
		\item[\rm (ii)] $L = {\rm PGU}_4(3)$ and $J$ is of type $\GU_1(3)\wr S_4$.
	\end{itemize}
\end{thm}

\begin{proof}
Here $P=S_2$ and $D(P)=2$. By Lemma \ref{l:plus1} we have $b(L\wr P) \leqs b(L)+1$ and thus $b(G) \leqs 3$. As explained in Remark \ref{r:sol}, we note that $J$ is soluble and thus $L$ is almost simple. If $r(L)\geqs 2$ then $b(L\wr P) = 2$ by Theorem \ref{thm:BC}, so we may assume $r(L) = 1$ and then inspect the possibilities for $(L,J)$ recorded in Table \ref{tab:r(L)}. As discussed above, Proposition \ref{prop:|L:T|_general} implies that $b(G) = 3$ only if $r(T) = |L:T|$ or $\tau(G) = 1$, so there are two cases to consider.
	
First assume $r(T) = |L:T|$. By inspecting Table \ref{tab:r(L)}, we see that $(L,J) = (\mathrm{M}_{10},5{:}4)$ or $(L,J) = (\mathrm{J}_2.2,5^2{:}(4\times S_3))$, so in both cases we have $|L:T| = 2$ and $|L\wr P:G| = 2$. More precisely, if we write $L = \la T, a \ra$ then we may assume $G = \la T^2, (a,a), P \ra$. Using {\sc Magma}, we can construct $G$ as a permutation group on $\O = \Gamma^2$ with point stabiliser $H$ and we can then find a complete set $R$ of $(H,H)$ double coset representatives. In both cases, it is routine to check that $|HxH|<|H|^2$ for all $x \in R$ and we conclude that $b(G)=3$ as claimed. 
	
For the remainder, let us assume $\tau(G) = 1$ and recall that $|L:T| \geqs 4$ is composite since $G$ is a proper subgroup of $L \wr P$. By inspecting 
Table \ref{tab:r(L)} we deduce that there are four possibilities for $(L,J)$: 
\begin{itemize}\addtolength{\itemsep}{0.2\baselineskip}
		\item[\rm (a)] $L = \UU_3(5){:}S_3$ and $J$ is of type $\GU_1(5)\wr S_3$.
		\item[\rm (b)] $L = \LL_3(4){:}D_{12}$ and $J$ is of type $\GL_1(4^3)$.
		\item[\rm (c)] $L = \POmega_8^+(3){:}2^2$ and $J$ is of type $\mathrm{O}_4^+(3)\wr S_2$.
		\item[\rm (d)] $L = \UU_4(3){:}[4]$ and $J$ is of type $\GU_1(3)\wr S_4$.
		\end{itemize}

In cases (a) and (b), we can use {\sc Magma} to check that $b(G) = 2$ (here we construct $G$ and $H$ as above, and then we use random search to find an element $g \in G$ with $H \cap H^g = 1$). Now let us turn to cases (c) and (d), so $|L:T|=4$ and $|L \wr P :G| \in \{2,4,8\}$. In fact, one can check that the condition $\tau(G) = 1$ forces $|L \wr P :G|=2$.  More precisely, if $L = T{:}\la a,b \ra = T{:}2^2$ then up to permutation isomorphism, we may assume that 
\begin{equation}\label{e:22}
G = \la T^2,(a,a),(b,1),P \ra. 
\end{equation}
Similarly, if $L = {\rm PGU}_4(3) = T{:}\la a \ra$, then we can take
\begin{equation}\label{e:4}
G = \la T^2, (a,a),(a^2,1), P \ra. 
\end{equation}

First assume $L = T{:}\la a,b \ra = T{:}2^2$, so \eqref{e:22} holds and $T = {\rm P\O}_8^{+}(3)$ or ${\rm U}_4(3)$. We claim that $b(G)=2$. To see this, fix $\a_0 \in \Gamma$ and let $\gamma_1, \ldots, \gamma_t$ be representatives of the regular $T_{\alpha_0}$-orbits on $\Gamma$, where $t = r(T)$ (as recorded in Table \ref{tab:r(L)}, we have $t = 12$ if $T = {\rm P\O}_8^{+}(3)$ and $t=11$ for $T = {\rm U}_4(3)$). We may assume that $L_{\a_0} \cap L_{\gamma_1}=1$ and $|L_{\alpha_0} \cap L_{\gamma_2}| = 2$ (the existence of $\gamma_2$ can be checked using {\sc Magma}). Now $T_{\a_0} \cap T_{\gamma_2} = 1$, so without loss of generality we may assume that $\{\a_0,\gamma_2\}$ is a base for $T{:}\la b\ra$. Suppose $x = (z_1,z_2)\sigma \in G$ fixes $\alpha= (\a_0,\a_0)$ and $\beta= (\gamma_1,\gamma_2)$, where $z_1,z_2 \in L$ and $\sigma \in P$. Then $z_i \in L_{\a_0}$ for $i=1,2$ and we note that $\s=1$ since $\gamma_1$ and $\gamma_2$ are contained in distinct $L_{\a_0}$-orbits. Therefore, $z_1 \in L_{\a_0} \cap L_{\gamma_1}=1$ and thus $z_1 = 1$. From the description of $G$ in \eqref{e:22}, it follows that $z_2 \in T{:}\la b\ra$ fixes $\a_0$ and $\gamma_2$, whence $z_2=1$ and thus $x=1$. This justifies the claim.

Finally, let us assume $L = {\rm PGU}_4(3) = T{:}\la a \ra$, so the structure of $G$ is given in \eqref{e:4}. 
We claim that $b(G)=3$. As before, fix $\a_0 \in \Gamma$ and set $\alpha = (\alpha_0,\alpha_0)$ and $\beta = (\gamma_1,\gamma_2)$, where $\gamma_1$ and $\gamma_2$ are contained in regular $T_{\a_0}$-orbits. It suffices to show that the pointwise stabiliser of $\a$ and $\b$ in $G$ is non-trivial. It will be useful to observe that a given element $(z_1,z_2) \in L^2$ is contained in $G$ if and only if $z_1z_2 \in T{:}\la a^2 \ra$.

First assume $\{\a_0,\gamma_1\}$ is not a base for $L$. Then using {\sc Magma} we see that $|L_{\a_0} \cap L_{\gamma_1}| \in \{2,4\}$, so there is an involution $y \in L$ fixing $\a_0$ and $\gamma_1$. Since every involution in $L$ is contained in $T{:}\la a^2 \ra$, it follows that $(y,1) \in G$ fixes $\a$ and $\b$. An entirely similar argument applies if $\{\a_0,\gamma_2\}$ is not a base for $L$.

Finally, suppose $\{\a_0,\gamma_1\}$ and $\{\a_0,\gamma_2\}$ are both bases for $L$, which means that $\gamma_1$ and $\gamma_2$ are contained in the unique regular $L_{\alpha_0}$-orbit on $\Gamma$. Therefore, there exists $z_1 \in L_{\alpha_0}$ such that $\gamma_1^{z_1} = \gamma_2$ and thus $(z_1,z_1^{-1})\s \in G$ is a non-trivial element fixing $\alpha$ and $\beta$, where $\s = (1,2) \in P$. This completes the proof of the theorem.
\end{proof}

\subsection{Base-two groups with $\tau(G)=0$}

For the remainder of Section \ref{s:gen}, we will continue to assume that Hypothesis \ref{h:6} holds and $b(L)=2$, but we will not impose any conditions on $k$. In exchange, we will focus on the groups with $\tau(G)=0$ (see \eqref{e:tau}), which is a natural restriction on the structure of $G$. In particular, this means that if $x = z\sigma \in G$, where $z = (z_1, \ldots,z_k) \in L^k \cap G$ and $\sigma \in P$, then either $z \in T^k$ or $z_j \in L \setminus T$ for all $j$. 

Our main result in this setting is Theorem \ref{thm:semiprimitive_sol_stab} and the proof will require several preliminary results. We begin with an easy corollary of Proposition \ref{prop:|L:T|_general}.

\begin{cor}
	\label{c:G diag}
Assume Hypothesis \ref{h:6}, where $b(L) = 2$ and $\tau(G) = 0$. If $b(G) \geqs 3$, then
\[
|L:T| \leqs r(T) \leqs |L:T|+D(P)-2.
\]
\end{cor}

\begin{proof}
This follows immediately from Proposition \ref{prop:|L:T|_general}, noting that $m=1$ in the statement of the proposition.
\end{proof}

\begin{rem}
Let us apply Corollary \ref{c:G diag} in the case where $H$ is soluble. Suppose $b(G) \geqs 3$, so $r(T) \leqs |L:T|+D(P)-2$ by the corollary. Now $r(T)\geqs r(L)|L:T|$ and the solubility of $P$ implies that $D(P)\leqs 5$ (see Theorem \ref{t:sol_dist}), whence
\[
(r(L)-1)|L:T| \leqs D(P)-2\leqs 3
\]
and either $r(L) = 1$, or $r(L) = 2$ and $|L:T| \in \{2,3\}$. In particular, the possibilities for $(L,J)$ can be read off from Table \ref{tab:r(L)} and we find that $r(L) = 2$, $|L:T| \in \{2,3\}$ and $r(T) \leqs |L:T|+3$ if and only if one of the following holds:
\begin{itemize}\addtolength{\itemsep}{0.2\baselineskip}
\item[{\rm (a)}] $L = \LL_3(3).2$ and $J$ is of type $\mathrm{O}_3(3)$.
\item[{\rm (b)}] $L = \LL_2(27).3$ and $J$ is of type $\GL_1(27^2)$.
\item[{\rm (c)}] $(L,J) = ({\rm M}_{10}, SD_{16})$.
\end{itemize}
Note that if $P$ is primitive (as in Theorem \ref{thm:semiprimitive_sol_stab}), then $D(P) \leqs 4$ by Theorem \ref{t:prim}, so the corollary implies that $r(T)\leqs |L:T|+2$ and this eliminates cases (a) and (b).
\end{rem}

In order to state our next result, we define the following condition on the group $P \leqs S_k$, with respect to a fixed integer $m$ in the range $D(P) \leqs m \leqs k$:
\begin{equation}\label{e:dag}
\begin{array}{c}
\mbox{\emph{If $\{\pi_1,\dots,\pi_m\}$ is a distinguishing partition for $P$,}}\\ \mbox{\emph{then for all $i$, there exists $\rho\in P$ such that $\pi_i\cap\pi_i^\rho$ is empty.}}
\end{array}
\tag{$\dagger$}
\end{equation}
Note that if there exists a distinguishing partition $\{\pi_1,\dots,\pi_m\}$ for $P$ with $|\pi_i| > k/2$ for some $i$, then $\pi_i \cap \pi_i^\rho$ is non-empty for all $\rho \in P$ and thus $(P,m)$ does not satisfy \eqref{e:dag}. 

\begin{rem}
Recall that $t_m$ is the number of (unordered) distinguishing partitions for $P$ with $m$ non-empty parts. As noted in Section \ref{s:reg}, if $t_m = |P|/m!$ then any two parts in such a partition are in the same $P$-orbit and thus $(P,m)$ satisfies the condition in \eqref{e:dag}.
\end{rem}

\begin{prop}
	\label{prop:diag}
	Assume Hypothesis \ref{h:6}, where $b(L) = 2$ and $\tau(G) = 0$. 
	If $b(G) \geqs 3$, then \eqref{e:dag} holds for all $D(P)\leqs m\leqs \min\{k,r(T)\}$.
\end{prop}

\begin{proof}
This is similar to the proof of Proposition \ref{prop:|L:T|_general}. Suppose there exists a distinguishing partition $\{\pi_1,\dots,\pi_m\}$ for $P$ such that $D(P) \leqs m \leqs r(T)$ and $\pi_1\cap \pi_1^\rho$ is non-empty for all $\rho\in P$. Fix $\a_0 \in \Gamma$. Since $r(T) \geqs m$, we can choose elements $\g_1,\ldots,\g_m$ that are contained in distinct regular $T_{\a_0}$-orbits on $\Gamma$, so each pair $\{\a_0,\g_i\}$ is a base for $T$. In addition, we may assume that $\{\a_0,\g_1\}$ is a base for $L$. Define $\alpha = (\alpha_0,\dots,\alpha_0)$ and $\beta=(\beta_1,\dots,\beta_k)$ as elements of $\Omega$, where $\beta_j = \g_i$ if $j \in\pi_i$. In order to prove the proposition, it suffices to show that $\{\a,\b\}$ is a base for $G$. 

Let $x \in G_{\a} \cap G_{\b}$ and write $x=z\sigma$, where $z = (z_1,\dots,z_k)\in L^k \cap G$ and $\sigma\in P$. Note that each $z_j$ is contained in $L_{\a_0}$. If $z \in T^k$ then $x \in T \wr P$ and thus $x=1$ since we know that $\{\a,\b\}$ is a base for $T \wr P$ by Lemma \ref{l:BG_l:2.8}. Therefore, since $\tau(G)=0$, we may assume $z_j \in L \setminus T$ for all $j$. 

Since $\pi_1\cap \pi_1^{\sigma}$ is non-empty, there exists $j \in \pi_1$ such that $j^{\s} \in \pi_1$. Then $\b_j = \g_1$ and by considering the $j^{\sigma}$-th coordinate of $\b^{x}$ we deduce that $\gamma_1 = \b_{j^{\s}} = \b_j^{z_j} = \g_1^{z_j}$. Therefore, $z_j \in L_{\a_0} \cap L_{\gamma_1} = 1$ and thus $z_j = 1$, which contradicts the fact that $z_j \in L \setminus T$. The result follows.
\end{proof}

Notice that $P$ is primitive in the statement of Theorem \ref{thm:semiprimitive_sol_stab}. Therefore, in view of Theorem \ref{t:prim}, we have a special interest in the case $D(P)=2$, which means that $P$ has a regular orbit on the power set of $[k] = \{1, \ldots, k\}$. This leads us naturally to consider the following condition, which coincides with \eqref{e:dag} when $m = 2$:
\begin{equation}\label{e:ddag}
\begin{array}{c}
\mbox{\emph{If the setwise stabiliser of $\Lambda \subseteq [k]$ in $P$ is trivial,}} \\
\mbox{\emph{then $\Lambda^\sigma = [k] \setminus \Lambda$ for some $\sigma \in P$.}}
\end{array}
\tag{$\ddagger$}
\end{equation}

Clearly, if $D(P)=2$ then \eqref{e:ddag} holds only if $k$ is even and every subset of $[k]$ with trivial setwise stabiliser in $P$ has size $k/2$. In other words, \eqref{e:ddag} holds only if $P$ has no regular orbit on the set $X$ defined in \eqref{e:X}. Therefore, if $P$ is primitive and $D(P)=2$, then Proposition \ref{p:ddagger} implies that \eqref{e:ddag} holds if and only if $(k,P) = (2,S_2)$ or $(16, 2^4{:}\mathrm{O}_4^-(2))$. 
We will return to this observation in the proof of Theorem \ref{thm:semiprimitive_sol_stab} below.

The following result is an immediate corollary of Proposition \ref{prop:diag}.

\begin{cor}\label{c:tau0,D(P)=2}
Assume Hypothesis \ref{h:6}, where $b(L) = D(P) = 2$ and $\tau(G) = 0$. If $b(G) \geqs 3$, then \eqref{e:ddag} holds.
\end{cor}

The final ingredient for the proof of Theorem \ref{thm:semiprimitive_sol_stab} is provided by the following lemma.

\begin{lem}
	\label{l:|L:T|=2}
Assume Hypothesis \ref{h:6}, where $b(L) = |L:T|=2$, $\tau(G)=0$ and $P = S_{r(T)}$. Then $b(G) \geqs 3$.
\end{lem}

\begin{proof}
Set $k=r(T)$ and fix $\alpha_0\in\Gamma$. Recall that an element $z = (z_1,\dots,z_k)\in L^k$ is contained in $G$ if and only if $z\in T^k$ or $z_i\in L\setminus T$ for all $i$.
In view of Lemma \ref{l:BG_l:2.8}, it suffices to show that $\alpha = (\alpha_0,\dots,\alpha_0)$ and $\beta = (\gamma_1,\dots,\gamma_k)$ do not form a base for $G$, where the $\gamma_i$ are contained in distinct regular $T_{\alpha_0}$-orbits. 

Since $|L:T| = 2$, each regular $L_{\alpha_0}$-orbit is a union of two regular $T_{\alpha_0}$-orbits. This allows us to define $r = r(L)$ distinct pairs $\{s,t\} \subseteq \{1,\dots,k\}$, where $\{s,t\}$ is a pair if and only if $\gamma_s$ and $\gamma_t$ are in the same regular $L_{\alpha_0}$-orbit. Let $\{s_1,t_1\},\dots,\{s_r,t_r\}$ be the pairs arising in this way. For each $i\in\{1,\dots,r\}$, there exist $z_{s_i},z_{t_i}\in L_{\alpha_0}$ such that $\gamma_{s_i}^{z_{s_i}} = \gamma_{t_i}$ and $\gamma_{t_i}^{z_{t_i}} = \gamma_{s_i}$. In addition, if $\ell\notin\{s_1,t_1,\dots,s_r,t_r\}$ then there exists $1\ne z_\ell \in L_{\alpha_0}$ such that $\gamma_\ell^{z_\ell} = \gamma_\ell$. By construction, all of the elements $z_{s_i}$, $z_{t_i}$ and $z_\ell$ are contained in $L\setminus T$. Therefore, if we define $z = (z_1,\dots,z_k)\in L^k$, then $z\in G$. Finally, we note that $1\ne z\sigma\in G_\alpha\cap G_\beta$, where $\sigma = (s_1,t_1)\cdots(s_r,t_r) \in P$, and we conclude that $\{\alpha,\beta\}$ is not a base for $G$.
\end{proof}

\begin{thm}\label{thm:semiprimitive_sol_stab}
Assume Hypothesis \ref{h:6}, where $b(L)=2$, $P$ is primitive, $\tau(G) = 0$ and $H$ is soluble. Then $b(G) \leqs 3$, with equality if and only if $(L,J)$ is one of the cases in Table \ref{tab:primitive} and either $r(T)<D(P)$, or $P = S_k$, $k \in \{2,3,4\}$ and $r(T) = D(P)=k$.
\end{thm}

\begin{table}
\[
\begin{array}{clll} \hline
r(T) & L & \mbox{Type of $J$}  & \mbox{Conditions} \\ \hline
2 & {\rm M}_{10} & 5{:}4  & \\
& \mathrm{J}_2.2 & 5^2{:}(4\times S_3)  & \\
3 & \LL_3(4).2 & \GU_3(2)  & L \ne \PSigmaL_3(4)\\
& \PGL_2(11) & 2_-^{1+2}.\mathrm{O}_2^-(2)  &  \\
& \PGL_2(7) & D_{12} & \\
4 & \PGL_2(q) & D_{2(q-1)} & q \in \{9,11\} \\
&  G_2(3).2 & \SL_2(3)^2 & \\
& S_7 & \mathrm{AGL}_1(7)  & \\
& \mathrm{M}_{10} & SD_{16}  & \\ \hline
\end{array}
\]
\caption{The groups $(L,J)$ in Theorem \ref{thm:semiprimitive_sol_stab}}
\label{tab:primitive}
\end{table}		

\begin{proof}
First note that $L$ is almost simple with soluble point stabiliser $J$, and $P$ is also soluble (see Remark \ref{r:sol}). By Lemma \ref{l:plus1} we have $b(L\wr P)\leqs b(L)+1$ and thus $b(G)\leqs 3$. Since $P$ is primitive,  Theorem \ref{t:prim} implies that either $D(P)=2$, or one of the following holds:
	\begin{itemize}\addtolength{\itemsep}{0.2\baselineskip}
		\item[{\rm (a)}] $D(P) = 3$ and $(k,P) = (4,A_4)$, $(3,S_3)$ or one of $8$  cases listed in \cite[Theorem 2]{S_dist} with $P$ soluble.
		\item[{\rm (b)}] $D(P) = 4$ and $(k,P) = (4,S_4)$.
	\end{itemize}

First assume $r(T)<D(P)$. Here \cite[Theorem 2.13]{BC} gives $b(T \wr P) \geqs 3$, so $b(G) = 3$ and the possibilities for $(L,J)$ can be read off from Table \ref{tab:r(L)}, noting that $L \ne T$ since we are assuming $G$ satisfies Hypothesis \ref{h:6}. In this way, we obtain the cases recorded in Table \ref{tab:primitive} with $r(T) \in \{2,3\}$.

For the remainder, we will assume $r(T) \geqs D(P)$. We now divide the proof into three cases, according to $D(P)$.

\vs

\noindent \emph{Case 1. $D(P)=2$.}

\vs
	
First assume $D(P)=2$. By combining Theorem \ref{thm:BC} and Corollary \ref{c:G diag}, we deduce that $b(G) = 3$ only if $r(L) = 1$ and $r(T) = |L:T|$. Therefore, by inspecting 
Table \ref{tab:r(L)} we see that $(L,J) = (\mathrm{M}_{10},5{:}4)$ or $(\mathrm{J}_2.2,5^2{:}(4\times S_3))$ are the only possibilities, and in both cases we have $|L:T| = 2$. If $b(G)=3$ then Corollary \ref{c:tau0,D(P)=2} implies that the condition \eqref{e:ddag} holds, which means that $P$ has no regular orbit on the set $X$ defined in \eqref{e:X}. 
In addition, since $P$ is soluble, Proposition \ref{p:ddagger} implies that $(k,P) = (2,S_2)$ and we conclude that $b(G) = 3$ via Lemma \ref{l:|L:T|=2}.
	
\vs

\noindent \emph{Case 2. $D(P)=3$.}

\vs

Next assume $D(P)=3$, so the possibilities for $(k,P)$ are described in case (a) above. Suppose $b(G) = 3$ and first observe that Proposition \ref{prop:diag} implies that \eqref{e:dag} holds with $m = 3$. By considering the cases in (a), with the aid of {\sc Magma} it is straightforward to check that \eqref{e:dag} holds with $m = 3$ if and only if $(k,P) = (3,S_3)$, $(4,A_4)$ or $(9,{\rm AGL}_2(3))$.

Next observe that $r(T) = |L:T|$ or $|L:T|+1$ by Corollary \ref{c:G diag}. Therefore,  $r(L) = 1$ and by inspecting Table \ref{tab:r(L)} we deduce that $(L,J)$ is one of the three cases recorded in Table \ref{tab:primitive} with $r(T) = |L:T|+1 = 3$. In particular, if $(k,P) = (3,S_3)$ then $b(G)=3$ in both cases by Lemma \ref{l:|L:T|=2}. We now consider the two remaining possibilities for $(k,P)$ in turn. 
	
Suppose $(k,P) = (9,{\rm AGL}_2(3))$. Using {\sc Magma}, we can find a distinguishing partition $\{\pi_1,\pi_2,\pi_3\}$ for $P$ such that $|\pi_i| = i+1$ for all $i$. Let $(L,J)$ be one of the relevant cases in Table \ref{tab:primitive} and fix $\alpha_0\in\Gamma = L/J$. Since $r(T) = 3$, there exist points $\gamma_1,\gamma_2,\gamma_3$ that are contained in distinct regular $T_{\alpha_0}$-orbits on $\Gamma$. In addition, we may assume that $\{\alpha_0,\gamma_1\}$ is not a base for $L$, whereas $\gamma_2$ and $\gamma_3$ are in the unique regular $L_{\a_0}$-orbit on $\Gamma$. Set $\alpha = (\alpha_0,\dots,\alpha_0)$ and $\beta = (\beta_1,\dots,\beta_k)$ in $\O = \Gamma^k$, where $\beta_j = \gamma_i$ if $j\in\pi_i$. We claim that $\{\alpha,\beta\}$ is a base for $G$, which is incompatible with our assumption that $b(G)=3$. To see this, suppose $x\in G_\alpha\cap G_\beta$ and write $x = z\sigma$, where $z = (z_1,\dots,z_k)\in L^k\cap G$ and $\sigma\in P$. Suppose $j \in \pi_1$, so $\beta_j = \gamma_1$. Since $\gamma_2$ and $\gamma_3$ are not in the $L_{\alpha_0}$-orbit of $\gamma_1$, it follows that $\beta_j^{z_j} = \gamma_1$. Therefore, 
$\sigma$ fixes $\pi_1$ and $\pi_2\cup \pi_3$ (setwise). As a consequence, since $|\pi_2| = 3$ and $|\pi_3| = 4$, we deduce that there exists $j\in\pi_3$ such that $j^\sigma\in \pi_3$. Therefore, $z_j \in L_{\alpha_0}\cap L_{\gamma_3}$ and thus $z_j = 1$. At this point, the condition $\tau(G) = 0$ forces $z \in T^k$ and thus Lemma \ref{l:BG_l:2.8} implies that $x=1$. Therefore, $\{\alpha,\beta\}$ is indeed a base for $G$ and so the case $(k,P) = (9,{\rm AGL}_2(3))$ is eliminated.
	
An almost identical argument also eliminates the case $(k,P) = (4,A_4)$, working with a distinguishing partition $\{\pi_1,\pi_2,\pi_3\}$ for $P$ with $|\pi_1| = |\pi_2| = 1$ and $|\pi_3| = 2$. We omit the details.
	
	\vs

\noindent \emph{Case 3. $D(P)=4$.}

\vs

Finally, let us assume $D(P)=4$ and $b(G)=3$, in which case $(k,P)=(4,S_4)$ (see case (b) above) and Corollary \ref{c:G diag} implies that $|L:T|\leqs r(T)\leqs |L:T|+2$. Therefore  $r(T)\leqs 2|L:T|$ and thus $r(L) \in \{1,2\}$. By inspecting Table \ref{tab:r(L)}, we deduce that either $(L,J)$ is one of the cases in Table \ref{tab:primitive} with $r(T) = |L:T|+2=4$, or $L = \Omega_8^+(2){:}3$ and $J$ is of type $\mathrm{O}_2^-(2)\times\GU_3(2)$. In the former case, Lemma \ref{l:|L:T|=2} shows that $b(G) = 3$, so it just remains to eliminate the latter possibility.
	
Suppose $L = T{:}\la a \ra = \Omega_8^+(2){:}3$ and $J$ is of type $\mathrm{O}_2^-(2)\times \GU_3(2)$, so $r(L)=1$ and $r(T)=5$. Since $|L:T| = 3$ is a prime and we are assuming that $\tau(G) = 0$ and $P = S_4$, it follows that 
\[
G = \langle T^4,(a,a,a,a),P\rangle,
\]
so an element $(z_1,z_2,z_3,z_4)\in L^4$ is contained in $G$ if and only if each $z_i$ is in the same coset of $T$ in $L$. Fix $\a_0,\gamma_1, \ldots, \gamma_4 \in \Gamma$, where the $\gamma_i$ are contained in distinct regular $T_{\alpha_0}$-orbits and $\{\alpha_0,\gamma_i\}$ is a base for $L$ if and only if $i \in \{1,2\}$. Notice that if $i \in \{3,4\}$ then $|L_{\a_0} \cap L_{\gamma_i}| = 3$, which implies that the $L_{\a_0}$-orbit and $T_{\a_0}$-orbit of $\gamma_i$ are equal (in particular, $\gamma_3$ and $\gamma_4$ are in distinct $L_{\a_0}$-orbits). Set $\a = (\alpha_0,\alpha_0,\alpha_0,\alpha_0)$ and $\b = (\gamma_1, \gamma_2, \gamma_3, \gamma_4)$ in $\O = \Gamma^4$. We claim that $\{\alpha,\beta\}$ is a base for $G$.
	
Assume $x\in G_\alpha\cap G_\beta$ and write $x = z\sigma$, where $z = (z_1,z_2,z_3,z_4)\in L^4 \cap G$ and $\sigma\in P$. Since none of the points $\gamma_1$, $\gamma_2$ and $\gamma_4$ are in the same $L_{\alpha_0}$-orbit as $\gamma_3$, we deduce that $3^\sigma = 3$. Similarly, $4^\sigma = 4$.
Suppose $\s = (1,2)$. Then $\gamma_1^{z_1} = \gamma_2$ and $\gamma_2^{z_2} = \gamma_1$, which implies that $z_1, z_2 \in L \setminus T$. Moreover, $z_1z_2 \in L_{\a_0} \cap L_{\gamma_1} = 1$, so $z_1 = z_2^{-1}$ and we deduce that $z_1$ and $z_2$ are contained in different cosets of $T$ in $L$. But this means that $z \notin G$ and we have reached a contradiction. This forces $\s=1$. Finally, since $\{\alpha_0,\gamma_1\}$ is a base for $L$ we deduce that $z_1 = 1$ and thus  $z \in T^4$. Since each $\{\a_0,\gamma_i\}$ is a base for $T$, we conclude that $z=1$ and the proof of both the claim and the theorem is complete.
\end{proof}

\subsection{Final remarks}

We conclude by briefly discussing the general problem of determining the base-two product type groups with soluble point stabilisers. Let $G \leqs L \wr P$ be such a group with socle $T^k$, and adopt all the usual notation as before. The case where $G = L \wr P$ is handled in Theorem \ref{thm:wr_base-two}, so we may assume $G < L \wr P$ and
$b(L \wr P) \geqs 3$. Continuing with the main theme of Section \ref{s:gen}, let us also assume that Hypothesis \ref{h:6} holds. Here $L \ne T$ and $b(G) = 2$ only if $b(T \wr P)=2$, so $r(T) \geqs D(P)$ and we deduce that $b(L) \in \{2,3\}$ as a consequence of 
\cite[Theorem 2]{B_sol}. In this setting, we have handled the cases 
\begin{itemize}\addtolength{\itemsep}{0.2\baselineskip}
\item[{\rm (a)}] $b(L)=k=2$ (see Theorem \ref{thm:P=C_2_sol_gen}); and
\item[{\rm (b)}] $b(L)=2$, $P$ is primitive and $\tau(G)=0$ (see Theorem \ref{thm:semiprimitive_sol_stab}).
\end{itemize}

So even under the assumption $b(L)=2$, there is more work to be done here and it would be interesting to see if it is possible to relax the conditions on $P$ and $\tau(G)$ in case (b). For example, it might be fruitful to consider the groups with $\tau(G) = k-1$ as a starting point.
 
As the following example demonstrates, we can also find base-two groups under Hypothesis \ref{h:6} when $b(L)=3$.

\begin{ex}\label{ex:bL3}
Take $L = {\rm P\Gamma L}_2(q) = T{:}\la a,b\ra = T{:}2^2$ and let $J$ be a maximal subgroup of type ${\rm GL}_1(q) \wr S_2$, where $q=p^2$, $p \geqs 3$ is a prime and ${\rm PGL}_2(q) = T{:}\la a \ra$ and ${\rm P\Sigma L}_2(q) = T{:}\la b \ra$. By \cite[Lemma 4.7]{B_sol} we have $b(L)=3$ and $b({\rm PGL}_2(q)) = b({\rm P\Sigma L}_2(q)) = 2$. Set $(k,P) = (2,S_2)$ and consider
\[
G = \la T^2,(a,a),(b,b),P \ra
\]
as a primitive product type group on $\O = \Gamma^2$, where $\Gamma = L/J$. 

We may identify $\Gamma$ with the set of distinct pairs of $1$-dimensional subspaces of the natural module for $T$. Given this identification, a precise description of the $2$-element bases for ${\rm PGL}_2(q)$ and ${\rm P\Sigma L}_2(q)$ is presented in \cite[Section 4.1]{BH_Saxl} and this allows us to choose 
bases $\{\a_0,\gamma_1\}$ and $\{\a_0,\gamma_2\}$ for ${\rm PGL}_2(q)$ and ${\rm P\Sigma L}_2(q)$, respectively, where $\gamma_1$ and $\gamma_2$ are contained in distinct $L_{\a_0}$-orbits. In addition, notice that $\{\a_0,\gamma_2\}$ is a base for $T{:}\la ab\ra$ by \cite[Lemma 4.5]{BH_Saxl}. Set $\a = (\a_0,\a_0)$ and $\b = (\gamma_1,\gamma_2)$. We claim that $\{\a,\b\}$ is a base for $G$ and thus $b(G)=2$. To see this, suppose $x = (z_1,z_2)\s \in G$ fixes $\a$ and $\b$. Then each $z_i$ is contained in $L_{\a_0}$ and thus $\s=1$ since $\gamma_1$ and $\gamma_2$ are in distinct $L_{\a_0}$-orbits. Since $x \in G$, we may write $z_i  = t_ic$ with $t_i \in T$ and $c \in \{1,a,b,ab\}$. If $c=1$ then $z_i \in T_{\a_0} \cap T_{\gamma_i}=1$ and thus $x=1$. If $c=a$ then $z_1 \in {\rm PGL}_2(q)_{\a_0} \cap {\rm PGL}_2(q)_{\gamma_1}=1$, which is a contradiction since $z_1 \in L \setminus T$. An entirely similar argument applies if $c \in \{b,ab\}$ and the proof of the claim is complete.
\end{ex}

Notice that $|L:T|=4$ in Example \ref{ex:bL3}. By the following result, there are no examples with $b(L)=3$, $|L:T|=2$ and $b(G)=2$. Here there is no need to assume that $G$ has soluble point stabilisers and it is worth noting that the same proof goes through  under the weaker hypothesis $T^k < L^k \cap G$.

\begin{prop}
	\label{prop:b(L)>2,|L:T|=2}
Assume Hypothesis \ref{h:6}, with $b(L) \geqs 3$ and $|L:T| = 2$. Then $b(G) \geqs 3$.
\end{prop}

\begin{proof}
We may as well assume $b(T)=2$. Fix $\a_0 \in \Gamma$ and set $\alpha = (\alpha_0,\dots,\alpha_0) \in \O$. It suffices to show that $\{\a,\b\}$ is not a base for $G$, where $\beta = (\gamma_1,\dots,\gamma_k)$ and each $\gamma_i$ is contained in a regular $T_{\alpha_0}$-orbit. Since  $b(L) \geqs 3$, for each $i$ there exists $x_i\in L\setminus T$ fixing both $\alpha_0$ and $\gamma_i$.
	
Fix $z = (z_1, \ldots, z_k) \in (L^k \cap G)\setminus T^k$ and define a subset $A$ of $[k]$ such that $i\in A$ if and only if $z_i \notin T$. Note that $A$ is non-empty since $z\notin T^k$. Set $y = (y_1,\dots,y_k)\in L^k$, where $y_i = x_i$ if $x\in A$, otherwise $y_i = 1$, and observe that $y$ is non-trivial and it fixes $\a$ and $\b$.
Finally, since $|L:T| = 2$ we deduce that $yz\in T^k \leqs G$, so $y\in G_{\a} \cap G_{\b}$ and the result follows.
\end{proof}

By the proposition, if $G$ has soluble point stabilisers and $b(L)=3$, then $b(G)=2$ only if $b(T) = 2$, $r(T) \geqs D(P)$ and $|L:T| \geqs 3$. Since $L$ is almost simple, the possibilities for $(L,J)$ can be read off from \cite[Tables 4-7]{B_sol}, noting that  $b(T)=2$ only if $\log_m |T|< 2$, where $m = |\Gamma|$. In \cite[Table 4]{B_sol}, we deduce that the only possibility is $L = A_6.2^2$ with $J = [32]$ or $D_{20}.2$ (here $L$ is isomorphic to ${\rm P\Gamma L}_2(9)$ and $J$ is of type ${\rm GL}_1(9) \wr S_2$ or ${\rm GL}_1(9^2)$, respectively). One can check that no examples arise in \cite[Tables 5 and 6]{B_sol}, while the relevant possibilities for $(L,J)$ in \cite[Table 7]{B_sol} are recorded in Table \ref{tab:poss} (here we implicitly assume the additional condition $|L:T| \geqs 3$). For the values of $r(T)$ in the first two rows, we have $(a,b)= (7,1)$ if $q \equiv 1 \imod{4}$, otherwise $(a,b) = (5,3)$ (see Remark \ref{r:LJ} and the proof of \cite[Lemma 7.9]{BH_gen}).

It is straightforward to check that none of these possibilities correspond to cases in Table \ref{tab:reg(L)}, which implies that ${\rm reg}(L,3) \geqs 5$ and we obtain the following result via Theorem \ref{thm:BC}. 

\begin{table}
\[
\begin{array}{llcl} \hline
T & \mbox{Type of $J$}  & r(T) & \mbox{Conditions} \\ \hline
{\rm L}_2(q) & {\rm GL}_1(q) \wr S_2 & (q+a)/4 & \mbox{$q$ odd, ${\rm PGL}_2(q) < L$} \\
& {\rm GL}_1(q^2) & (q-b)/4 & \mbox{$q$ odd, ${\rm PGL}_2(q) \leqs L$} \\
{\rm L}_3(4) & {\rm GU}_3(2) & 3 &  \\
{\rm L}_4(3) & {\rm O}_4^{+}(3) & 6 & L = T.2^2 \\
{\rm U}_4(3) & {\rm GU}_1(3) \wr S_4 & 11 & L = T.D_8 \\
\Omega_8^{+}(2) & {\rm O}_{2}^{-}(2) \times {\rm GU}_3(2) & 5 & L = T.S_3 \\
{\rm P\Omega}_8^{+}(3) & {\rm O}_{4}^{+}(3) \wr S_2 & 12 & |L:T| \geqs 6 \\ \hline
\end{array}
\]
\caption{The groups $(L,J)$ in Proposition \ref{p:last}(ii)}
\label{tab:poss}
\end{table}		

\begin{prop}\label{p:last}
Assume Hypothesis \ref{h:6} and $H$ is soluble. Then $b(G)=2$ only if one of the following holds:
\begin{itemize}\addtolength{\itemsep}{0.2\baselineskip}
\item[{\rm (i)}] $b(L)=2$ and either $b(L \wr P) = 2$, or $b(L \wr P) = 3$, $r(L) <D(P)$ and $(L,J)$ is one of the cases in Table \ref{tab:r(L)}.
\item[{\rm (ii)}] $b(L)=b(L \wr P) = 3$ and $(L,J)$ is one of the cases in Table \ref{tab:poss}, where $r(T) \geqs D(P)$ and $|L:T| \geqs 3$.
\end{itemize}
\end{prop} 

The study of bases for product type groups becomes rather more complicated if we drop Hypothesis \ref{h:6} and this will be the focus of a future paper. As an illustration, we close with the following example.

\begin{ex}\label{ex:pgl}
Let $G \leqs L \wr P$ be a product type primitive group on $\O = \Gamma^3$, where $P = S_3$ and $L \leqs {\rm Sym}(\Gamma)$ is a primitive group with point stabiliser $J$. As before, we may assume $G$ induces $L$ on each factor of $\O$, and $P$ on the set of factors of the socle $T^3$. Take $L = {\rm PGL}_2(11) = T{:}\la a \ra = T.2$ and $J = S_4$, so $b(L) = 2$, $r(L)=1$ and $r(T) = 3$ (see Table \ref{tab:r(L)}, noting that $J$ is of type $2_{-}^{1+2}.{\rm O}_{2}^{-}(2)$). 

First assume $G$ contains $P$ and note that there are three possibilities, namely $L \wr P$, $\la T^3,(a,a,a),P\ra$ and $\la T^3,(a,a,1),P\ra$. For the full wreath product $G = L \wr P$, Theorem \ref{thm:wr_base} implies that $b(G) = 3$ since $r(L) < D(P)$. Similarly, if $G = \la T^3,(a,a,a),P\ra$ then $\tau(G)=0$ (see \eqref{e:tau}) and thus $b(G) = 3$ by Theorem \ref{thm:semiprimitive_sol_stab}, while a {\sc Magma} calculation gives $b(G) = 3$ if $G = \la T^3,(a,a,1),P\ra$.

Now assume $G$ does not contain $P$, so $|P \cap G| \in \{1,2,3\}$. With the aid of {\sc Magma}, we find that there are $8$ product type primitive groups of this form, up to permutation isomorphism. More precisely, there are $3$ groups with $P \cap G = 1$ and in each case $b(G)=2$. There are also $3$ groups with $|P \cap G|=2$ and here we find that $b(G)=3$. Finally, let $G_1$ and $G_2$ be the two remaining groups with $|P \cap G_i|=3$. For one of them, say $G_1$, we have $\la G_1, P\ra = L \wr P$ and $b(G_1) = 3$. On the other hand, $\la G_2, P\ra$ has index $4$ in $L \wr P$ and one can check that $b(G_2)=2$.
\end{ex}

\section{The tables}\label{s:tab}

Here we present Tables \ref{tab:r(L)} and \ref{tab:reg(L)}, which arise in the statement of Theorems \ref{thm:reg(L,b(L))}, \ref{thm:wr_base-two} and \ref{thm:wr_base}.

Let $L \leqs {\rm Sym}(\Gamma)$ be an almost simple primitive group with socle $T$ and soluble point stabiliser $J$ (in terms of the notation introduced in Section \ref{s:intro}, this means that $L \in \mathcal{S}$). Let $b(L)$ be the base size of $L$ and let ${\rm reg}(L,b(L))$ be the number of regular orbits of $L$ on the Cartesian product 
$\Gamma^{b(L)}$.  By Theorem \ref{thm:reg(L,b(L))}, we have 
${\rm reg}(L,b(L)) \leqs 4$ if and only if $(L,J)$ is one of the cases recorded in Tables  
\ref{tab:r(L)} and \ref{tab:reg(L)}. Note that if $b(L)=2$ then ${\rm reg}(L,b(L)) = r(L)$ is the number of regular suborbits of $L$ on $\Gamma$ (and similarly, $r(T)$ is the number of regular suborbits of $T$).

\begin{rem}\label{r:cases}
Suppose $T$ is a classical group over $\mathbb{F}_q$. Then due to the existence of isomorphisms between various low-dimensional classical groups (see \cite[Proposition 2.9.1]{KL}, for example), we may assume $T$ is one of the following:
\[
{\rm L}_n(q), \, n \geqs 2; \; {\rm U}_n(q), \, n \geqs 3; \; {\rm PSp}_n(q), \, n \geqs 4; \; {\rm P\O}_n^{\e}(q), \, n \geqs 7.
\]
Furthermore, in order to avoid unnecessary repetition, we will 
exclude the groups with socle 
\[
{\rm L}_{2}(4), \, {\rm L}_{2}(5), \, {\rm L}_{2}(9), \, {\rm L}_{3}(2), \, {\rm L}_{4}(2), \, {\rm PSp}_{4}(2)', \, {\rm PSp}_{4}(3), \, {}^2G_2(3)', \, G_2(2)'
\]
in Tables \ref{tab:r(L)} and \ref{tab:reg(L)}. This is consistent with \cite{B_sol} (see \cite[Remark 9.1]{B_sol}) and it is justified by the existence of the following well known isomorphisms:
\[
{\rm L}_{2}(4) \cong {\rm L}_{2}(5) \cong A_5,\, {\rm L}_{2}(9) \cong {\rm PSp}_{4}(2)' \cong A_6, \, {\rm L}_{3}(2) \cong {\rm L}_{2}(7),
\] 
\[
{\rm L}_{4}(2) \cong A_8, \, {\rm PSp}_{4}(3) \cong {\rm U}_{4}(2), \, {}^2G_2(3)' \cong {\rm L}_{2}(8), \, G_2(2)' \cong {\rm U}_{3}(3).
\]
So for example, a reader who is interested in the groups with socle ${\rm L}_{4}(2)$ should inspect the tables for any cases with $L=S_8$ or $A_8$.
\end{rem}

In both tables we refer to the \emph{type} of $J$. If $L$ is a classical group, then this  provides an approximate description of the structure of $J$, which is consistent with its usage in \cite{KL}, where the precise structure of $J$ is determined. For example, in the first row of Table \ref{tab:r(L)} we have $L = {\rm PGL}_2(q)$ and $J$ is of type ${\rm GL}_1(q) \wr S_2$, which indicates that $J$ is the stabiliser in $L$ of a direct sum decomposition $V = V_1 \oplus V_2$ of the natural module for $L$, where the $V_i$ are $1$-dimensional subspaces (see \cite[Proposition 4.2.9]{KL} for the precise structure). In Table \ref{tab:reg(L)} we use the standard $P_m$ notation for maximal parabolic subgroups, where $P_m$ is the stabiliser of a totally singular $m$-dimensional subspace of the natural module. For $T = {\rm L}_n(q)$, we also use $P_{1,n-1}$ for the stabiliser of a flag of subspaces $0<V_1<V_{n-1}<V$ of $V$, where $\dim V_i = i$ (again, this notation is consisted with \cite{KL}). For the remaining groups where $L$ is not classical, we record the precise structure of $J$ in the third column of both tables. In addition, in Table \ref{tab:r(L)} we use the standard notation $\delta, \phi, \gamma$, etc. for automorphisms of a simple classical group (see \cite[Section 1.7.2]{Low-Dimensional}). 

\begin{rem}\label{r:class}
Let us record some additional comments on Tables \ref{tab:r(L)} and \ref{tab:reg(L)}.  
\begin{itemize}\addtolength{\itemsep}{0.2\baselineskip}
\item[{\rm (i)}] In the first row of Table \ref{tab:r(L)} we have $L = {\rm PGL}_2(q)$ and $J = D_{2(q-1)}$ is a subgroup of type ${\rm GL}_1(q) \wr S_2$. In view of Remark \ref{r:cases}, we may assume $q \geqs 7$ and $q \ne 9$. More precisely, $J$ is non-maximal when $q=5$, while the cases $q=4$ and $9$ are recorded as $(L,J) = (A_5,D_6)$ and $(A_6.2,D_{16})$ in Table \ref{tab:r(L)}. As noted in Remark \ref{r:LJ}
we have
\[
r(T) = \left\{\begin{array}{ll}
1 & \mbox{$q$ even} \\
(q+a)/4 & \mbox{$q$ odd} 
\end{array}\right.
\]
where $a=7$ if $q \equiv 1 \imod{4}$, otherwise $a=5$. 
\item[{\rm (ii)}] Consider the first row in Table \ref{tab:reg(L)} with $b(L)=3$. Here $L = {\rm L}_2(q).2$, $J = P_1$ is a Borel subgroup and $q = p^f$ is odd, so in view of Remark \ref{r:cases}, we may assume $q \ne 5,9$. Then ${\rm reg}(L,b(L))=1$ if and only if $L$ is sharply $3$-transitive, which means that either $L = {\rm PGL}_2(q)$, or 
$f$ is even and $L = T.\la \delta\phi^{f/2}\ra$.
\item[{\rm (iii)}] Up to isomorphism, there are three almost simple groups of the form ${\rm L}_4(3).2$, one of which is ${\rm PGL}_4(3)$. In addition, we have ${\rm L}_4(3).2_2$ and ${\rm L}_4(3).2_3$, which contain involutory graph automorphisms $x$ with $C_{T}(x) = {\rm PGSp}_4(3)$ and ${\rm PSO}_{4}^{-}(3).2$, respectively.
\item[{\rm (iv)}] In the fourth column of Table \ref{tab:reg(L)} we record ${\rm reg}(L,b(L))$. In a few cases, this is presented as $r_1,r_2, \ldots$, which means that ${\rm reg}(L,b(L)) = r_i$ when $q$ is contained in the $i$-th set appearing in the fifth column. For example, if $L = \PGammaL_2(q)$ and $J = P_1$, then ${\rm reg}(L,b(L)) = 3$ if $q=16$ and ${\rm reg}(L,b(L)) = 2$ if $q=8$. 
\end{itemize}
\end{rem}

{\small
	\begin{table}
		\[
		\begin{array}{cllcl} \hline
		r(L) & L & \mbox{Type of $J$} & r(T) & \mbox{Comments} \\ \hline
		1&\PGL_2(q) & \GL_1(q)\wr S_2 & \mbox{See Remark \ref{r:class}(i)} & \mbox{$q\geqs 7$, $q \ne 9$} \\
		&{\rm P\Omega}_8^+(3).2^2 & {\rm O}_4^+(3)\wr S_2 & 12 & \mbox{Both groups of this shape} \\
		&\Omega_8^+(2).3 & {\rm O}_2^-(2) \times {\rm GU}_3(2) & 5 & \\
		&{\rm SO}_7(3) & {\rm O}_4^+(3) \perp {\rm O}_3(3) & 5 &\\
		&{\rm PSp}_6(3) & {\rm Sp}_2(3)\wr S_3 & 1 &\\ 
		&{\rm PGL}_4(3) & {\rm O}_4^+(3) & 6 &\\
		&{\rm U}_4(3).[4] & {\rm GU}_1(3)\wr S_4 & 11 & L \ne T.\la \delta^2,\phi\ra \\
		&{\rm U}_4(3) & {\rm GU}_2(3)\wr S_2 & 1 &\\
		&{\rm L}_3(4).D_{12} & {\rm GL}_1(4^3) & 44 & \\
		&{\rm L}_{3}(4).2 & \GU_3(2) & 3 & L \ne {\rm P\Sigma L}_{3}(4) \\
		&{\rm U}_3(5).S_3 & {\rm GU}_1(5)\wr S_3 & 11 &\\
		&{\rm U}_3(4) & {\rm GU}_1(4)\wr S_3 & 1 &\\
		&{\rm PGL}_2(11) & 2^{1+2}_{-}.{\rm O}_{2}^{-}(2) &  3 &\\
		&G_2(3).2 & {\rm SL}_2(3)^2 & 4 &\\
		&S_7 & {\rm AGL}_1(7) & 4 & \\
		& A_6.2 & D_{16} & 4 & L = {\rm PGL}_2(9) \\ 
		&A_6.2 & 5{:}4 & 2 & L = {\rm M}_{10} \\ 
		& A_5 & D_6 & 1 & \\
		&{\rm J}_2.2 & 5^2{:}(4\times S_3) & 2 &\\
		&{\rm M}_{11} & 2.S_4 & 1 &\\ 
		2 &\POmega_8^+(3).4 & \mathrm{O}_4^+(3)\wr S_2 & 12 &\\
&\LL_4(3).2_2 & \mathrm{O}_4^+(3) & 6 & \\
		&\Sp_4(4).4 & \mathrm{O}_2^-(4)\wr S_2 & 9 & \\
		&\LL_3(3) & \GL_1(3^3) & 2 & \\
		&\LL_3(3).2 & \mathrm{O}_3(3) & 5 & \\
		&\UU_3(5).S_3 & 3^{1+2}.\Sp_2(3) & 21 & \\
		&\LL_2(27).3 & \GL_1(27)\wr S_2 & 8 &  \\
		& & \GL_1(27^2) & 6 &  \\
		& {\rm PGL}_2(13) & 2_-^{1+2}.\mathrm{O}_2^-(2) & 6 & \\
		&\LL_2(11) & \GL_1(11^2) & 2 & \\
		& {}^2B_2(8).3 & 13{:}12 & 7 &\\
		&A_9 & \mathrm{ASL}_2(3) & 2 &\\
		& A_6.2 & SD_{16} & 4 & L = {\rm M}_{10} \\ 
		&\mathrm{J}_2 & 5^2{:}D_{12} & 2 &\\	
		3 &\POmega_8^+(3).3 & \mathrm{O}_4^+(3)\wr S_2 & 12  & \\
		&\POmega_8^+(3).2 & \mathrm{O}_4^+(3)\wr S_2 & 12   & L = T.\langle \gamma\rangle\\
		&\LL_4(3).2_3 & \mathrm{O}_4^+(3) & 6 &  \\
		&{\rm L}_3(4).S_3& {\rm GL}_1(4^3) & 44 & L \ne {\rm P\Gamma L}_3(4) \\
		&{\rm PGU}_3(5) & {\rm GU}_1(5)\wr S_3 & 11 & \\
		& \LL_2(25).2 & \GL_1(25)\wr S_2 & 8 & L = \PSigmaL_2(25) \\
		& & \GL_1(25^2) & 6  & L \ne\PGL_2(25)\\
		&\LL_2(17) & 2_-^{1+2}.\mathrm{O}_2^-(2) & 3 & \\
		&\LL_2(13) & \GL_1(13^2) & 3 & \\
		&{^2}F_4(2) & 5^2{:}4S_4 & 6 & \\
		&\mathrm{J}_3.2 & 3^{2+1+2}{:}8.2 & 10 & \\
		&\mathrm{HS}.2 & 5^{1+2}.[2^5] & 9 & \\
		4 & {\rm PSO}_8^{+}(3) & \mathrm{O}_4^+(3)\wr S_2 & 12  & \\
		&{\rm U}_4(3).2 & {\rm GU}_1(3)\wr S_4 & 11 &   L \ne T.\la \delta^2\phi\ra \\
		&\LL_2(25).2 & \GL_1(25)\wr S_2 & 8 &  L = T.\langle \delta\phi\rangle\\
		&\LL_2(q) & \GL_1(q^2) & 4 & q =17, 19\\
		&G_2(3) & \SL_2(3)^2 & 4 &\\
		&\mathrm{M}_{12}.2 & S_4\times S_3 & 13 &\\
		&\mathrm{Suz}.2 & 3^{2+4}{:}2.(S_4\times D_8) & 16 &\\
		\hline
		\end{array}
		\]
		\caption{The groups $L\in\mathcal{S}$ with $b(L)=2$ and $r(L)\leqs 4$}
		\label{tab:r(L)}
\end{table}}

{\small
	\begin{table}
		\[
		\begin{array}{cllcl} \hline
		b(L) & L & \mbox{Type of $J$} & \reg(L,b(L))  & \mbox{Comments} \\ \hline
		4 &\LL_3(3) & P_1, P_2 & 1 &\\
		&\PGammaL_2(q) & P_1 & 3,2 & q \in\{16\},  \{8\}\\
		&A_6.2^2 & 3^2{:}SD_{16} & 3 &\\
		&S_5 & S_4 & 1&\\
		3 &\LL_2(q).2 & P_1 & 1 & \mbox{See Remark \ref{r:class}(ii)} \\
		&\LL_2(q) & P_1 & 2 - \delta_{2,p} & \mbox{$q \geqs 7$, $q \ne 9$} \\	
		& {\rm P\Omega}_8^{+}(3) & P_2 & 3 & \\
		&\Omega_7(3) & P_2 & 3 &\\
		&\PSp_6(3) & P_2 & 3 &\\
		&\LL_4(3).2 & P_{1,3} & 3 & L\ne \PGL_4(3)\\
		&\UU_4(3).2 & P_1 & 1 & L\not < \PGU_4(3)\\
		&\UU_4(3) & P_1 & 3 &\\
		&\Aut(\LL_3(q)) & P_{1,2} & 4,3,1 & q \in \{3,25,27,64\}, \{32\}, \{8,9,16\} \\
		&\LL_3(16).(2\times 4) & P_{1,2} & 4 &\\
		&\LL_3(16).D_{12} & P_{1,2} & 2 & \\
		&\LL_3(16).{12} & P_{1,2} & 3 & \\
		&\LL_3(4).2^2 & P_{1,2} & 4 &\\
		&\LL_3(4).6 & P_{1,2} & 4&\\
		&\Aut(\UU_3(q)) & P_1 & 4,3,2,1 & q \in \{27\}, \{7,32\}, \{5,9,16\}, \{3,4,8\} \\
		& \UU_3(16).4 & P_1 & 4 & \\
		&\UU_3(9).2 & P_1 & 4 & \\
		&\UU_3(8).S_3 & P_1 & 4 & \\
		&\UU_3(8).6 & P_1 & 3 & \\
		&\UU_3(8).3^2 & P_1 & 2 & \\
		&\UU_3(4).2 & P_1 & 2 & \\
		&\UU_3(4) & P_1 & 4 & \\
		&\UU_3(3) & P_1 & 3 & \\
		&\LL_2(7) & 2_-^{1+2}.\mathrm{O}_2^-(2) & 1 &\\
		& {}^2B_2(8).3 & [8^2]{:}7.3 & 2 & \\ 
		&G_2(3) & [3^5]{:}\GL_2(3) & 4 &\\
		&S_7 & S_4\times S_3 & 1 &\\
		&A_6.2 & 3^2{:}Q_8 & 1& L = \PGL_2(9) \\
		& & \mathrm{AGL}_1(9) & 1& L = {\rm M}_{10} \\
		&A_6 & (S_3\wr S_2)\cap L & 2&\\
		&S_5 & S_3\times S_2 & 4 &\\
		& & 5{:}4 & 1 &\\
		&A_5 & A_4 & 1 & \\
		& & D_{10} & 2 & \\ \hline
		\end{array}
		\]
		\caption{The groups $L\in\mathcal{S}$ with $b(L) \in \{3,4\}$ and $\reg(L,b(L))\leqs 4$}
		\label{tab:reg(L)}
\end{table}}

\end{document}